\documentclass[11pt]{article}
\usepackage[english]{babel}
\usepackage{titlesec}
\usepackage{geometry}
\usepackage{enumitem}
\usepackage[utf8]{inputenc}
\usepackage{amsmath}
\usepackage{amssymb}
\usepackage{amsthm}
\usepackage[mathcal]{euscript}
\usepackage{color}
\usepackage{tikz-cd}
\usepackage{mathabx}
\usepackage{physics}
\usepackage{MnSymbol}
\usepackage{csquotes}
\usepackage{url}
\usepackage{caption}
\usepackage{hyperref}
\usepackage{mathtools}
\mathtoolsset{showonlyrefs=true}
\allowdisplaybreaks

\usepackage{xparse}
\let\realItem\item %
\makeatletter
\NewDocumentCommand\myItem{ o }{%
   \IfNoValueTF{#1}%
      {\realItem}
      {\realItem[#1]\def\@currentlabel{#1}}
}
\makeatother

\usepackage{enumitem}    
\setlist[enumerate]{
    before=\let\item\myItem, %
    label=\textnormal{(\arabic*)}, %
}

\numberwithin{equation}{section}

\title{Long-Time Large Deviation Asymptotics for Mean-Field Granular Media Equations}

\author{Amarjit Budhiraja\textsuperscript{*}
\quad and \quad
Jack Sullivan\textsuperscript{**}\\[0.5em]
\small Department of Statistics and Operations Research,\\
\small University of North Carolina at Chapel Hill, Chapel Hill, NC 27599, USA}
\date{}

\newtheorem{theorem}{Theorem}[section]
\newtheorem{lemma}[theorem]{Lemma}
\newtheorem{corollary}[theorem]{Corollary}
\newtheorem{proposition}[theorem]{Proposition}
\newtheorem{remark}[theorem]{Remark}
\newtheorem{definition}[theorem]{Definition}
\newtheorem{assumption}{Assumption}[section]

\newcommand{\R}{\mathbb R}
\newcommand{\NN}{\mathbb N}
\newcommand{\N}{\mathbb N}
\newcommand{\E}{\mathbb E}
\newcommand{\PP}{\mathbb P}
\newcommand{\cP}{\mathcal P}
\newcommand{\cF}{\mathcal F}
\newcommand{\cE}{\mathcal E}
\newcommand{\cR}{\mathcal R}
\newcommand{\cH}{\mathcal H}
\newcommand{\cX}{\mathcal X}
\newcommand{\cI}{\mathcal I}
\newcommand{\cV}{\mathcal V}
\newcommand{\cD}{\mathcal D}
\newcommand{\cW}{\mathcal W}
\newcommand{\cln}{\mathcal N}

\newcommand{\cS}{\mathcal S}

\newcommand{\cJ}{\mathcal J}
\newcommand{\cL}{\mathcal L}
\newcommand{\cA}{\mathcal A}

\newcommand{\clp}{\mathcal P}
\newcommand{\cls}{\mathcal S}

\newcommand{\clx}{\mathcal X}

\newcommand{\Law}{\operatorname{Law}}
\newcommand{\Ent}{\operatorname{Ent}}

\begin{document}
\maketitle
\begingroup
\renewcommand{\thefootnote}{*}
\footnotetext{Email: \texttt{budhiraj@email.unc.edu}}
\renewcommand{\thefootnote}{**}
\footnotetext{Email: \texttt{jsully@unc.edu}}
\endgroup
\vspace{-1.25em}

\begin{abstract}
\small
We study long-time mean-field large deviations for granular-media diffusions. Under a  dissipativity assumption, our main result shows that for
arbitrary $T_N\to\infty$, the occupation measure of the $N$-particle empirical
measure satisfies a large deviation principle (LDP) on $\cP(\cP(\R^d))$ with speed $NT_N$ and rate function
$\cI(\Gamma)=\int_{\cP(\R^d)}\mathfrak i(\mu)\,\Gamma(d\mu)$
where $\mathfrak i(\mu)$ is one quarter of the free-energy dissipation, where the latter is defined as the squared $L^2$-norm of the Wasserstein gradient of the free energy $\cF$ associated with the model, namely
$$\mathfrak i(\mu) = \frac14 \left\|
    \nabla\frac{\delta\cF}{\delta\mu}(\mu)
  \right\|_{L^2(\mu)}^2.$$
In the noninteracting case, the local cost $\mathfrak i$ reduces to the classical Donsker--Varadhan rate function, while $\cI$ is its $\Gamma$-average. For interacting systems, $\mathfrak i(\mu)$ can be interpreted as the per-unit-time Dawson--G\"artner action of the constant path at $\mu$, so that $\cI(\Gamma)$ averages these holding costs according to $\Gamma$.
The main difficulty in the upper bound is that empirical measures are atomic, so the free energy is infinite on them. A key ingredient in overcoming this difficulty is an analysis of the evolution of a suitably regularized free-energy functional evaluated along the empirical-measure flow of controlled granular-media diffusions. The lower bound uses approximation by finitely supported occupation measures whose atoms have smooth densities with Gaussian tails.
As an application of the main result, we discuss the case of quadratic potentials and show that
the variational problem associated with an atypically large overall
space-time variance exhibits a transition: below a critical
value the minimizer is $\Gamma = \delta_{\nu}$, where $\nu$ is Gaussian, whereas above it the
minimizers are non-Dirac occupation measures $\Gamma$ supported on Gaussian
profiles with different centers.
\end{abstract}

\smallskip
{\small
\noindent\textbf{Keywords:} large deviations; occupation measure of empirical measures;
weakly interacting diffusions; large-time, mean-field asymptotics; McKean--Vlasov
equations; granular-media equations; Wasserstein gradient flows; free-energy
dissipation; relative Fisher information; Dirichlet energy.\par}

\smallskip
{\small
\noindent\textbf{Mathematics Subject Classification (2020):}
60F10 (primary); 60K35, 35Q84, 49Q22 (secondary).\par}

\section{Introduction}
\label{sec:intro}

Mean-field systems of weakly interacting diffusions arise naturally in
statistical mechanics, probability, and the study of nonlinear
Fokker--Planck equations. In this work we consider a class of granular
media models in which the particles evolve according to
\begin{equation}\label{eq:1}
  dX_i^N(t)
  =
  -\nabla V(X_i^N(t))\,dt
  -\frac1N\sum_{j=1}^N
    \nabla W(X_i^N(t)-X_j^N(t))\,dt
  +\sqrt{2}\,dB_i(t),
\end{equation}
$i=1,\ldots, N$, where $B_1,\dots,B_N$ are independent standard $d$-dimensional Brownian motions.
Here $V:\R^d \to \R$ is a confining potential and $W:\R^d \to \R$ describes the mean-field
interaction. Writing
\begin{equation} \label{eq:2}
  \mu_t^N=\frac1N\sum_{i=1}^N\delta_{X_i^N(t)}, \; t \ge 0
\end{equation}
for the empirical measure of the particles, the classical mean-field
law of large numbers says that, under conditions, over every fixed time interval,
$\mu^N$ converges in probability, as $N\to\infty$, to the solution of the
McKean--Vlasov equation
\begin{equation} \label{eq:3}
  \partial_t\mu_t
  =
  \Delta\mu_t+
  \nabla\cdot\bigl(\mu_t\nabla(V+W*\mu_t)\bigr).
\end{equation}
For the class of potentials considered here this equation has the
Wasserstein gradient-flow structure associated with the free energy
\[
  \cF(\mu)
  =
  \int V\,d\mu
  +\frac12\iint W(x-y)\,\mu(dx)\mu(dy)
  +\Ent(\mu),
\]
where $\Ent(\mu)$ is the entropy of measure $\mu$; see Section \ref{sec:1} for precise definitions.
Under our dissipativity assumption (see Assumption \ref{ass:3}) the solution to 
\eqref{eq:3} has a unique equilibrium
$\mu_*$ and the solution $\mu_t$ converges exponentially fast to $\mu_*$ as $t\to \infty$. 
Uniform-in-time propagation of chaos and convergence to
equilibrium for granular-media systems, under closely related assumptions, have been studied in
\cite{Malrieu2003,CattiauxGuillinMalrieu2008}.

The question studied in this paper concerns a large deviation asymptotic
regime, in which the number of particles and the observation time both
tend to infinity. For $T>0$ we consider the occupation measure
\begin{equation}\label{eq:4}
  \Gamma^{N,T}
  =
  \frac1T\int_0^T\delta_{\mu_t^N}\,dt,
\end{equation}
which is a random element of
$\cP(\cP(\R^d))$, where for a polish space $S$, $\cP(S)$ is the space of probability measures on $S$ equipped with the weak convergence topology. Given an arbitrary sequence $T_N\to\infty$, our
main goal is to determine the large deviation behavior of
$\Gamma^{N,T_N}$ as $N\to\infty$. Thus the object of interest records
the fractions of a long time interval during which the empirical
distribution of the $N$ particles is close to different probability
measures on $\R^d$. We note that, under our conditions, the typical behavior of this object is given by the Dirac measure on the equilibrium measure $\mu_*$, namely
$\Gamma^{N,T_N} \to \delta_{\mu_*}$, in probability, as $N\to\infty$.

Large deviations for weakly interacting diffusions over a fixed time
horizon have a long history. Dawson and G\"artner
\cite{DawsonGartner1987} established a path-space large deviation
principle (LDP) for the empirical-measure process, $\{\mu^N_t: 0 \le t \le T\}$, for fixed $T<\infty$, as $N\to\infty$, with
speed $N$, and characterized the associated dynamical action
functional. The paper
\cite{BuDuFi} subsequently developed a
weak-convergence approach to this problem based on stochastic control
representations. This approach applies under quite general conditions
and avoids several of the discretization arguments and exponential estimates used in earlier
proofs. Related developments concerning the structure of the rate
function and more general interaction mechanisms include
\cite{Fischer2014,HoeksemaHoldingMaurelliTse2024}.

There is an important distinction between these finite-time results
and the problem considered here. A fixed-horizon large deviation
principle describes atypical trajectories
$\{\mu_t^N:0\leq t\leq T\}$ on a prescribed interval and has speed
$N$. Here the horizon $T_N$ itself diverges, with no restriction on
its rate of growth relative to $N$, and the relevant speed becomes
$NT_N$. A fixed-time large deviation principle cannot be directly 
applied when the time horizon also grows with $N$, since standard estimates used in its proof do not provide the
uniformity in the time horizon needed for this simultaneous limit. More significantly,
our observable is the occupation measure of the empirical-measure
process rather than its time-ordered trajectory. Thus, since $\{\mu_t^N:t\ge 0\}$ is a Markov process on $\cP(\R^d)$, our result is closer in spirit to large deviation results for occupation measures of Markov processes \cite{DonskerVaradhanI,DonskerVaradhanII,DonskerVaradhanIII} rather than Friedlin-Wentzell path space large deviations. However, note that, unlike these works, we do not have the occupation measure of a fixed Markov process, since both the time horizons $T_N$ and the Markov processes $\{\mu^N_t\}$ vary with $N$. There is also a useful analogy with a large deviation principle for occupation measures of small noise diffusions in which the time horizon and the inverse temperature parameter simultaneously become large. This connection is discussed later in the introduction.

A notable feature of the setting we consider is that the rate function associated with the LDP has a novel, interpretable, and particularly simple
structure.
Specifically, for $\mu(dx)=\rho(x)\,dx$ with a square-integrable weak
logarithmic gradient (see Section \ref{sec:1}), define
\begin{equation} \label{eq:5}
  \mathfrak i(\mu)
  =
  \frac14\int_{\R^d}
  \left|
    \nabla\log\rho(x)
    +\nabla V(x)
    +\nabla W*\mu(x)
  \right|^2\,\mu(dx),
\end{equation}
with $\mathfrak i(\mu)=\infty$ otherwise. Our main result gives an
LDP for $\{\Gamma^{N,T_N}\}$ with speed $NT_N$ and rate function
\begin{equation}\label{eq:6}
  \cI(\Gamma)
  =
  \int_{\cP(\R^d)}
  \mathfrak i(\mu)\,\Gamma(d\mu).
\end{equation}
The function $\mathfrak i(\mu)$ can be written in terms of the $L^2$-norm of the Wasserstein gradient of the free energy (see Section \ref{sec:1}), as
\begin{equation} \label{eq:7}
  4\mathfrak i(\mu)
  =
  \left\|
    \nabla\frac{\delta\cF}{\delta\mu}(\mu)
  \right\|_{L^2(\mu)}^2.
\end{equation}
Thus, roughly, the cost associated with deviation from $\delta_{\mu_*}$ to an atypical behavior of the form $\Gamma$ is given by the squared 
gradient of the free energy, associated with measures $\mu$ in the support of $\Gamma$, weighted according to the likelihood of $\mu$ under $\Gamma$.
 The functional $\mathfrak i$ is closely related to quantities that have
appeared previously in the analysis of granular-media equations. In fact,
$4\mathfrak i(\mu)$
  is the entropy-dissipation functional associated with the free energy
$\cF$ in \eqref{eq:11}; see, for example,
\cite[Proposition~2.1]{carmccvil}. This quantity can also be interpreted as the mean field relative Fisher information
for the McKean--Vlasov diffusion (cf. 
\cite[(3.30)]{TschidererYeung2023}).  
For related level-three Fisher-information functionals on
mixtures, see \cite{HaurayMischler2014}.

The formula for the rate function also has an interesting connection with classical,
Donsker--Varadhan theory,
\cite{DonskerVaradhanI,DonskerVaradhanII,DonskerVaradhanIII}. When
$W=0$, the particles are independent and, under conditions on $V$, $\mathfrak i$ is exactly the
Donsker--Varadhan rate function for the occupation measure of the
reversible diffusion
\[
  dX(t)=-\nabla V(X(t))\,dt+\sqrt2\,dB(t).
\]
Indeed, by the Dirichlet-form characterization of Donsker--Varadhan
\cite[Theorem~5, p.~7]{DonskerVaradhanI}
(cf.\ Guillin--L\'eonard--Wu--Yao \cite[Section~2.1]{GuillinLeonardWuYao2009}),
the rate function governing the occupation measure LDP for the above diffusion is
$I_{\rm DV}(\mu)
  =
  \int_{\R^d}|\nabla\sqrt f|^2\,d\pi$,
where 
$f=\frac{d\mu}{d\pi}$, and $\pi$ is the stationary distribution of the diffusion.
Consequently,
\begin{align*}
  I_{\rm DV}(\mu)
  &=
  \frac14\int_{\R^d}
  |\nabla\log f|^2\,d\mu =
  \frac14\int_{\R^d}
  |\nabla\log\rho+\nabla V|^2\,d\mu
  =
  \mathfrak i(\mu),
\end{align*}
where $\mu(dx)=\rho(x)\,dx$. However, note that this latter rate function gives large deviation behavior of random objects in $\cP(\R^d)$, whereas the result in the current work studies a LDP on $\cP(\cP(\R^d))$.

A related interpretation holds in the interacting setting through
the Dawson--G\"artner large deviation principle for weakly interacting
diffusions \cite{DawsonGartner1987}. For $\nu\in\cP(\R^d)$, let
\[
  \cL_\nu f
  =
  \Delta f
  -
  \bigl(\nabla V+\nabla W*\nu\bigr)\cdot\nabla f .
\]
Then the Dawson--G\"artner action for a path
$\{\nu_t\}_{0\leq t\leq T}$ (see \cite[(1.5)]{DawsonGartner1987}) can be written as
\[
  S_T(\nu)
  =
  \int_0^T
  \bigl\|
    \dot\nu_t-\cL_{\nu_t}^*\nu_t
  \bigr\|_{*,\nu_t}^2\,dt,
\]
where, 
\[
  \|\theta\|_{*,\nu}^2
  =
  \frac14
  \sup_{\varphi}
  \frac{\left(\int \varphi d\theta\right)^2}
       {\int_{\R^d}|\nabla\varphi|^2\,d\nu};
\]
and the above supremum is over all $\varphi \in C_c^{\infty}(\R^d)$ for which
$\int_{\R^d}|\nabla\varphi|^2\,d\nu>0$.
  Taking the constant path
$\nu_t\equiv\mu$, with $\mu(dx)=\rho(x)\,dx$, gives
\[
  -\cL_\mu^*\mu
  =
  -\nabla\cdot
  \left[
    \rho\bigl(
      \nabla\log\rho+\nabla V+\nabla W*\mu
    \bigr)
  \right]\,dx.
\]
Hence
\[
  S_T(\nu_t\equiv\mu) = T \bigl\|-\cL_\mu^*\mu\bigr\|_{*,\mu}^2
  =
  T \frac14
  \int_{\R^d}
  \left|
    \nabla\log\rho+\nabla V+\nabla W*\mu
  \right|^2\,d\mu
  =
  T \mathfrak i(\mu).
\]
 Hence $\mathfrak i(\mu)$ is the cost per unit
time of maintaining the empirical measure near $\mu$, and
$\cI(\Gamma)$ can be interpreted as the average of these holding costs when the fractions
of time spent near different profiles are described by $\Gamma$.

Thus the form of the rate function in \eqref{eq:6} is a natural conjecture for this problem; however, as will be seen, our proofs do not go through the Dawson--G\"artner machinery or the later results in \cite{BuDuFi} on this topic.
In a paper subsequent to \cite{DawsonGartner1987}, Dawson and G\"artner in
\cite{DawsonGartner1989} studied a related but different large deviation problem for this
class of models. For reversible mean-field interacting
diffusions they investigated the large deviation behavior of invariant
distributions as $N\to\infty$ and developed the 
dynamical quasipotential for this mean field setting and studied its relation with the free-energy functional, with
tunnelling and metastability among the motivating questions. Related long-time fluctuation questions for weakly interacting
diffusions were also considered in the earlier works
\cite{DawsonGartnerLongTime1987,DawsonGartnerLongTime1988}. The
object and scale considered in the current work are different. Rather than the
large deviations of the stationary distribution, or the cost of
transition between macroscopic states, we study the long-time
occupation measure itself. Correspondingly, the quantity appearing
in our rate function is not the free energy or quasipotential, but its
squared Wasserstein gradient, measuring the instantaneous cost of
holding the system away from the McKean--Vlasov equilibrium.

More recently, Gvalani and Schlichting \cite{GvalaniSchlichting2020}
studied free-energy barriers and metastability for weakly interacting
diffusions. In a specific model on the torus, they used such a barrier
to obtain an exponential-in-$N$ upper bound on the probability of
transition between competing macroscopic states. In a related direction, equilibrium large deviations for related mean-field systems have 
been studied by Reygner \cite{Reygner2018}, who treats
translation-invariant systems seen from their center of mass, including McKean--Vlasov particle
systems without an external confining potential, with a free-energy
type equilibrium rate function. However, all of these results are described in terms of the free energy and not its squared gradient, and to the best of our knowledge, the current work is the first to observe the role of the averaged entropy-dissipation functional in 
\eqref{eq:6} in large deviation analysis of weakly interacting diffusions over a long time horizon.
We note, however, that
entropy-dissipation methods, involving the functional above or the closely related Fisher-information
functionals, have previously been used to establish quantitative long-time
properties, including convergence rates to equilibrium for granular-media
equations \cite{carmccvil} and uniform-in-time propagation of chaos for
interacting diffusions \cite{LackerLeFlem2023}.

As mentioned previously, there is also an analogy with Freidlin--Wentzell small-noise
theory. At a heuristic level, the fluctuations of the empirical
measure around its McKean--Vlasov limit are of order $N^{-1/2}$, so
that $N^{-1}$ plays a role analogous to a small-noise variance. The
joint limit considered here is therefore analogous to the simultaneous
small-noise and long-time asymptotics studied for finite-dimensional diffusions in
\cite{BudhirajaZoubouloglou2024}. In that work the occupation measure
of a diffusion with noise amplitude $s(\varepsilon)$ over a time
horizon of order $\varepsilon^{-1}$ satisfies an LDP with speed
$(\varepsilon s^2(\varepsilon))^{-1}$, and the speed $NT_N$ for the LDP in the current work is consistent with the above analogy. A related
long-time/small-noise problem for stochastic evolution equations in
infinite dimensions, allowing the driving noise to become
asymptotically rough, is studied in
\cite{BudhirajaCerrai2026}. The present work differs from both of
these settings in that the effective state variable is itself a
probability measure and its deterministic limiting dynamics is the
nonlinear McKean--Vlasov flow.

We briefly describe the main ideas of the proof. The starting point
is the Bou\'e--Dupuis variational representation\cite{BoDu}, after rescaling the
time interval $[0,T_N]$ to $[0,1]$. This converts the Laplace
asymptotics needed for the large deviation analysis into a sequence of stochastic control problems. For the
large deviation upper bound, bounded control cost together with the dissipativity
condition (Assumption \ref{ass:3}) yields tightness of the controlled empirical measures.
The central difficulty is identifying the limiting control cost.
To understand this difficulty, it is useful to recall the analogous simultaneous small-noise and long-time problem for finite-dimensional diffusions studied in \cite{BudhirajaZoubouloglou2024}.
There, the limiting control cost is identified by applying It\^{o}'s formula to the potential generating the gradient drift, yielding a key orthogonality relation which furnishes an asymptotic lower bound on the cost (which is then seen to be optimal through the proof of the large deviation lower bound).
In the present setting, the natural analogue of this potential is the free energy $\cF$, since the McKean--Vlasov equation is its Wasserstein gradient flow.
However, the empirical measure $\mu_t^N$ is atomic and therefore has infinite
entropy, so that It\^{o}'s formula cannot be applied directly to
$\cF(\mu_t^N)$. We introduce a Gaussian regularization of the
empirical measure and apply It\^{o}'s formula to a corresponding
regularized free energy. One then sees an analogue of the orthogonality calculation 
in \cite{BudhirajaZoubouloglou2024} which
separates the control energy from the free-energy dissipation and
produces the desired lower bound. This is the most demanding part of the upper bound proof which appears in Section \ref{sec:5}, as 
Proposition \ref{prop:1}.
Controlling the diagonal terms created by the atomic empirical
measure, and showing that the regularization and integration-by-parts
remainders vanish, are important technical parts of this argument. See a more detailed discussion of the proof idea below the statement of Proposition \ref{prop:1}.

The proof that $\cI$ is a rate function once again uses
dissipativity (Assumption \ref{ass:3}). Under this assumption, the free energy is strongly geodesically convex in
Wasserstein space, which yields the coercive estimate
$\mathfrak i(\mu)
  \geq
  c\,\cW_2^2(\mu,\mu_*)$, where $\cW_2$ is the $2$-Wasserstein distance.
Together with a dual characterization of $\mathfrak i$ 
(see Lemma \ref{lem:8}), this gives
lower semicontinuity and compactness of the level sets of $\cI$.

The lower bound, notably, does not require the dissipativity
condition. We first approximate a near-optimal
$\Gamma\in\cP(\cP(\R^d))$ by a finite mixture of sufficiently regular
target measures. For such a regular target measure $\mu(dx)=\rho(x)\,dx$, the feedback
control is chosen to cancel the original McKean--Vlasov drift and
replace it by the score $\nabla\log\rho$. Under this control the
particles become independent copies of the score diffusion
\[
  dZ(t)=\nabla\log\rho(Z(t))\,dt+\sqrt2\,dB(t),
\]
whose invariant distribution is precisely $\mu$. Short transition
periods, whose length vanishes on the $T_N$ scale but tends to infinity
on the original time scale, allow the system to equilibrate near
successive target measures; the non-vanishing hold periods produce the
desired occupation fractions. This construction simultaneously
recovers the prescribed occupation measure and the desired cost.

 In Section~\ref{sec:15}, we discuss an explicit quadratic example
that illustrates several features of the result. For
this setting the local cost $ \mathfrak i(\mu)$
 separates into a center-of-mass contribution and a
centered relative Fisher-information contribution (see Proposition \ref{prop:9}).
This permits a number of constrained variational problems to be solved
explicitly. Specifically, it shows that, for some such variational problems, Gaussian measures are the unique optimizers
under prescribed mean and covariance constraints. As another striking consequence,
the variational problem associated with an atypically large overall
space-time variance exhibits a transition, in that, below a critical
value the optimum is a single Gaussian profile, whereas above it the
optimum is a nontrivial occupation measure supported on Gaussian
profiles with different centers. Thus excess dispersion is achieved
more efficiently by spending time near several displaced clouds than
by broadening a single cloud.  

Finally, we remark on the role of dissipativity in our proofs. Such a condition is needed due to the
noncompactness of $\R^d$. On a compact state space, such as
$\mathbb T^d$, with smooth periodic $V,W$ and even $W$, we expect an
analogous result to hold without convexity or dissipativity. In the nonconvex setting, however, the
McKean--Vlasov equation may have several stationary distributions,
and the zero set of $\cI$ should then consist of occupation measures
supported on this collection rather than a single
$\delta_{\mu_*}$.  
This possibility of multiple stationary states in the periodic setting
is well understood; see, for example,
Carrillo, Gvalani, Pavliotis and Schlichting
\cite{CarrilloGvalaniPavliotisSchlichting2020},
who analyze bifurcations and phase transitions for the
McKean--Vlasov equation on the torus.
This compact state space setting also leads to interesting
metastability questions related to transitions between distinct stable stationary states.
Related particle-level metastability results include
Bashiri and Menz \cite{BashiriMenz2021}, who establish
Eyring--Kramers-type asymptotics for transitions between macroscopic
wells in a continuous mean-field model.
We remark that these transitions are governed by finer, $N$-dependent quasipotential scales which are
not resolved by the $NT_N$-scale LDP considered here.  
We leave the
compact nonconvex setting and its connection with metastable behavior
for future work.

\subsection{Notation}
The following notation will be used. For $n\in\mathbb{N}$, write
$[n]\doteq\{1,\ldots,n\}$. For $m\in\mathbb{N}$, let $C^m(\mathbb{R}^d)$ denote the space of functions
$f:\mathbb{R}^d\to\mathbb{R}$ whose partial derivatives $\partial^\alpha f$ exist and are continuous for every multi-index $\alpha$ with $|\alpha|\le m$. We write
$
C^\infty(\mathbb{R}^d)
\doteq
\bigcap_{m=1}^\infty C^m(\mathbb{R}^d)$
for the space of infinitely differentiable functions on $\mathbb{R}^d$.
We write $C_c^\infty(\mathbb{R}^d)$ for those with compact support and similarly use
$C_c^\infty(\mathbb{R}^d;\mathbb{R}^q)$ for vector-valued functions.
For a Polish space $S$, let $C_b(S)$ be the space of continuous and bounded real functions on $S$ and let $\mathcal{P}(S)$ denote the space of Borel probability measures on $S$, equipped with the topology of weak convergence, metrized by the bounded Lipschitz distance
\[
d_{\mbox{\tiny{BL}}}(\mu,\nu)
\doteq
\sup_{f \in \mbox{\tiny{BL}}_1}
\left|
\langle f,\, \mu \rangle- \langle f,\, \nu \rangle
\right|.
\]
Here, $\langle f,\, \gamma \rangle = \int_S f d\gamma$ for $\gamma \in \clp(S)$,
$
\mbox{\small{BL}}_1\doteq \{f:S\to\mathbb{R}: \|f\|_\infty\le 1,\ \operatorname{Lip}(f)\le 1\}$,
where
$$\|f\|_\infty \doteq \sup_{x\in S}|f(x)|,
\qquad
\operatorname{Lip}(f)
\doteq
\sup_{\substack{x,y\in S\\ x\neq y}}
\frac{|f(x)-f(y)|}{d_S(x,y)},
$$
and $d_S$ is a metric generating the topology of $S$.
For an interval $J\subset[0,\infty)$, we denote by $C(J;S)$ the space of continuous functions from $J$ to $S$, equipped with the topology of uniform convergence on compact time intervals.
For a measurable map $T:S\to S'$ and $\mu\in\cP(S)$, $T_\#\mu$ denotes the pushforward of $\mu$ under $T$.
Let
\[
\mathcal{P}_2(\mathbb{R}^d)
\doteq
\left\{
\mu\in\mathcal{P}(\mathbb{R}^d):
\int_{\mathbb{R}^d}|x|^2\,\mu(dx)<\infty
\right\}.
\]
We equip $\mathcal{P}_2(\mathbb{R}^d)$ with the $2$-Wasserstein distance
\[
\mathcal{W}_2(\mu,\nu)
\doteq
\left(
\inf_{\pi\in\Pi(\mu,\nu)}
\int_{\mathbb{R}^d\times\mathbb{R}^d}
|x-y|^2\,\pi(dx,dy)
\right)^{1/2},
\]
where $\Pi(\mu,\nu)$ denotes the set of couplings of $\mu$ and $\nu$, namely the collection of $\Gamma \in \clp(\R^d\times \R^d)$ such that
$[\Gamma]_1(\cdot) \doteq \Gamma(\cdot \times \R^d) = \mu$ and $[\Gamma]_2(\cdot) \doteq \Gamma(\R^d \times \cdot) = \nu$. We write $\clx\doteq\cP(\cP(\R^d))$.

For a finite Borel measure $\mu$ on $\mathbb{R}^d$ and a measurable function
$f:\mathbb{R}^d\to\mathbb{R}$, their convolution is defined by
$(f*\mu)(x)
\doteq
\int_{\mathbb{R}^d} f(x-y)\,\mu(dy)$,
$x\in\mathbb{R}^d$,
whenever the integral is well defined.
For $k \in \NN$, Polish spaces $S_i$, $i \in [k]$, and $\Gamma \in \clp(S_1 \times S_2 \times \cdots S_k)$, $[\Gamma]_i \in \clp(S_i)$ denotes the $i$-th marginal distribution of $\Gamma$. For $1\le i < j \le k$, $[\Gamma]_{ij} \in \clp(S_i\times S_j)$ is defined similarly. For a $d$-dimensional vector $m$ and a $d\times d$ non-negative definite matrix $\Sigma$, we denote by $\cln(m, \Sigma)$  the $d$-dimensional Normal distribution with mean $m$ and covariance matrix $\Sigma$.

\section{Setting and Main Result}
\label{sec:1}

Let $V$ and $W$ be as in Section \ref{sec:intro}. Consider the weakly interacting $N$ particle system, defined by \eqref{eq:1}
with $X^N_i(0) = x^N_i$,
$i=1,\dots,N$,
 where $x^N_i \in \R^d$ for all $N \in \NN$ and $i \in [N]$.  
We will make the following basic assumption on the functions $V,W$.

\begin{assumption}
\label{ass:1}
The functions $V,W\in C^1(\R^d)$ and their gradients
$\nabla V$ and $\nabla W$ are globally Lipschitz.
\end{assumption}

Assumption~\ref{ass:1} 
ensures that, for every $N\in\N$, the particle system
\eqref{eq:1}, as well as its controlled version appearing
in the variational representation, has a unique pathwise solution. This assumption will be taken to hold throughout and therefore will not be noted explicitly in statements of various results.

We will also need one of the two conditions in the following assumption. The first condition will be used in the proof of the upper bound, while the second stronger condition will be needed for the lower bound proof.
\begin{assumption}
\label{ass:2}
\begin{enumerate}[label={(\alph*)}] 
\item We have
$
 \sup_{N\ge1}
 \frac1N\sum_{i=1}^N|x^N_i|^2
 <\infty.$
\item There exists $\alpha>0$ such that
$
\sup_{N\ge1}
\frac1N\sum_{i=1}^N
|x^N_i|^{2+\alpha}
<\infty$.
\end{enumerate}
\end{assumption}

The lower bound in the LDP will be proved under Assumption~\ref{ass:1} and
Assumption~\ref{ass:2}(b).

For the upper bound and the compactness of the level sets of the rate
function, we strengthen Assumption~\ref{ass:1} as follows.

\begin{assumption}
\label{ass:3}
The functions $V,W\in C^2(\R^d)$, $W$ is even, and there exist
constants $\lambda\in(0,\infty)$ and $C\in[0,\infty)$ with $\lambda>C$
such that, for all 
$x,y,z,z'\in\R^d$,
\begin{align}
 &(x-y)\cdot\bigl(\nabla V(x)-\nabla V(y)\bigr)
 +(x-y)\cdot
 \bigl(\nabla W(x-z)-\nabla W(y-z')\bigr)
 \ge
 \lambda|x-y|^2-C|z-z'|^2.
\end{align}
\end{assumption}

\begin{remark}
Assumption~\ref{ass:3} describes a strongly confined
regime in which the interaction need not be convex, provided its
destabilizing effects are dominated by the confinement. Natural
examples include uniformly convex $V$ together with a convex
interaction having controlled curvature, or a sufficiently small
nonconvex perturbation. This is closely related to
\cite[Theorem~2.1]{carmccvil}, where $V$ is uniformly convex and $W$
is convex; part~(iv) there allows nonconvex $W$ when its negative curvature
is sufficiently small relative to that of $V$. For example, when
$V(x)=\frac m2|x|^2$ and $W(x)=-\frac\beta2|x|^2$, our condition holds
precisely when $\beta<m/2$, which is the condition that appears there. In
contrast, the  setting of
\cite[Section~10.4.7 and Example~11.2.7]{AGS} assumes separate
convexity of $V$ and $W$, whereas our condition directly yields strong
convexity of the combined free energy and permits some nonconvex
interactions.

Assumption~\ref{ass:3} necessarily makes $V$ uniformly
convex. Indeed, taking $z=x$ and $z'=y$ in
the inequality in the assumption gives uniform convexity of $V$ with modulus
$\lambda-C$. If $V=0$, translation invariance of both the entropy and
interaction terms prevents the free energy $\cF$ from being strongly
geodesically convex on $\cP_2(\R^d)$, and the interaction controls only
moments about the center of mass. Consequently, the absolute moment
estimate in Lemma~\ref{lem:2} and the compactness
argument in Proposition~\ref{prop:8} no longer apply.
A centered formulation based on
$\widetilde\mu_t^N=(x\mapsto x-\bar x(\mu_t^N))_\#\mu_t^N$, where
$\bar x(\mu)=\int x\,\mu(dx)$, which takes values in the zero-barycenter
subspace of $\cP_2(\R^d)$,  may permit treatment of this degenerate
setting. This will be explored in future work.

Finally, the global Lipschitz assumptions on $\nabla V$ and $\nabla W$
provide well-posedness, linear-growth and moment bounds, and the
regularization and approximation estimates used, for example, in
Lemma~\ref{lem:6} and
Proposition~\ref{prop:3}. They are not intrinsic
to the dissipativity mechanism and could be replaced by suitable
growth, moment, continuity, and well-posedness hypotheses.
\end{remark}

Define the empirical measure process $\{\mu_t^N, t \ge 0\}$ by \eqref{eq:2}.
Then for each $t \ge 0$ and $N \in \NN$, $\mu_t^N$ is a $\cP(\R^d)$ valued random variable and $t\mapsto \mu_t^N$ has sample paths in 
$C([0,\infty);\cP(\R^d))$.

For $T>0$, define the occupation measure of the empirical measure process by \eqref{eq:4}
Then $\Gamma^{N,T}$ is a random element of $\cP(\cP(\R^d))$.
Consider the
McKean--Vlasov equation \eqref{eq:3}, associated with the weakly interacting particle system,
with initial condition 
 $\mu_0 \in \mathcal P_2(\mathbb R^d)$. Under Assumption~\ref{ass:1}, there is a unique weak solution
 $\mu\in C([0,\infty);\mathcal P(\mathbb R^d))$ of \eqref{eq:3} satisfying
 \[
 \sup_{t\leq T}\int_{\mathbb R^d}|x|^2\,\mu_t(dx)<\infty
 \qquad\text{for every }T<\infty;
 \]
 see \cite[Theorem~2.9 and Proposition~5.1]{gart}. Standard mean-square estimates for the associated
 McKean--Vlasov SDE further give $\mu\in C([0,\infty);\mathcal P_2(\mathbb R^d))$.
We will write $\mu_t = \Phi_t(\mu_0)$ and note that $\{\Phi_t\}$ defines a semigroup on $\cP_2(\R^d)$, namely
$\Phi_t(\Phi_s(\mu)) = \Phi_{t+s}(\mu)$ for all $t,s \ge 0$ and $\mu \in \cP_2(\R^d)$.

In terms of the particle system, under Assumption~\ref{ass:2}(a), if 
$\mu^N_0 = \frac1N\sum_{i=1}^N \delta_{x^N_i}$ converges in $\cP(\R^d)$ to $\mu_0$, it 
follows from the classical law of large numbers for weakly
interacting diffusions \cite[Theorem~3.1]{gart} that 
the empirical measure process $\mu^N$ converges in probability in
$C([0,\infty);\cP(\R^d))$ to the unique solution $\{\mu_t\}_{t\ge 0}$ of 
\eqref{eq:3}.

Furthermore, for every
$\mu_0\in\mathcal P_2(\mathbb R^d)$ and every $t>0$, the measure
$\mu_t$ is absolutely continuous with respect to Lebesgue measure.
This follows, for example, from the parabolic regularity result
\cite[Corollary~3.9]{bogkryroc}. Indeed, for a fixed
McKean--Vlasov solution $(\mu_t)$ the equation is a linear
Fokker--Planck equation with uniformly elliptic diffusion matrix and
locally bounded drift
$b(t,x)=-\nabla V(x)-(\nabla W*\mu_t)(x)$ and results of \cite{bogkryroc} can be applied.
Writing
$\mu_t(dx)=\rho_t(x)\,dx$,
the family $(\rho_t)_{t>0}$ satisfies
\begin{equation}\label{eq:10}
  \partial_t\rho_t
  =
  \Delta\rho_t+
  \nabla\cdot\bigl(\rho_t\nabla(V+W*\rho_t)\bigr)
\end{equation}
on $(0,\infty)\times\mathbb R^d$ in the sense of distributions, with
$\rho_t(x)\,dx\to\mu_0$ in $\mathcal P_2(\mathbb R^d)$ as $t\downarrow0$.

Suppose now that Assumption \ref{ass:3} holds.
Define the {\em free energy} $\cF: \cP_2(\R^d) \to (-\infty,\infty]$ as
\begin{equation}\label{eq:11}
  \cF(\mu)
  = \int_{\R^d} V(x)\,\mu(dx)
    +\frac12\iint_{\R^d\times\R^d} W(x-y)\,\mu(dx)\mu(dy)
    + \int_{\R^d}\rho^{\mu}(x)\log\rho^{\mu}(x)\,dx,
\end{equation}
if $\mu$ admits a density $\rho^{\mu}$ and set $\cF(\mu)=\infty$ otherwise.
Let $\mbox{Dom}(\cF) = \{\mu \in \cP_2(\R^d): \cF(\mu)<\infty\}$.
For $\mu(dx)=\rho(x)\,dx\in \mbox{Dom}(\cF)$, 
if $\rho$ admits a weak logarithmic gradient
$\nabla\log\rho\in L^2(\mu;\R^d)$ in the sense discussed below
\eqref{eq:14}, then the Wasserstein gradient of $\cF$ at $\mu$ is
given by (see Lemma~\ref{lem:8} and its proof)
\begin{equation}\label{eq:13}
  \nabla\frac{\delta\cF}{\delta\mu}(\mu)
  =
  \nabla V+\nabla W*\mu+\nabla\log\rho
  \qquad \mu\text{-a.e.}
\end{equation}
Using \eqref{eq:13}, the McKean--Vlasov equation
\eqref{eq:10} can be written formally as
\[
\partial_t\rho_t
=
\nabla\cdot
\left(
\rho_t
\nabla\frac{\delta\cF}{\delta\mu}(\mu_t)
\right),
\]
which is the Wasserstein gradient-flow equation for $\cF$.
Indeed, by Lemma~\ref{lem:24} and the standard
lower-semicontinuity and coercivity properties of the entropy,
potential, and interaction terms, $\cF$ satisfies the hypotheses of
\cite[Theorem~11.2.1]{AGS}, and the first-variation argument of
\cite[Theorem~10.4.13]{AGS}, with the combined convexity supplied by
Lemma~\ref{lem:24}, shows that the
 Wasserstein gradient flow for $\cF$ solves
\eqref{eq:10}. By uniqueness of the weak solution noted
above, the semigroup of this gradient flow coincides with $(\Phi_t)_{t\ge0}$. See 
\cite[Example~11.2.7 and Theorem~11.2.8]{AGS} for a result of this type under slightly different assumptions on $V$ and $W$.

Under Assumption~\ref{ass:3}, the flow is globally
exponentially stable. Indeed, as noted above,
Lemma~\ref{lem:24} shows that the free
energy $\cF$ is $\kappa$-convex along generalized geodesics, where
$\kappa=\lambda-C>0$. Hence
\cite[Theorem~11.2.1, (11.2.2) and (11.2.5)]{AGS} shows that $\cF$ has a
unique minimizer $\mu_*$, which is the unique fixed point of
$(\Phi_t)_{t\ge0}$, and
\[
  \cW_2(\Phi_t(\mu),\mu_*)
  \le e^{-\kappa t}\cW_2(\mu,\mu_*),
  \qquad \mu\in\cP_2(\R^d),\quad t\ge0.
\]

\subsection{Main Result}
Fix a positive sequence $\{T_N\}$ such that $T_N \to \infty$ as $N\to \infty$.
In this work, we will establish a large deviation principle (LDP) for the collection $\{\Gamma^{N,T_N}\}_{N \in \N}$ of $\cP(\cP(\R^d))$-valued random variables.
We recall the following definition.
\begin{definition}\label{def:1}
Let $S$ be a Polish space. and let $I:S\to[0,\infty]$ be a given function.
A sequence of $S$-valued random variables $\{X^N\}$ is said to satisfy a large deviation lower bound on $S$ with speed $a_N\to\infty$ and function $I$ if, for every open set $G\subset S$
$$\liminf_{N\to\infty}\frac{1}{a_N}\log\mathbb{P}(X^N\in G) \ge -\inf_{x\in G} I(x),$$
and it is said to satisfy the corresponding large deviation upper bound if for every closed $F \subset S$
$$\limsup_{N\to\infty}\frac{1}{a_N}\log\mathbb{P}(X^N\in F)
\le
-\inf_{x\in F} I(x).$$
The function $I$ is called a rate function if its level sets
$\{x\in S:I(x)\le M\}$
are compact for every $M<\infty$.  
\end{definition}

We now introduce the function that will characterize the large deviation behavior of $\{\Gamma^{N,T_N}\}_{N \in \N}$.

We recall that, for $\mu(dx)=\rho(x)\,dx\in\cP_2(\R^d)$, the
\emph{weak logarithmic gradient} of $\rho$ is said to exist if there
is a vector field $\xi\in L^2(\mu;\R^d)$ such that
\begin{equation}\label{eq:14}
  \int_{\R^d}\partial_i\varphi(x)\,\mu(dx)
  =
  -\int_{\R^d}\varphi(x)\xi_i(x)\,\mu(dx),
  \qquad
  \varphi\in C_c^\infty(\R^d),\quad i=1,\ldots,d.
\end{equation}
In this case the distributional derivative of $\rho$ is given by
$\nabla\rho=\rho\,\xi$.
Since $\xi\in L^2(\mu;\R^d)$, we have
$\rho\xi\in L^1_{\rm loc}(\R^d;\R^d)$, and hence
$\rho\in W^{1,1}_{\rm loc}(\R^d)$. We write
$\nabla\log\rho \doteq \xi$
   $\mu$-a.e.
and call this vector field the weak logarithmic gradient of $\rho$.
See \cite[Lemma~3.5]{HaurayMischler2014} and also
 \cite[Theorem~10.4.6]{AGS}, specialized to $F(r)=r\log r$ and
$p=q=2$.

For $\mu\in\cP_2(\R^d)$ with density $\rho$, such that its weak logarithmic gradient $\nabla \log \rho$ exists, define
$\mathfrak i(\mu)$ by \eqref{eq:5}.
Note that this quantity is finite due to linear growth of $\nabla V$ and $\nabla W$.
We set $\mathfrak i(\mu)=\infty$ for all other $\mu \in \cP(\R^d)$. Equivalently, from \eqref{eq:13}, when it is finite, $\mathfrak i(\mu)$ is given by
\eqref{eq:7}.

For $\Gamma\in\cP(\cP(\R^d))$, define $\cI(\Gamma)$ by \eqref{eq:6}, namely
\begin{equation}\label{eq:15}
  \cI(\Gamma)
  = \int_{\cP(\R^d)} \mathfrak i(\mu)\,\Gamma(d\mu) = \frac14\int_{\cP(\R^d)} \left\|
      \nabla\frac{\delta\cF}{\delta\mu}(\mu)
    \right\|_{L^2(\mu)}^2\,\Gamma(d\mu).
\end{equation}

It can be seen that $\cI$ vanishes at $\delta_{\mu_*}$, namely $\cI(\delta_{\mu_*})=0$. Indeed, 
the invariant measure $\mu_*$ has a strictly positive density $\rho_*$
given by the formula
\begin{equation}\label{eq:16}
  \rho_*(x)
  =
  \frac{1}{Z_*}
  \exp\bigl\{-V(x)-(W*\mu_*)(x)\bigr\},
  \qquad
  Z_*
  =
  \int_{\mathbb R^d}
  \exp\bigl\{-V(x)-(W*\mu_*)(x)\bigr\}\,dx.
\end{equation}
Consequently,
\begin{equation}\label{eq:17}
  \nabla\log\rho_*
  =
  -\nabla V-\nabla W*\mu_*,
\end{equation}
and hence, using \eqref{eq:5},
\[
  \mathfrak i(\mu_*)
  =
  \frac14\int_{\mathbb R^d}
  \left|
    \nabla V+\nabla W*\mu_*+\nabla\log\rho_*
  \right|^2\,d\mu_*
  =0.
\]

Since $\mathfrak i\geq0$, the preceding calculation shows that
$\cI(\delta_{\mu_*})=0$. For the converse, suppose that
Assumption~\ref{ass:3} holds.  Then the coercivity
estimate in Lemma~\ref{lem:25} gives
$\mathfrak i(\mu)=0$ if and only if $\mu=\mu_*$. Consequently, under
Assumption~\ref{ass:3}, for $\Gamma\in\cP(\cP(\R^d))$,
$\cI(\Gamma)=0$
  if and only if
  $\Gamma=\delta_{\mu_*}$.

The following is the main result. It uses an equivalent formulation of a LDP in terms of Laplace asymptotics (cf. \cite[Theorem 1.8]{buddupbook}).

\begin{theorem}\label{thm:1}
Suppose that Assumption \ref{ass:2}(b) holds.
Then the family
$
  \{\Gamma^{N,T_N}\}_{N\ge1}$
satisfies a large deviation lower bound on $\clx$ with speed $NT_N$ and function $\cI$ given by \eqref{eq:15}, equivalently, for every 
$F\in C_b(\clx)$,
\begin{equation}\label{eq:18}
  \limsup_{N\to\infty}
  -\frac1{NT_N}\log
  \mathbb E\exp\{-NT_N F(\Gamma^{N,T_N})\}
  \le
  \inf_{\Gamma\in\clx}
  \{F(\Gamma)+\cI(\Gamma)\}.
\end{equation}

Suppose that Assumption \ref{ass:2}(a) and Assumption 
\ref{ass:3} hold. Then
$\cI: \clx \to [0, \infty]$ is a rate function and the family
$
  \{\Gamma^{N,T_N}\}_{N\ge1}$
satisfies a large deviation upper bound on $\clx$ with speed $NT_N$ and rate function $\cI$ given by \eqref{eq:15}, namely, for every 
$F\in C_b(\clx)$,
\begin{equation}\label{eq:19}
  \liminf_{N\to\infty}
  -\frac1{NT_N}\log
  \mathbb E\exp\{-NT_N F(\Gamma^{N,T_N})\}
  \ge
  \inf_{\Gamma\in\clx}
  \{F(\Gamma)+\cI(\Gamma)\}.
\end{equation}

Consequently under Assumptions \ref{ass:2}(b) and 
\ref{ass:3}, 
$\{\Gamma^{N,T_N}\}_{N\ge1}$
satisfies a LDP with speed $NT_N$ and rate function $\cI$.
\end{theorem}

\begin{remark}
\label{rem:1}
The strict inequality in Assumption~\ref{ass:3} is important
for the rate-function property as can be seen from the following simple example.
Let
$V(x)=0$,
$W(x)=\frac12|x|^2$.
Then Assumption~\ref{ass:3} fails, since the inequality in this assumption becomes an equality with $\lambda=C=1/2$.

For $m\in\R^d$, let
$\mu_m=\mathcal N(m,I_d)
$.
Its density $\rho_m$ satisfies
$
  \nabla\log\rho_m(x)=-(x-m)
$
and
$
  \nabla W*\mu_m(x)
  =
  \int_{\R^d}(x-y)\,\mu_m(dy)
  =
  x-m.
$
Consequently,
$
  \mathfrak i(\mu_m)=0$,
  for all $m\in\R^d$.
In particular,
$
  \cI(\delta_{\mu_m})=0$,
for all $m\in\R^d$.
Since the set $\{\delta_{\mu_{m}}, m \in \R^d\}$ is not relatively compact in $\cP(\cP(\R^d))$, we see that the 
zero level set of $\cI$ is not compact. 

This example also explains the role of the coercivity estimate in
Lemma~\ref{lem:25}. When
$\kappa=\lambda-C>0$, that lemma gives
$\mathfrak i(\mu)
  \ge
  \frac{\kappa^2}{4}
  \cW_2^2(\mu,\mu_*),
$
which rules out the behavior exhibited above.
\end{remark}

In the remainder of the paper, we complete the proof of Theorem~\ref{thm:1}. The argument is organized into four sections. We first establish in 
Section \ref{sec:2} a variational representation for the Laplace functional on the left side of the inequalities in Theorem~\ref{thm:1}. We then prove the large deviation upper bound in 
Section \ref{sec:3} and the large deviation lower bound in Section \ref{sec:8}. Finally, in Section \ref{sec:14}, we show that the functional $\cI$ defined in \eqref{eq:15} is a rate function on $\cP(\cP(\R^d))$.

\section{Variational Representation}
\label{sec:2}

It will be convenient to rescale time so that the occupation measure is written on a fixed
time interval. For $t\in[0,1]$, define
\begin{equation}\label{eq:20}
  \widetilde X_i^N(t) \doteq X_i^N(T_N t),\; \widetilde B_i(t) \doteq T_N^{-1/2} B_i(T_N t),
  \qquad i=1,\ldots,N.
\end{equation}
Then $\widetilde B_1,\ldots,\widetilde B_N$ are independent standard
$d$-dimensional Brownian motions. The time-scaled particles solve
\begin{equation}\label{eq:21}
\begin{aligned}
  d\widetilde X_i^N(t)
  &=
  -T_N \nabla V(\widetilde X_i^N(t))\,dt
  -T_N \nabla W*\widetilde\mu_t^N(\widetilde X_i^N(t))\,dt
  +\sqrt{2T_N}\,d\widetilde B_i(t),
  \\
  \widetilde X_i^N(0)&=x^N_i,
  \qquad i\in [N],
\end{aligned}
\end{equation}
where
\begin{equation}\label{eq:22}
  \widetilde\mu_t^N
  \doteq
  \frac1N\sum_{i=1}^N\delta_{\widetilde X_i^N(t)}
  =
  \mu_{T_Nt}^N .
\end{equation}
Consequently,
\begin{equation}\label{eq:23}
  \Gamma^{N,T_N}
  =
  \int_0^1 \delta_{\widetilde\mu_t^N}\,dt.
\end{equation}

Let $\cA_N$ denote the collection of progressively measurable (with respect to the standard augmentation of the filtration generated by $\{\tilde B_i, 1 \le i \le N\}$) processes
\[
  v^N=(v_1^N,\ldots,v_N^N):[0,1]\times\Omega\to(\R^d)^N
\]
such that
\[
  \mathbb E\left[
    \sum_{i=1}^N\int_0^1 |v_i^N(t)|^2\,dt
  \right]<\infty .
\]
For $v^N\in\cA_N$, define the controlled time-scaled system
\begin{equation}\label{eq:24}
\begin{aligned}
  d\bar X_i^N(t)
  &=
  -T_N \nabla V(\bar X_i^N(t))\,dt
  -T_N \nabla W*\bar\mu_t^N(\bar X_i^N(t))\,dt \\
  &\qquad
  +\sqrt{2T_N}\,v_i^N(t)\,dt
  +\sqrt{2T_N}\,d\widetilde B_i(t),
  \qquad i=1,\ldots,N,
\end{aligned}
\end{equation}
with
\begin{equation}\label{eq:25}
  \bar\mu_t^N
  =
  \frac1N\sum_{i=1}^N\delta_{\bar X_i^N(t)}
\end{equation}
and controlled occupation measure
\begin{equation}\label{eq:26}
  \bar\Gamma^{N,T_N}
  =
  \int_0^1 \delta_{\bar\mu_t^N}\,dt .
\end{equation}

By the Bou\'e--Dupuis variational representation \cite{BoDu} 
(see also \cite[Theorem 8.3]{buddupbook}) applied to the
$Nd$-dimensional Brownian motion
$(\widetilde B_1,\ldots,\widetilde B_N)$, for every
$F\in C_b(\clx)$,
\begin{equation}\label{eq:27}
\begin{aligned}
  -\frac1{NT_N}\log
  \mathbb E\exp\{-NT_N F(\Gamma^{N,T_N})\}
  =
  \inf_{v^N\in\cA_N}
  \mathbb E\Bigg[
    \frac1{2NT_N}
    \sum_{i=1}^N\int_0^1 |v_i^N(t)|^2\,dt
    +
    F(\bar\Gamma^{N,T_N})
  \Bigg].
\end{aligned}
\end{equation}

Replacing each control $v_i^N$ in
\eqref{eq:24} by
$v_i^N=\sqrt{\frac{T_N}{2}}\,u_i^N$, 
for $i\in [N]$, where $u^N \in \cA_N$,
\begin{equation}\label{eq:28}
\begin{aligned}
  d\bar X_i^N(t)
  &=
  -T_N \nabla V(\bar X_i^N(t))\,dt
  -T_N \nabla W*\bar\mu_t^N(\bar X_i^N(t))\,dt
  +T_Nu_i^N(t)\,dt
  +\sqrt{2T_N}\,d\widetilde B_i(t).
  \end{aligned}
\end{equation}
With this normalization, the variational representation takes the form
\begin{equation}\label{eq:29}
\begin{aligned}
  -\frac1{NT_N}\log
  \mathbb E\exp\{-NT_NF(\Gamma^{N,T_N})\}
  =
  \inf_{u^N\in\cA_N}
  \mathbb E\Bigg[
    \frac1{4N}
    \sum_{i=1}^N\int_0^1|u_i^N(t)|^2\,dt
    +
    F(\bar\Gamma^{N,T_N})
  \Bigg].
\end{aligned}
\end{equation}
\section{Upper Bound} \label{sec:3}
Throughout this section Assumption \ref{ass:2}(a) and Assumption 
\ref{ass:3} will be taken to hold.
The goal of this section is to prove the large deviation upper bound under these assumptions. In Section \ref{sec:14} we will prove that under these assumptions $\cI$ is a rate function. Thus, in view of \cite[Theorem 1.8]{buddupbook}, it suffices to
 show that for every $F\in C_b(\clx)$,{}
\begin{equation}\label{eq:30}
  \liminf_{N\to\infty}
  -\frac1{NT_N}\log
  \mathbb E\exp\{-NT_N F(\Gamma^{N,T_N})\}
  \ge
  \inf_{\Gamma\in\clx}
  \{F(\Gamma)+\cI(\Gamma)\}.
\end{equation}

One can assume without loss of generality (see \cite[Corollary 1.10]{buddupbook}) that $F$ is Lipschitz
in the $d_{\mathrm{BL}}$ (bounded-Lipschitz) metric generating the weak topology
on $\clx=\cP(\cP(\R^d))$, with Lipschitz
constant $L_F<\infty$.

Now fix such a $F$ and, using \eqref{eq:29}, for each $N$, select $u^N\in\cA_N$ such that
\begin{equation}\label{eq:31}
  -\frac1{NT_N}\log
 \mathbb E\exp\{-NT_N F(\Gamma^{N,T_N})\}+\frac{1}{N}
  \ge
  \mathbb E\Bigg[
    \frac1{4N}
    \sum_{i=1}^N\int_0^1|u_i^N(t)|^2\,dt
    +
    F(\bar\Gamma^{N,T_N})\Bigg].
\end{equation}
As an immediate consequence of the above inequality we have that 
\begin{equation}\label{eq:32}
\E\left[
     \int_0^1
     \frac1N\sum_{i=1}^N |u_i^N(t)|^2\,dt
  \right]
  \le 4(2\|F\|_{\infty} + 1) \doteq C_F.
  \end{equation}
We now modify the controls a little, which will be useful later (see proof of Lemma \ref{lem:10}).

\begin{lemma}
\label{lem:1}
Let
$u^N=(u_1^N,\ldots,u_N^N)\in\cA_N$ be as above.
Let $\varepsilon_N>0$ satisfy
  $\varepsilon_N\to0$,
  $T_N\varepsilon_N\to\infty$,
and define
\begin{equation}\label{eq:33}
  \hat u_i^N(t)
  \doteq
  u_i^N(t)\mathbf 1_{[0,1-\varepsilon_N]}(t),
  \qquad i=1,\ldots,N.
\end{equation}
Let $\hat X_i^N$, $\hat\mu_t^N$, and
$\hat\Gamma^{N,T_N}$ denote the corresponding controlled particles,
empirical measures, and occupation measure. Then
\begin{equation}\label{eq:34}
\begin{aligned}
  &-\frac1{NT_N}
  \log \E\exp\{-NT_NF(\Gamma^{N,T_N})\}
  +\frac1N+2L_F\varepsilon_N\\
  &\qquad\ge
  \E\left[
    \frac1{4N}\sum_{i=1}^N\int_0^1|\hat u_i^N(t)|^2\,dt
    +
    F(\hat\Gamma^{N,T_N})
  \right].
\end{aligned}
\end{equation}

Moreover,
\begin{equation}\label{eq:35}
  \E\left[
    \frac1N\sum_{i=1}^N\int_0^1|\hat u_i^N(t)|^2\,dt
  \right]
  \le
  \E\left[
    \frac1N\sum_{i=1}^N\int_0^1|u_i^N(t)|^2\,dt
  \right],
\end{equation}
and, for every $a>0$,
\begin{equation}\label{eq:36}
  \E\int_0^1
    e^{-aT_N(1-t)}
    \frac1N\sum_{i=1}^N|\hat u_i^N(t)|^2\,dt
  \to0.
\end{equation}
\end{lemma}

\begin{proof}
Set
\[
  U_N(t)
  \doteq
  \frac1N\sum_{i=1}^N|u_i^N(t)|^2,
  \qquad
  \hat U_N(t)
  \doteq
  \frac1N\sum_{i=1}^N|\hat u_i^N(t)|^2.
\]
The inequality in \eqref{eq:35} is immediate by construction,
since
\[
  \hat U_N(t)
  =
  U_N(t)\mathbf 1_{[0,1-\varepsilon_N]}(t),\;\;
  \int_0^1\hat U_N(t)\,dt
  \le
  \int_0^1U_N(t)\,dt.
\]
Construct the controlled systems as in \eqref{eq:28}, associated with $u^N$ and
$\hat u^N$ using the same initial conditions and the same Brownian
motions. Since the controls agree on $[0,1-\varepsilon_N]$,
pathwise uniqueness of a solution to \eqref{eq:24} gives
\[
  \hat X_i^N(t)=\bar X_i^N(t),
  \qquad
  \hat\mu_t^N=\bar\mu_t^N,
  \qquad
  0\le t\le1-\varepsilon_N.
\]
Thus, since $F$ is $L_F$-Lipschitz,
\[
\begin{aligned}
\left|
    F(\bar\Gamma^{N,T_N})
    -
    F(\hat\Gamma^{N,T_N})
  \right|
  \le
  L_F
  d_{\mathrm{BL}}
  \bigl(
    \bar\Gamma^{N,T_N},
    \hat\Gamma^{N,T_N}
  \bigr)
  &\le
  L_F\int_{1-\varepsilon_N}^1
  d_{\mathrm{BL}}
  \bigl(
    \delta_{\bar\mu_t^N},
    \delta_{\hat\mu_t^N}
  \bigr)\,dt\le
  2L_F\varepsilon_N.
\end{aligned}
\]
Consequently,
\[
\begin{aligned}
  \E\left[
    \frac14\int_0^1\hat U_N(t)\,dt
    +
    F(\hat\Gamma^{N,T_N})
  \right]\le
  \E\left[
    \frac14\int_0^1U_N(t)\,dt
    +
    F(\bar\Gamma^{N,T_N})
  \right]
  +
  2L_F\varepsilon_N.
\end{aligned}
\]
Combining this inequality with \eqref{eq:31} proves
\eqref{eq:34}. 

Finally, $\hat U_N(t)=0$ for
$t\in(1-\varepsilon_N,1]$, and hence
\[
\begin{aligned}
  \E\int_0^1
    e^{-aT_N(1-t)}\hat U_N(t)\,dt
  =
  \E\int_0^{1-\varepsilon_N}
    e^{-aT_N(1-t)}U_N(t)\,dt\le
  e^{-aT_N\varepsilon_N}
  \E\int_0^1U_N(t)\,dt.
\end{aligned}
\]
The right-hand side converges to zero in view of \eqref{eq:32} and since
$T_N\varepsilon_N\to\infty$. This proves
\eqref{eq:36}.
\end{proof}

Henceforth, to keep notation simple, we will fix a sequence $\{\varepsilon_N\}$ as in the above lemma and will continue to denote the associated $\hat u^N, \hat X_i^N, \hat\mu_t^N, 
\hat\Gamma^{N,T_N}$ as $u^N, \bar X_i^N, \bar\mu_t^N, 
\bar\Gamma^{N,T_N}$.

 Define a sequence of random maps from $[0,1]$ to $\clp(\R^d\times \R^d)$ 
as 
$$
\eta^N_t(dx\, dy) \doteq \frac1N\sum_{i=1}^N
  \delta_{\bar X^N_i(t)}(dx) \delta_{u^N_i(t)}(dy) , \; t \in [0,1].$$
	Let,
	for $t \in [0,1]$, 
	$$\cls_t \doteq \cP(\R^d) \times \cP(\R^d \times \R^d) \times [0,t].$$
	We will write $\cls_1 \equiv \cls$.

  Define a sequence $\Theta^N$ of $\clp(\clp(\R^d) \times \clp(\R^d\times\R^d) \times [0,1]) = \clp(\cls)$ valued random variables as
\begin{equation}\label{eq:38}
  \Theta^N(d\gamma\, d\eta\, dt)
  \doteq
  \delta_{(\bar \mu^N_t, \eta^N_t)}(d\gamma\, d\eta) dt.
\end{equation}
Then
$$\frac1{4N}
     \sum_{i=1}^N
     \int_0^1 |u^N_i(t)|^2 dt = \frac1{4}\int_{\cls} 
     \int_{\R^d} |y|^2 [\eta]_2(dy) \Theta^N(d\gamma\, d\eta\, dt)$$
and we may write the bound in \eqref{eq:34} as
\begin{equation}\label{eq:39}
\begin{aligned}
  &-\frac1{NT_N}
  \log \E\exp\{-NT_NF(\Gamma^{N,T_N})\}
  +\frac1N +2L_F\varepsilon_N\\
  &\ge
  \E\Bigg[
     \frac14\int_{\cls} 
     \int_{\R^d} |y|^2 [\eta]_2(dy) \Theta^N(d\gamma\, d\eta\, dt)
     +
     F([\Theta^N]_1)
  \Bigg],
\end{aligned}
\end{equation}
where $[\Theta^N]_1$ denotes the first marginal of $\Theta^N$.

\subsection{Tightness of \texorpdfstring{$\Theta^N$}{ThetaN}}
\label{sec:4}

In this section we prove tightness of the sequence
  $\{\Theta^N\}_{N\ge 1}$ of $\cP(\cS)$ valued random variables.
From Lemma \ref{lem:1} and \eqref{eq:32} it follows that
\begin{equation}\label{eq:40}
\sup_{N\ge 1}\E\Bigg[
     \int_{\cls} \int_{\R^d} |y|^2 [\eta]_2(dy) \Theta^N(d\gamma\, d\eta\, dt)\Bigg] \le C_F < \infty.
\end{equation}

\begin{lemma}\label{lem:2}
There exists $C<\infty$, such that, for all $N\in \N$,
\begin{equation}\label{eq:42}
  \E\left[
    \int_0^1
       \int_{\R^d}|x|^2\,\bar\mu_t^N(dx)\,dt
  \right]
  =
  \E\left[
    \int_0^1
       \frac1N\sum_{i=1}^N |\bar X_i^N(t)|^2\,dt
  \right]
  \le C .
\end{equation}
Furthermore, $\{\Theta^N\}$ is tight sequence of $\cP(\cls)$ valued random variables.
\end{lemma}

\begin{proof}
Set
\begin{equation}\label{eq:43}
  S_N(t)\doteq \frac1N\sum_{i=1}^N |\bar X_i^N(t)|^2,
  \qquad
  U_N(t)\doteq \frac1N\sum_{i=1}^N |u_i^N(t)|^2 .
\end{equation}
 Since $W$ is even,
$\nabla W(0)=0$. Consider $(x_1, \ldots, x_N) \in (\R^d)^N$. Taking $x=x_i, y=0$, $z=x_j$, and $z'=0$ in 
Assumption~\ref{ass:3}, we get
\[
\begin{aligned}
  x_i\cdot \nabla V(x_i)
  +x_i\cdot \nabla W(x_i-x_j)
  &\ge
  \lambda |x_i|^2
  -C|x_j|^2
  +x_i\cdot \nabla V(0).
\end{aligned}
\]
Using Young's inequality on the last term and then averaging over $i,j$, and
using $\lambda>C$, there are constants $a>0$ and $c<\infty$ such that
\begin{equation}\label{eq:44}
  \frac1N\sum_{i=1}^N
  x_i\cdot
  \left[
     \nabla V(x_i)
     +\frac1N\sum_{j=1}^N \nabla W(x_i-x_j)
  \right]
  \ge
  a\,\frac1N\sum_{i=1}^N |x_i|^2
  -c .
\end{equation}

Applying It\^o's formula to $S_N(t)$ using
\eqref{eq:28} and taking expectations,
\[
\begin{aligned}
  \frac{d}{dt}\E S_N(t)
  &=
  -2T_N\E\left[
     \frac1N\sum_{i=1}^N
     \bar X_i^N(t)\cdot
     \left(
       \nabla V(\bar X_i^N(t))
       +\nabla W*\bar\mu_t^N(\bar X_i^N(t))
     \right)
  \right] \\
  &\quad
  +2T_N\E\left[
      \frac1N\sum_{i=1}^N
      \bar X_i^N(t)\cdot u_i^N(t)
  \right]
  +2dT_N .
\end{aligned}
\]
By \eqref{eq:44},
\[
  -2T_N\E\left[
     \frac1N\sum_{i=1}^N
     \bar X_i^N(t)\cdot
     \left(
       \nabla V(\bar X_i^N(t))
       +\nabla W*\bar\mu_t^N(\bar X_i^N(t))
     \right)
  \right]
  \le
  -2aT_N\E S_N(t)+2cT_N .
\]
Moreover, by Young's inequality,
\[
  2\frac1N\sum_{i=1}^N
      \bar X_i^N(t)\cdot u_i^N(t)
  \le
  a S_N(t)+a^{-1}U_N(t).
\]
Combining these estimates gives
\begin{equation}\label{eq:45}
  \frac{d}{dt}\E S_N(t)
  \le
  -aT_N\E S_N(t)
  +c_1 T_N\bigl(1+\E U_N(t)\bigr),
\end{equation}
for a constant $c_1<\infty$ independent of $N$.

Integrating \eqref{eq:45} over $[0,1]$, gives
\[
\begin{aligned}
  aT_N \int_0^1 \E S_N(t)\,dt
  &\le
   S_N(0)-\E S_N(1)
  +c_1 T_N
  +c_1 T_N\int_0^1 \E U_N(t)\,dt \\
  &\le
   S_N(0)
  +c_1 T_N
  +c_1 T_N\int_0^1 \E U_N(t)\,dt .
\end{aligned}
\]
The estimate in \eqref{eq:42} now follows on dividing by $aT_N$, and using Assumption \ref{ass:2}(a) and the control estimate \eqref{eq:40}.

We now prove tightness of $\{\Theta^N\}$. Combining
\eqref{eq:40} with \eqref{eq:42}, we get
\begin{equation}\label{eq:46}
  \sup_{N\ge1}
  \E\left[
     \int_{\cls}
     \left(
       \int_{\R^d}|x|^2\,\gamma(dx)
       +
       \int_{\R^d}(|z|^2+|y|^2)\,\eta(dz\, dy)
     \right)
     \Theta^N(d\gamma\,d\eta\,dt)
  \right]
  <\infty .
\end{equation}

For $R>0$, define
\[
  K_R^x
  \doteq
  \left\{
    \gamma\in\cP(\R^d):
    \int_{\R^d}|x|^2\,\gamma(dx)\le R
  \right\},
\]
and
\[
  K_R^u
  \doteq
  \left\{
    \eta\in\cP(\R^d\times \R^d):
    \int_{\R^d}(|z|^2+|y|^2)\,\eta(dz\,dy)\le R
  \right\}.
\]
Then
$K_R^x$ and $K_R^u$ are compact subsets of $\cP(\R^d)$ and $\cP(\R^d\times \R^d)$ respectively.  
The tightness of the laws of $\Theta^N$ as random variables with
values in $\cP(\cls)$
is now immediate on using \eqref{eq:46} by an application of Markov's inequality.
\end{proof}
\subsection{A lower bound on the limit cost.}
\label{sec:5}
In this section, in Proposition \ref{prop:1}, we establish  a key lower bound on the asymptotic cost.
Let $\{\Theta^N\}_{N\ge1}$ be as in the previous section.
Since, from Lemma \ref{lem:2}, this sequence is tight, it converges, in distribution, in the weak convergence topology, along a subsequence, to some limit point $\Theta$. We assume without loss of generality that the convergence holds along the full sequence and holds a.s. (by appealing to the Skorokhod representation theorem), in the weak convergence topology.

Recall that we denote the marginal distributions of $\tilde\Theta \in \cP(S)$ on the $i$-th (resp. $(i,j)$-th) coordinate as $[\tilde\Theta]_i$ (resp. $[\tilde\Theta]_{ij}$), for $1\le i<j\le 3$.

The following proposition is key.

\begin{proposition}\label{prop:1}
Let $\{\Theta^N\}$, $\Theta$ be as introduced above.
Then, 
\begin{equation}\label{eq:47}
 \liminf_{N\to \infty} \E\int_{\clp(\R^d\times \R^d)} \left(\int_{\R^d} |y|^2 \eta(dz\,dy)\right) 
 [\Theta^N]_{2}(d\eta) \ge 4\E \int_{\clp(\R^d)} \mathfrak i(\gamma) [\Theta]_{1}(d\gamma).
\end{equation}
\end{proposition}
The rest of this section is devoted to the proof of the above proposition.
Key steps in the proof are Lemma \ref{lem:6} and Corollary
\ref{cor:1}. The first gives an upper bound on a
mollified version of the quantity on the right side of
\eqref{eq:47}, whereas the second relates its limit as the
mollification parameter vanishes to the actual quantity on the right
side of \eqref{eq:47}.
The right side of the inequality in Lemma \ref{lem:6} involves a
term as on the left side of \eqref{eq:47} and several auxiliary
terms. These auxiliary terms vanish in the limit as a consequence of Lemmas 
\ref{lem:10}, \ref{lem:11}, and Proposition \ref{prop:2}, which then establishes the desired inequality in 
\eqref{eq:47}. Lemma  \ref{lem:6}, and the related Lemma \ref{lem:5},  are essentially consequences of 
It\^{o}'s formula applied to the free energy $\cF$ along the stochastic process $\bar \mu^N_t$. However, since the empirical measure $\bar\mu_t^N$ is
atomic and therefore has infinite entropy, we first  introduce a regularized free energy
$\cF_\delta(\bar\mu_t^N)$, by using convolutions  with 
Gaussian kernels, and  use this object for It\^{o} formula calculation. The key calculation that is at the heart of the above inequality appears in the proof of
Lemma \ref{lem:6}, in \eqref{eq:75} which can be viewed as a `completing the
squares' type orthogonality relation.

We now turn to the construction of  mollifications needed to define $\cF_{\delta}$.
Define
\[
  \kappa(x)
  \doteq
  (2\pi)^{-d/2}e^{-|x|^2/2},
  \qquad
  \kappa_\delta(x)
  \doteq
  \delta^{-d}\kappa(x/\delta), \; x \in \R^d, \delta >0.
\]
For $\mu\in\cP_2(\R^d)$, let
  $\rho_\delta^\mu=\kappa_\delta*\mu$.
  In defining the regularized free energy for a measure $\mu$, that may not have a density, we will replace $\rho^{\mu}$ with the mollified quantity $\rho_\delta^\mu$.
We record the following elementary lemma. The proof is omitted.
\begin{lemma}\label{lem:3}
Fix $\delta>0$ and $\mu\in\cP_2(\R^d)$.
The following statements hold.
\begin{enumerate}[label={(\alph*)}]
\item The function $\rho_\delta^\mu$ belongs to
$C^\infty(\R^d)$ and is strictly positive on $\R^d$.

\item For every $z\in\R^d$,
\begin{equation}\label{eq:48}
  \nabla\rho_\delta^\mu(z)
  =
  -\frac1{\delta^2}
  \int_{\R^d}(z-y)\kappa_\delta(z-y)\,\mu(dy).
\end{equation}

\item Defining
\[
  \pi_{\delta,z}^\mu(dy)
  \doteq
  \frac{\kappa_\delta(z-y)}
       {\rho_\delta^\mu(z)}\,\mu(dy),
\]
one has $\pi_{\delta,z}^\mu\in\cP(\R^d)$,
$\int_{\R^d}|y|\,\pi_{\delta,z}^\mu(dy)<\infty$,
and
\begin{equation}\label{eq:49}
  \nabla\log\rho_\delta^\mu(z)
  =
  -\frac1{\delta^2}
  \left(
    z-\int_{\R^d}y\,\pi_{\delta,z}^\mu(dy)
  \right).
\end{equation}

\item There exists a finite constant
$C_{\delta,\mu}$ such that
\begin{equation}\label{eq:50}
  \left|\nabla\log\rho_\delta^\mu(z)\right|
  \le
  C_{\delta,\mu}(1+|z|),
  \qquad z\in\R^d.
\end{equation}
Consequently, for every $x\in\R^d$,
\begin{equation}\label{eq:51}
  a_\delta(\mu,x)
  \doteq
  \int_{\R^d}
  \kappa_\delta(z-x)
  \nabla\log\rho_\delta^\mu(z)\,dz
\end{equation}
is well defined and finite.
\end{enumerate}
\end{lemma}

For $\mu\in\cP_2(\R^d)$, let
\begin{equation}\label{eq:52}
  b_\mu(x)
  \doteq
  \nabla V(x)+\nabla W*\mu(x).
\end{equation}
We now set
\begin{equation}\label{eq:53}
  G_\delta(\mu,x)
  \doteq
  a_\delta(\mu,x)+b_\mu(x)
\end{equation}
which can be viewed as the regularized analogue of the formal quantity
\[
  \nabla\frac{\delta\cF}{\delta\mu}(\mu)(x)
  =
  \nabla\log\rho^{\mu}(x)+b_\mu(x).
\]

Although $b_\mu$ is already globally Lipschitz, mollifying the measure
replaces the drift flux $b_\mu\mu$ by
$\kappa_\delta*(b_\mu\mu)$. It is therefore useful to express this
mollified flux relative to the density
$\rho_\delta^\mu=\kappa_\delta*\mu$. Define the effective mollified drift
\begin{equation}\label{eq:54}
  B_\delta[\mu](z)
  \doteq
  \frac{
    \displaystyle
    \int_{\R^d}
      \kappa_\delta(z-x)b_\mu(x)\,\mu(dx)}
  {\rho_\delta^\mu(z)}.
\end{equation}
Thus,
$\rho_\delta^\mu B_\delta[\mu]
  =
  \kappa_\delta*(b_\mu\mu)$.
Recall $\bar \mu^N_t$ from \eqref{eq:25} and define
$
  B_{\delta,t}^N
  =
  B_\delta[\bar\mu_t^N]$.
Also, define
\begin{equation}\label{eq:55}
  \rho_{\delta,t}^N(z)
  \doteq
  \rho_\delta^{\bar\mu_t^N}(z)
  =
  \frac1N\sum_{i=1}^N
  \kappa_\delta(z-\bar X_i^N(t)).
\end{equation}
Let
\begin{equation}\label{eq:56}
  j_{\delta,t}^{N,u}(z)
  \doteq
  \frac1N\sum_{i=1}^N
  \kappa_\delta(z-\bar X_i^N(t))
  u_i^N(t),\;\; 
U_{\delta,t}^N(z)
  \doteq
  \frac{j_{\delta,t}^{N,u}(z)}
       {\rho_{\delta,t}^N(z)}.
\end{equation}
We now apply It\^{o}'s formula to the semimartingale $\bar X^N(t)$ to obtain a semimartingale representation for $\rho_{\delta,t}^N(z)$.
\begin{lemma}\label{lem:4}
For $z \in \R^d$,
\begin{equation}\label{eq:57}
\begin{aligned}
  d\rho_{\delta,t}^N(z)
  &=
  T_N\nabla\cdot
  \left[
    \rho_{\delta,t}^N(z)
    \left(
      \nabla\log\rho_{\delta,t}^N(z)
      +
      B_{\delta,t}^N(z)
      -
      U_{\delta,t}^N(z)
    \right)
  \right]dt
  +
  d\mathfrak M_{\delta,t}^N(z),
\end{aligned}
\end{equation}
where
\begin{equation}\label{eq:58}
  \mathfrak M_{\delta,t}^N(z)
  \doteq
  -
  \frac{\sqrt{2T_N}}{N}
  \sum_{i=1}^N
  \int_0^t
  \nabla\kappa_\delta(z-\bar X_i^N(s))
  \cdot d\widetilde B_i(s).
\end{equation}
\end{lemma}
\begin{proof}
Fix $z\in\R^d$. Applying It\^o's formula to
$x\mapsto\kappa_\delta(z-x)$
along $\bar X_i^N$, and using
\[
  \nabla_x\kappa_\delta(z-x)
  =
  -\nabla_z\kappa_\delta(z-x),
  \qquad
  \Delta_x\kappa_\delta(z-x)
  =
  \Delta_z\kappa_\delta(z-x),
\]
gives
\[
\begin{aligned}
 d\kappa_\delta(z-\bar X_i^N(t))
 &=
 T_N b_{\bar\mu_t^N}(\bar X_i^N(t))
 \cdot
 \nabla_z\kappa_\delta(z-\bar X_i^N(t))\,dt-
 T_N u_i^N(t)\cdot
 \nabla_z\kappa_\delta(z-\bar X_i^N(t))\,dt\\
 &\quad+
 T_N\Delta_z
 \kappa_\delta(z-\bar X_i^N(t))\,dt-
 \sqrt{2T_N}
 \nabla_z\kappa_\delta(z-\bar X_i^N(t))
 \cdot d\widetilde B_i(t).
\end{aligned}
\]
Summing over $i$ and dividing by $N$, we obtain
\begin{equation}\label{eq:59}
\begin{aligned}
  d\rho_{\delta,t}^N(z)
  &=
  T_N\Delta\rho_{\delta,t}^N(z)\,dt
  +
  T_N\nabla\cdot j_{\delta,t}^{N,b}(z)\,dt-
  T_N\nabla\cdot j_{\delta,t}^{N,u}(z)\,dt
  +
  d\mathfrak M_{\delta,t}^N(z),
\end{aligned}
\end{equation}
where
\begin{equation*}
\begin{aligned}
  j_{\delta,t}^{N,b}(z)
  &\doteq
  \frac1N\sum_{i=1}^N
  \kappa_\delta(z-\bar X_i^N(t))
  b_{\bar\mu_t^N}(\bar X_i^N(t)).
  \end{aligned}
\end{equation*}
The result follows on noting that $\rho_{\delta,t}^N(z)B_{\delta,t}^N(z)
  =j_{\delta,t}^{N,b}(z)$ and $\rho_{\delta,t}^N(z)U_{\delta,t}^N(z)
  =j_{\delta,t}^{N,u}(z)$.

\end{proof}

We now introduce the regularized form of the free energy 
  $\cF(\mu)$ defined in \eqref{eq:11} that can be used in our It\^{o} formula calculations.
 For $\mu \in \cP_2(\R^d)$,
\begin{equation}\label{eq:60}
  \cF_\delta(\mu)
  \doteq
  \int_{\R^d}
     \rho_\delta^\mu(z)\log\rho_\delta^\mu(z)\,dz
  +
  \int_{\R^d}V(x)\,\mu(dx)
  +
  \frac12
  \iint_{\R^d\times\R^d}
     W(x-x')\,\mu(dx)\mu(dx').
\end{equation}
In preparation for a semimartingale evolution equation for the regularized free energy along the process $\bar \mu^N_t$ (namely, for $\cF_\delta(\bar\mu_t^N)$), define
for $\mu \in \clp_2(\R^d)$, 
\begin{equation}\label{eq:61}
  C_\delta(\mu)
  \doteq
  \int_{\R^d}
  \nabla\log\rho_\delta^\mu(z)
  \cdot B_\delta[\mu](z)
  \rho_\delta^\mu(z)\,dz,\;\;
  \cV_\delta(\mu)
  \doteq
  \int_{\R^d}
  \left|
    \nabla\log\rho_\delta^\mu(z)\right|^2
  \rho_\delta^\mu(z)\,dz.
\end{equation}
Also, for $x = (x_1, \ldots,x_N) \in (\R^d)^N$, define $\mu^N_x \doteq \frac1N\sum_{i=1}^N\delta_{x_i}$ and
\[
  \rho_\delta^N(z,x)
  \doteq
  \rho_\delta^{\mu_x^N}(z)
  =
  \frac1N\sum_{j=1}^N\kappa_\delta(z-x_j).
\]
For $\mu\in\clp_2(\R^d)$, let
\[
A_N(\mu)
\doteq
\int_{\R^d}\Delta V(y)\,\mu(dy)
+
\iint_{\R^d\times\R^d}\Delta W(y-z)\,\mu(dy)\mu(dz)
-\frac{\Delta W(0)}{N}.
\]
Thus, for $x=(x_1,\ldots,x_N)\in(\R^d)^N$,
\[
A_N(\mu_x)
=
\frac{1}{N}\sum_{i=1}^N\Delta V(x_i)
+
\frac{1}{N^2}\sum_{\substack{i,j=1\\j\ne i}}^N
\Delta W(x_i-x_j).
\]

Let
			 \begin{equation}\label{eq:62}
			   D_{N,\delta}(x)
		   \doteq
		   \frac1{N^2}
		   \sum_{i=1}^N
		   \int_{\R^d}
		   \frac{
		     |\nabla\kappa_\delta(z-x_i)|^2}
		   {\rho_\delta^N(z,x)}\,dz,
\end{equation}
and 		 
\begin{equation}\label{eq:63}
  \widetilde R_{N,\delta}(\mu)
  \doteq
  A_N(\mu)+C_\delta(\mu),
  \qquad \mu\in\clp_2(\R^d).
\end{equation}

By using the semimartingale representation from the last lemma we now have the following result.
\begin{lemma}\label{lem:5} For $N \in \NN$,
\begin{multline}\label{eq:64}
\frac1{T_N}
 \left[
   \cF_\delta(\bar\mu_1^N)
   -
   \cF_\delta(\bar\mu_0^N)
 \right]\\
 =
 -
 \int_0^1
 \left[
   \cV_\delta(\bar\mu_t^N)
   +
   2C_\delta(\bar\mu_t^N)
   +
   \int_{\R^d}
   |b_{\bar\mu_t^N}(x)|^2
   \bar\mu_t^N(dx)
 \right]dt\\
+
 \int_0^1
 \left[
   \int_{\R^d}
   \nabla\log\rho_{\delta,t}^N(z)
   \cdot U_{\delta,t}^N(z)
   \rho_{\delta,t}^N(z)\,dz
   +
   \frac1N\sum_{i=1}^N
   b_{\bar\mu_t^N}(\bar X_i^N(t))
   \cdot u_i^N(t)
 \right]dt\\
 +
 \int_0^1D_{N,\delta}(\bar X^N(t))\,dt
 +
 \int_0^1
 \widetilde R_{N,\delta}
 (\bar\mu_t^N)\,dt +
 \frac1{T_N}\widetilde M_{N,\delta}(1),
\end{multline}	
	where
	$$\widetilde M_{N,\delta}(t) \doteq \frac{\sqrt{2T_N}}{N}
 \sum_{i=1}^N
 \int_0^t
[ a_\delta(\bar\mu_s^N,\bar X_i^N(s)) + b_{\bar\mu_s^N}(\bar X_i^N(s))]
 \cdot d\widetilde B_i(s).$$
\end{lemma}
\begin{proof}
	For $\mu \in \clp_2(\R^d)$, let $
  \cH_\delta(\mu)
  \doteq
  \int_{\R^d}
  \rho_\delta^\mu(z)\log\rho_\delta^\mu(z)\,dz.
$
Then, from Lemma \ref{lem:4}, an application of It\^o's formula, and integrating by parts, we have 
\begin{multline}\label{eq:65}
 \frac1{T_N}
 \left[
   \cH_\delta(\bar\mu_1^N)
   -
   \cH_\delta(\bar\mu_0^N)
 \right]
 =
 -
 \int_0^1
 \int_{\R^d}
 |\nabla\log\rho_{\delta,t}^N(z)|^2
 \rho_{\delta,t}^N(z)\,dz\,dt\\
-
 \int_0^1
 \int_{\R^d}
 \nabla\log\rho_{\delta,t}^N(z)
 \cdot B_{\delta,t}^N(z)
 \rho_{\delta,t}^N(z)\,dz\,dt+
 \int_0^1
 \int_{\R^d}
 \nabla\log\rho_{\delta,t}^N(z)
 \cdot U_{\delta,t}^N(z)
 \rho_{\delta,t}^N(z)\,dz\,dt\\
+
 \int_0^1D_{N,\delta}(\bar X^N(t))\,dt
 +
 \frac1{T_N}M_{N,\delta}^{\mathrm{ent}}(1),
\end{multline}
where 
\begin{equation}\label{eq:66}
\begin{aligned}
 M_{N,\delta}^{\mathrm{ent}}(t)
 &=
 \frac{\sqrt{2T_N}}{N}
 \sum_{i=1}^N
 \int_0^t
 a_\delta(\bar\mu_s^N,\bar X_i^N(s))
 \cdot d\widetilde B_i(s).
\end{aligned}
\end{equation}
Next,
define, for $\mu \in \clp_2(\R^d)$,
\[
  \cE(\mu)
  \doteq
  \int_{\R^d}V(x)\,\mu(dx)
  +
  \frac12
  \iint_{\R^d\times\R^d}
  W(x-y)\,\mu(dx)\mu(dy), \; \mu \in \clp_2(\R^d).
\]
Note that, for $x \in (\R^d)^N$,
\[
  \nabla_{x_i}\cE(\mu_x^N)
  =
  \frac1N b_{\mu^N}(x_i),
\]
Applying It\^o's formula to $x \mapsto \cE(\mu_x^N)$ gives
\begin{multline}\label{eq:67}
 \frac1{T_N}
 \left[
   \cE(\bar\mu_1^N)
   -
   \cE(\bar\mu_0^N)
 \right]=
 -
 \int_0^1
 \int_{\R^d}
 |b_{\bar\mu_t^N}(x)|^2
 \bar\mu_t^N(dx)\,dt\\
 \quad +
 \int_0^1
 \frac1N\sum_{i=1}^N
 b_{\bar\mu_t^N}(\bar X_i^N(t))
 \cdot u_i^N(t)\,dt +
 \int_0^1A_N(\bar\mu_t^N)\,dt
 +
 \frac1{T_N}M_N^{\mathrm{pot}}(1),
\end{multline}
where
\[
  M_N^{\mathrm{pot}}(t)
  =
  \frac{\sqrt{2T_N}}{N}
  \sum_{i=1}^N
  \int_0^t
  b_{\bar\mu_s^N}(\bar X_i^N(s))
  \cdot d\widetilde B_i(s).
	\]
Adding \eqref{eq:65} and
\eqref{eq:67}, and using the identities
$\cF_{\delta}(\mu) = \cH_{\delta}(\mu) + \cE(\mu)$, 
$A_N(\mu)=-C_\delta(\mu)+\widetilde R_{N,\delta}(\mu)$,
completes the proof of the lemma.
\end{proof}

We now introduce a regularized form of $\mathfrak i(\mu)$ defined as
\begin{equation}\label{eq:68}
\begin{aligned}
  \cJ_\delta(\mu)
  &\doteq
  \int_{\R^d}
  \left|
    \nabla\log\rho_\delta^\mu(z)
    +
    B_\delta[\mu](z)
  \right|^2
  \rho_\delta^\mu(z)\,dz, \; \mu \in \clp_2(\R^d).
\end{aligned}
\end{equation}
Since $\nabla V$ and $\nabla W$ are globally Lipschitz, there is
$L_b<\infty$, such that
\begin{equation}\label{eq:69}
  |b_\mu(x)-b_\mu(x')|
  \le L_b|x-x'|,
  \qquad x,x'\in\R^d, \; \mu \in \clp_2(\R^d).
\end{equation}
Recall the random measures $\Theta^N$ from \eqref{eq:38}. For the regularized energy dissipation functional $\cJ_\delta(\mu)$, we now have the following result.
\begin{lemma}\label{lem:6}
For $N\in \NN$,
\begin{multline}\label{eq:70}
 \frac12
 \int_{\cP(\R^d)}
 \cJ_\delta(\gamma)\,[\Theta^N]_1(d\gamma)
 \le
 \frac12
 \int_{\cP(\R^d\times\R^d)}
 \int_{\R^d\times\R^d}
 |y|^2\,\eta(dx\,dy)\,[\Theta^N]_2(d\eta)\\
 +
 L_b\sqrt d\,\delta
 \left[
 \int_{\cP(\R^d\times\R^d)}
 \int_{\R^d\times\R^d}
 |y|^2\,\eta(dx\,dy)\,[\Theta^N]_2(d\eta)
 \right]^{1/2}\\
 -
 \frac1{T_N}
 \left[
   \cF_\delta(\bar\mu_1^N)
   -
   \cF_\delta(\bar\mu_0^N)
 \right] +
 \int_0^1D_{N,\delta}(\bar X^N(t))\,dt\\
 +
 \int_0^1
 \widetilde R_{N,\delta}
 (\bar\mu_t^N)\,dt
 +
 \frac1{T_N}\widetilde M_{N,\delta}(1).
\end{multline}
	
\end{lemma}
\begin{proof}
We will use the identity from \eqref{eq:64}.
For $x,u \in (\R^d)^N$, define
\begin{equation}\label{eq:73}
\begin{aligned}
  \cR_{\delta}^{u}(x,u)
  &\doteq
  \frac1N\sum_{i=1}^N b_{\mu^N_x}(x_i)\cdot u_i
 -
  \int_{\R^d}
  B_\delta[\mu^N_x](z)\cdot U_\delta(x,z)
  \rho_\delta^{\mu^N_x}(z)\,dz,
\end{aligned}
\end{equation}
where
\[
  U_\delta(x,z)
  =
  \frac{
    \frac1N\sum_{i=1}^N
      \kappa_\delta(z-x_i)u_i}
  {\rho_\delta^{\mu^N_x}(z)}.
\]
Let
\[
\begin{aligned}
 \Psi(x,u) \doteq &\int_{\R^d}
 \left[
   \nabla\log\rho_\delta^{\mu^N_x}(z)
   +
   B_\delta[\mu^N_x](z)
 \right]
 \cdot U_\delta(x,z)
 \rho_\delta^{\mu^N_x}(z)\,dz
 +
 \cR_\delta^u(x,u).
\end{aligned}
\]
Then the control term in
\eqref{eq:64} (namely, the sum of the third and fourth terms) can be written as $\int_0^1 \Psi(\bar X^N(t),u^N(t)) dt$.
By Young's inequality and Jensen's inequality,
\begin{equation}\label{eq:74}
\begin{aligned}
 \int
 \left(
   \nabla\log\rho_\delta^{\mu^N_x}(z)+B_\delta[\mu^N_x](z)
 \right)\cdot U_\delta(x,z)\,\rho_\delta^{\mu^N_x}(z)\,dz &\le
 \frac12\cJ_\delta(\mu^N_x)
 +
 \frac12\int|U_\delta(x,z)|^2\rho_\delta^{\mu^N_x}(z) dz\\
 & \le \frac12\cJ_\delta(\mu^N_x)
 +\frac1{2N}\sum_{i=1}^N|u_i|^2.
\end{aligned}
\end{equation}
Thus
\begin{equation}
\begin{aligned}
\int_0^1 \Psi(\bar X^N(t),u^N(t))\,dt
&\le \frac12 \int_0^1 \cJ_\delta(\bar\mu^N_t)\,dt
 +\frac1{2N}\sum_{i=1}^N\int_0^1 |u^N_i(t)|^2\,dt\\
&\quad + \int_0^1 \cR_\delta^u(\bar X^N(t),u^N(t))\,dt.
\end{aligned}
\end{equation}
We now estimate the first integral on the right side of \eqref{eq:64}.
From Jensen's inequality,
\begin{equation}\label{eq:71}
\int|B_\delta[\mu](z)|^2\rho_\delta^\mu(z) dz
\le
\int|b_\mu(x)|^2\,\mu(dx).
\end{equation}
Consequently, recalling the definition of $\cV_{\delta}$ and $C_{\delta}$,
\begin{equation}\label{eq:72}
\begin{aligned}
  \cV_\delta(\mu)
  +
  2C_\delta(\mu)
  +
  \int|b_\mu(x)|^2\,\mu(dx)\ge 
  \cV_\delta(\mu)
  +
  2C_\delta(\mu)
  +
  \int|B_\delta[\mu](z)|^2\rho_\delta^\mu(z) dz &=
  \cJ_\delta(\mu).
\end{aligned}
\end{equation}

Combining
Lemma \ref{lem:5}, \eqref{eq:72},
\eqref{eq:74}, we now obtain
\begin{multline}\label{eq:75}
\frac12\int_0^1
 \cJ_\delta(\bar\mu_t^N)\,dt
 \le
 \frac1{2N}
 \sum_{i=1}^N\int_0^1|u_i^N(t)|^2\,dt
 -
 \frac1{T_N}
 \left[
   \cF_\delta(\bar\mu_1^N)
   -
   \cF_\delta(\bar\mu_0^N)
 \right]\\
 \quad+
 \int_0^1D_{N,\delta}(\bar X^N(t))\,dt
 +
 \int_0^1
 \widetilde R_{N,\delta}
 (\bar\mu_t^N)\,dt\\
 \quad+
 \int_0^1
 \cR_{\delta}^{u}
 (\bar X^N(t),u^N(t))\,dt
 +
 \frac1{T_N}\widetilde M_{N,\delta}(1).
\end{multline}
Finally, we now estimate the contribution from $\cR_{\delta}^{u}$.

The following observation will be used frequently. Let $\mu\in\cP_2(\R^d)$ and let
$X\sim\mu$, $G\sim N(0,I_d)$ which is independent of $X$, and
$Z=X+\delta G$. Then
$B_\delta[\mu](Z)=\E[b_\mu(X)\mid Z]$ and consequently, properties of 
conditional expectation and \eqref{eq:69} give
\begin{equation}\label{eq:76}
 \int_{\R^d}
 |B_\delta[\mu](z)-b_\mu(z)|^2
 \rho_\delta^\mu(z)\,dz
 \le
 \E|b_\mu(X)-b_\mu(Z)|^2
 \le L_b^2d\delta^2.
\end{equation}
Now let $I$ be a 
uniform random variable on $\{1,\ldots,N\}$. Let $X=x_I$ and $U=u_I$, where $x = (x_1, \ldots , x_N)$ and $u = (u_1, \ldots , u_N)$ are in $(\R^d)^N$, and let
$Z=X+\delta G$, where $G\sim N(0,I_d)$ is independent of $I$. Then, from the above observation,
$B_\delta[\mu_x^N](Z)=\E[b_{\mu_x^N}(X)\mid Z]$ and, similarly,
$U_\delta(x,Z)=\E[U\mid Z]$. Consequently,
\[
  \cR_\delta^u(x,u)
  =
  \E\left[
    \bigl(b_{\mu_x^N}(X)-B_\delta[\mu_x^N](Z)\bigr)
    \cdot\bigl(U-U_\delta(x,Z)\bigr)
  \right].
\]
Furthermore, \eqref{eq:76} gives 
$\E|b_{\mu_x^N}(X)-B_\delta[\mu_x^N](Z)|^2
  \le L_b^2d\delta^2$,
while $\E|U-U_\delta(x,Z)|^2\le\E|U|^2$. Cauchy--Schwarz inequality therefore
gives
\begin{equation}\label{eq:77}
 |\cR_\delta^u(x,u)|
 \le
 L_b\sqrt d\,\delta
 \left(
   \frac1N\sum_{i=1}^N|u_i|^2
 \right)^{1/2}.
\end{equation}

Using \eqref{eq:77} in \eqref{eq:75}, 
Rewriting in terms of $\Theta^N$ and using the Cauchy--Schwarz inequality again gives \eqref{eq:70} and completes the proof of the lemma.
\end{proof}

We extend $\cJ_\delta$ to all of $\cP(\R^d)$ by setting
$\cJ_\delta(\mu)=+\infty$ when $\mu\notin\cP_2(\R^d)$. For
$\varphi\in C_c^\infty(\R^d;\R^d)$ and $\nu\in\cP_2(\R^d)$, set
\begin{equation}\label{eq:79}
\begin{aligned}
  \cH_{\delta,\varphi}(\nu)
  &\doteq
  -2\int_{\R^d}\nabla\cdot\varphi(z)\rho_\delta^\nu(z)\,dz
  +2\int_{\R^d}\varphi(z)\cdot B_\delta[\nu](z)
       \rho_\delta^\nu(z)\,dz\\
  &\qquad-
  \int_{\R^d}|\varphi(z)|^2\rho_\delta^\nu(z)\,dz.
\end{aligned}
\end{equation}
Hilbert-space duality, together with
$\nabla\rho_\delta^\nu
=\rho_\delta^\nu\nabla\log\rho_\delta^\nu$, gives
\begin{equation}\label{eq:80}
  \cJ_\delta(\nu)
  =
  \sup_{\varphi\in C_c^\infty(\R^d;\R^d)}
  \cH_{\delta,\varphi}(\nu).
\end{equation}

\begin{lemma}\label{lem:7}
For every $\delta>0$, the function
$\cJ_\delta:\cP(\R^d)\to[0,\infty]$ is lower semicontinuous for the
weak topology. Consequently, the map
$\Theta\mapsto
  \int_{\cP(\R^d)}\cJ_\delta(\gamma)\,[\Theta]_1(d\gamma)$
is lower semicontinuous on $\cP(\cls)$.
\end{lemma}

\begin{proof}
Let $\gamma_n\Rightarrow\gamma$ and suppose, after passing to a
subsequence, that $\sup_n\cJ_\delta(\gamma_n)<\infty$. Let
$X_n\sim\gamma_n$, let $G\sim N(0,I_d)$ be independent of $X_n$, and
set $Z_n=X_n+\delta G$. Write
$s_n=\nabla\log\rho_\delta^{\gamma_n}$ and
$B_n=B_\delta[\gamma_n]$. The Tweedie--Robbins--Esposito identity
\cite[(3)]{DytsoPoorShamai2023}, applied with Gaussian noise covariance
$\delta^2I_d$, gives
$\E[X_n\mid Z_n=z]=z+\delta^2s_n(z)$, and hence
$s_n(Z_n)=-\delta^{-1}\E[G\mid Z_n]$. Therefore, Jensen's inequality gives
\[
  \int|s_n(z)|^2\rho_\delta^{\gamma_n}(z) dz\le \frac d{\delta^2}.
\]
Since $B_n=(s_n+B_n)-s_n$,
we have from our assumption, that $\sup_n\int|B_n(z)|^2 \rho_\delta^{\gamma_n}(z) dz<\infty$.
Furthermore, from \eqref{eq:76}, we have that
$\E|b_{\gamma_n}(X_n)-B_n(Z_n)|^2
  \le L_b^2d\delta^2$.
Combining the above facts, it follows that
\begin{equation}\label{eq:81}
\sup_n\int|b_{\gamma_n}|^2\,d\gamma_n<\infty.
\end{equation} 
We will use the following
coercivity estimate. For some $c_0, C_0>0$, and all $\nu \in \cP_2(\R^d)$,
\[
  \int x\cdot b_\nu(x)\,\nu(dx)
  \ge c_0\int|x|^2\,\nu(dx)-C_0.
\]
Indeed, this follows by taking $y=z'=0$ in Assumption~\ref{ass:3}, integrating $(x,z)$ against $\nu\otimes\nu$, using $\nabla W(0)=0$, and applying Young's inequality.
Applying this estimate and another use of Young's inequality, together with the uniform bound in \eqref{eq:81}, we have that
$\sup_n\int|x|^2\,\gamma_n(dx)<\infty$.
In particular, $\gamma\in\cP_2(\R^d)$.

Fix $\varphi\in C_c^\infty(\R^d;\R^d)$ and set
$\psi_\delta=\kappa_\delta*\varphi$. Since
$B_\delta[\nu]\rho_\delta^\nu=\kappa_\delta*(b_\nu\nu)$, Fubini's
theorem gives
\begin{equation}\label{eq:82}
  \int\varphi(z)\cdot B_\delta[\nu](z)\rho_\delta^\nu(z)\,dz
  =
  \int\psi_\delta(x)\cdot b_\nu(x)\,\nu(dx).
\end{equation}
The function $\psi_\delta$ is bounded, continuous, and has Gaussian
decay. Therefore weak convergence, the uniform moment bound
\eqref{eq:81}, and a standard truncation
argument using the linear growth and global Lipschitz continuity of
$\nabla V$ and $\nabla W$ yield
\[
  \int\psi_\delta\cdot b_{\gamma_n}\,d\gamma_n
  \to
  \int\psi_\delta\cdot b_\gamma\,d\gamma.
\]
Furthermore, the first and the third terms in \eqref{eq:79} (with $\nu$ replaced with $\gamma_n$) converge
because $\gamma_n*\kappa_\delta\Rightarrow
\gamma*\kappa_\delta$. Thus
$\cH_{\delta,\varphi}(\gamma_n)\to
\cH_{\delta,\varphi}(\gamma)$. Taking the supremum over $\varphi$,
and recalling \eqref{eq:80}, proves
$\cJ_\delta(\gamma)
  \le\liminf_{n\to\infty}\cJ_\delta(\gamma_n)$.
The assertion for $\Theta$ follows from Fatou's lemma.
\end{proof}

We next pass to the limit as $\delta\downarrow0$.
We begin with a convenient variational formula for the function $\mathfrak i(\cdot)$ in \eqref{eq:5}. 
For $\gamma\in\cP_2(\R^d)$, define
\begin{equation}\label{eq:83}
\begin{aligned}
  \cD(\gamma)
  \doteq
  \sup_{\varphi\in C_c^\infty(\R^d;\R^d)}
  \left\{
    -2\int_{\R^d}
      \nabla\cdot\varphi(x)\,\gamma(dx)
    +
    2\int_{\R^d}
      \varphi(x)\cdot b_\gamma(x)\,\gamma(dx)
    -
    \int_{\R^d}
      |\varphi(x)|^2\,\gamma(dx)
  \right\}.
\end{aligned}
\end{equation}
For $\gamma\notin\cP_2(\R^d)$, set
$\cD(\gamma)=+\infty$.

The following standard characterization identifies $\cD$ with the
squared $L^2(\gamma)$-norm of the Wasserstein gradient of $\cF$.
Proof is in the appendix.

\begin{lemma}\label{lem:8}
Let $\gamma\in\cP_2(\R^d)$. Then $\cD(\gamma)<\infty$ if and only if
$\gamma(dx)=\rho(x)\,dx$ for some density $\rho$ for which the weak logarithmic gradient exists.
In this case,
\begin{equation}\label{eq:84}
  \cD(\gamma)
  =
  \int_{\R^d}
  |\nabla\log\rho(x)+b_\gamma(x)|^2\,\gamma(dx)
  =
  \left\|
    \nabla\frac{\delta\cF}{\delta\mu}(\gamma)
  \right\|_{L^2(\gamma)}^2.
\end{equation}
Consequently, 
  $\cD(\gamma)
  = 4 \mathfrak i(\gamma)$
  for every $\gamma\in\cP(\R^d)$.
\end{lemma}
The next lemma relates the regularized energy dissipation with the original dissipation functional $\mathfrak i$.
\begin{lemma}\label{lem:9}
For every $\gamma\in\cP(\R^d)$,
$
  \cD(\gamma)
  \le
  \liminf_{\delta\downarrow0}\cJ_\delta(\gamma)$.
\end{lemma}

\begin{proof}
The assertion is immediate when $\gamma\notin\cP_2(\R^d)$. Suppose
that $\gamma\in\cP_2(\R^d)$ and fix
$\varphi\in C_c^\infty(\R^d;\R^d)$. By
\eqref{eq:80} and
\eqref{eq:82},
\[
\begin{aligned}
  \cJ_\delta(\gamma)
  \ge{}&
  -2\int\nabla\cdot\varphi(z)\rho_\delta^\gamma(z)\,dz
  +2\int(\kappa_\delta*\varphi)(x)\cdot b_\gamma(x)\,\gamma(dx)
  -\int|\varphi(z)|^2\rho_\delta^\gamma(z)\,dz.
\end{aligned}
\]
The first and third terms converge by weak convergence of
$\rho_\delta^\gamma(z)\,dz$ to $\gamma$. The middle term converges by
dominated convergence, since $b_\gamma\in L^1(\gamma)$ and
$\kappa_\delta*\varphi\to\varphi$ pointwise and $\kappa_\delta*\varphi$ is bounded, uniformly in $\delta$.
Taking the supremum over $\varphi$ and recalling the definition of $\cD$ gives
the desired inequality.
\end{proof}

The following is an immediate consequence of the above result and Fatou's lemma.
\begin{corollary}\label{cor:1}
For every $\Theta\in\cP(\cls)$,
\begin{equation}\label{eq:85}
  \int_{\cP(\R^d)}\cD(\gamma)\,[\Theta]_1(d\gamma)
  \le
  \liminf_{\delta\downarrow0}
  \int_{\cP(\R^d)}\cJ_\delta(\gamma)\,[\Theta]_1(d\gamma).
\end{equation}
\end{corollary}

We now control the last four terms
on the right-hand side of \eqref{eq:70}. We first treat the
boundary and martingale terms.
\begin{lemma}
\label{lem:10}
For each $\delta>0$, 
\[
\lim_{N\to\infty}
\E\left(\left|
\frac1{T_N}
\left[
  \cF_\delta(\bar\mu_1^N)
  -
  \cF_\delta(\bar\mu_0^N)
\right]\right|
+\left|
\frac1{T_N}\widetilde M_{N,\delta}(1)
\right|\right)
=0.
\]
\end{lemma}

\begin{proof}
Since $\nabla V$ and
$\nabla W$ are globally Lipschitz, there is a constant $C<\infty$
such that
$
  |V(x)|+|W(x)|\le C(1+|x|^2)$, for all $x \in \R^d$.
Consequently, for some $C<\infty$, and all $\mu \in \clp_2(\R^d)$,
\[
\begin{aligned}
  \left|\int_{\R^d}V(x)\,\mu(dx)\right| + \left|
  \frac12\iint_{\R^d\times\R^d}
     W(x-y)\,\mu(dx)\mu(dy)
  \right|
  \le
  C\left(1+\int_{\R^d}|x|^2\,\mu(dx)\right).
\end{aligned}
\]

We next bound the entropy of $\rho_\delta^\mu=\kappa_\delta*\mu$.
By convexity of $r\mapsto r\log r$,
\[
\begin{aligned}
 \int_{\R^d}\rho_\delta^\mu(z)
       \log\rho_\delta^\mu(z)\,dz
 &\le
 \int_{\R^d}\int_{\R^d}
   \kappa_\delta(z-x)\log\kappa_\delta(z-x)
   \,\mu(dx)\,dz\\
 &=
 \int_{\R^d}\kappa_\delta(z)\log\kappa_\delta(z)\,dz
 \doteq C_\delta^{+}.
\end{aligned}
\]
On the other hand, using the inequality $\int \rho_\delta^\mu(z) \log (\rho_\delta^\mu(z)/\kappa_\delta(z)) dz \ge 0$, we see that 
there is a constant $C^-_\delta<\infty$ such that
\[
  -\int_{\R^d}\rho_\delta^\mu(z)
       \log\rho_\delta^\mu(z)\,dz
  \le
  C^-_\delta \left(1+
  \int_{\R^d}|x|^2\,\mu(dx)
  \right).
\]
Combining the last three estimates, we obtain, for some $C_{\delta}>0$,
\begin{equation}\label{eq:86}
  |\cF_\delta(\mu)|
  \le
  C_\delta
  \left(
    1+\int_{\R^d}|x|^2\,\mu(dx)
  \right),
  \qquad \mu\in\cP_2(\R^d).
\end{equation}

We now estimate the second moment under $\bar \mu^N_t$. 
Recall $S_N(t)$ and $U_N(t)$ from \eqref{eq:43}.
With $a>0$ as in \eqref{eq:44}, we have from
\eqref{eq:45},
\[
\begin{aligned}
  \E S_N(1)
  &\le
  e^{-aT_N} S_N(0)
  +
  c_1T_N\int_0^1e^{-aT_N(1-t)}\,dt+
  c_1T_N\int_0^1
     e^{-aT_N(1-t)}\E U_N(t)\,dt\\
  &\le
  e^{-aT_N} S_N(0)
  +
  \frac{c_1}{a}
  +
  c_1T_N\int_0^1
     e^{-aT_N(1-t)}\E U_N(t)\,dt.
\end{aligned}
\]
Using the uniform bound on the initial data in Assumption \ref{ass:2}(a), and \eqref{eq:36} in Lemma \ref{lem:1} (and recalling that we denote $\hat u^N$ as $u^N$), it follows that
  $\frac1{T_N}\E S_N(1)\to0$. Using this in \eqref{eq:86}, and using the uniform bound in Assumption \ref{ass:2}(a) once more, gives
\begin{equation}\label{eq:87}
  \frac1{T_N}
  \E\left|
    \cF_\delta(\bar\mu_1^N)
    -
    \cF_\delta(\bar\mu_0^N)
  \right|
  \to0.
\end{equation}
We now consider the martingale term. Fix $\mu \in \clp_2(\R^d)$. Let $X,G,Z$ be as introduced above \eqref{eq:76}. 
Then, using Jensen's inequality,
\begin{align*}
\int_{\R^d}|a_\delta(\mu,x)|^2\,\mu(dx)
  \le \cV_\delta(\mu) = \frac1{\delta^2}
  \E\left|\E[G\mid Z]\right|^2 \le \frac d{\delta^2},
\end{align*}
where the middle equality uses the fact that $\nabla\log\rho_\delta^\mu(Z)
  =
  -\frac1\delta\E[G\mid Z]$ (cf. proof of Lemma \ref{lem:7}).

Since $\nabla V$ and $\nabla W$ are globally Lipschitz, it follows that, for some $C>0$,
\[
  \int_{\R^d}|b_\mu(x)|^2\,\mu(dx)
  \le
  C\left(
    1+\int_{\R^d}|x|^2\,\mu(dx)
  \right), \; \mu \in \clp_2(\R^d).
\]
Combining the above estimates, and recalling the definition of $G_{\delta}$ from \eqref{eq:53}, we obtain, for some $C_{\delta}>0$,
\begin{equation}\label{eq:88}
  \|G_\delta(\mu,\cdot)\|_{L^2(\mu)}^2
  \le
  C_\delta
  \left(
    1+\int_{\R^d}|x|^2\,\mu(dx)
  \right), \; \mu \in \clp_2(\R^d).
\end{equation}
Lemma~\ref{lem:2} now gives
\begin{equation}\label{eq:89}
  \sup_{N\ge1}
  \E\int_0^1
     \frac1N\sum_{i=1}^N
     |G_\delta(\bar\mu_t^N,\bar X_i^N(t))|^2\,dt
  <\infty.
\end{equation}
Finally,
\[
\begin{aligned}
  \frac1{T_N}\E|\widetilde M_{N,\delta}(1)| \le \frac1{T_N}\left[\E\left\langle \widetilde M_{N,\delta}\right\rangle_1\right]^{1/2}
  &=
  \frac{\sqrt{2}}{N\sqrt{T_N}}
  \left(\sum_{i=1}^N
  \E\int_0^1
  |G_\delta(\bar\mu_t^N,\bar X_i^N(t))|^2\,dt\right)^{1/2} \le \frac{C_\delta}{\sqrt{NT_N}}.
\end{aligned}
\]
Sending $N\to \infty$, and combining this with
\eqref{eq:87} completes the proof. 
\end{proof}

We next treat the fourth term on the right-hand side of
\eqref{eq:70}.
\begin{lemma}
\label{lem:11}
For each $\delta>0$,
$\E\int_0^1
  D_{N,\delta}(\bar X^N(t))\,dt
  \to0$.
\end{lemma}

\begin{proof}
We first prove that, for every fixed $R>0$, there exists
$C_{\delta,R}<\infty$ such that, for every $N\in\mathbb{N}$ and every
$x\in(\R^d)^N$,
\begin{equation}\label{eq:90}
  D_{N,\delta}(x)
  \le
  \frac{C_{\delta,R}}{N}
  +
  \frac d{\delta^2}\mu_{x}^N(B_R^c),
\end{equation}
where $B_R$ is a ball of radius $R$ in $\R^d$.
Let
$I_R
  \doteq
  \{i:x_i\in B_R\}$ and 
 $m_R\doteq |I_R|$.

For $i\notin I_R$, using the inequality
$\rho_\delta^N(z,x)
  \ge
  \frac1N\kappa_\delta(z-x_i)$, we have
\begin{align}
  \frac1{N^2}
  \int_{\R^d}
  \frac{
    |\nabla\kappa_\delta(z-x_i)|^2
  }{
    \rho_\delta^N(z,x)
  }\,dz
  &\le
  \frac1N
  \int_{\R^d}
  \frac{
    |\nabla\kappa_\delta(z-x_i)|^2
  }{
    \kappa_\delta(z-x_i)
  }\,dz=
  \frac d{N\delta^2}.
\end{align}
Summing over $i\notin I_R$ gives
\begin{equation}\label{eq:91}
  \frac1{N^2}
  \sum_{i\notin I_R}
  \int_{\R^d}
  \frac{
    |\nabla\kappa_\delta(z-x_i)|^2
  }{
    \rho_\delta^N(z,x)
  }\,dz
  \le
  \frac d{\delta^2}
  \mu_x^N(B_R^c).
\end{equation}
If $m_R=0$, the contribution from the particles in $B_R$ is zero and thus \eqref{eq:90} is immediate from \eqref{eq:91}.
Suppose therefore that $m_R\ge1$.
Since for every $z\in\R^d$,
$\rho_\delta^N(z,x)
  \ge
  \frac1N\sum_{j\in I_R}\kappa_\delta(z-x_j)$, we have 
\begin{align}
  &\frac1{N^2}
  \sum_{i\in I_R}
  \int_{\R^d}
  \frac{
    |\nabla\kappa_\delta(z-x_i)|^2
  }{
    \rho_\delta^N(z,x)
  }\,dz \le
  \frac1N
  \sum_{i\in I_R}
  \int_{\R^d}
  \frac{
    |\nabla\kappa_\delta(z-x_i)|^2
  }{
    \sum_{j\in I_R}\kappa_\delta(z-x_j)
  }\,dz.
  \label{eq:92}
\end{align}
We now bound the denominator on the right side of the above expression. 
Let 
$\bar x_R
  \doteq
  \frac1{m_R}\sum_{j\in I_R}x_j$. Then
Jensen's inequality gives
\begin{align}
  \sum_{j\in I_R}\kappa_\delta(z-x_j)
  &\ge
  m_R(2\pi\delta^2)^{-d/2}
  \exp\left\{
    -\frac1{2\delta^2m_R}
    \sum_{j\in I_R}|z-x_j|^2
  \right\}\notag\\
  &=
  m_R\kappa_\delta(z-\bar x_R)
  \exp\left\{
    -\frac1{2\delta^2m_R}
    \sum_{j\in I_R}|x_j-\bar x_R|^2
  \right\}\notag\\
  &\ge m_R e^{-2R^2/\delta^2}
  \kappa_\delta(z-\bar x_R),
  \label{eq:93}
\end{align}
where the last inequality uses the fact that 
for every $i\in I_R$,
$|x_i-\bar x_R|\le 2R$.
Substituting the above inequality into \eqref{eq:92}, we obtain
\begin{multline*}
  \frac1{N^2}
  \sum_{i\in I_R}
  \int_{\R^d}
  \frac{
    |\nabla\kappa_\delta(z-x_i)|^2
  }{
    \rho_\delta^N(z,x)
  }\,dz\le
  \frac{e^{2R^2/\delta^2}}{Nm_R}
  \sum_{i\in I_R}
  \int_{\R^d}
  \frac{
    |\nabla\kappa_\delta(z-x_i)|^2
  }{
    \kappa_\delta(z-\bar x_R)
  }\,dz\\
  = \frac{e^{2R^2/\delta^2}}{Nm_R}
  \sum_{i\in I_R} e^{|x_i-\bar x_R|^2/\delta^2}
  \left(
    \frac d{\delta^2}
    +
    \frac{|x_i-\bar x_R|^2}{\delta^4}
  \right)
  \le \frac{e^{2R^2/\delta^2}}{Nm_R} m_Re^{4R^2/\delta^2}
  \left(
    \frac d{\delta^2}
    +
    \frac{4R^2}{\delta^4}
  \right) = \frac{C_{\delta,R}}{N},
\end{multline*}
with $C_{\delta,R}
  =
  e^{6R^2/\delta^2}
  \left(
    \frac d{\delta^2}
    +
    \frac{4R^2}{\delta^4}
  \right)$.

Combining the above estimate with
\eqref{eq:91} proves
\eqref{eq:90}.
Applying this estimate now with $x=\bar X^N(t)$ and integrating over
$t\in[0,1]$, we obtain
\[\begin{aligned}
  \E\int_0^1D_{N,\delta}(\bar X^N(t))\,dt
  &\le
  \frac{C_{\delta,R}}{N}
  +
  \frac d{\delta^2}
  \E\int_0^1\bar\mu_t^N(B_R^c)\,dt \le \frac{C_{\delta,R}}{N}+
  \frac{d}{\delta^2R^2}
  \E\int_0^1\int_{\R^d}|x|^2\,\bar\mu_t^N(dx).
\end{aligned}
\]
Using Lemma \ref{lem:2} and sending $N\to \infty$, $R \to \infty$ completes the proof.
\end{proof}

We now estimate the term $\widetilde R^{N,\delta}(\bar \mu^N_t)$ on the right side
of \eqref{eq:70}.
Define, for $\mu \in \clp_2(\R^d)$,
$h_\mu(x)
  \doteq
  \nabla\cdot b_\mu(x)
  =
  \Delta V(x)+\Delta W*\mu(x)$
and
\begin{equation}\label{eq:94}
  \alpha_\delta(\mu)
  \doteq
  \left|
    \int_{\R^d}h_\mu(x)\,\mu(dx)
    -
    \int_{\R^d}h_\mu(z)\rho_\delta^\mu(z)\,dz
  \right|.
\end{equation}

\begin{lemma}\label{lem:12}
There exists $C<\infty$ such that, for every $x \in (\R^d)^N$,
\begin{equation}\label{eq:95}
\begin{aligned}
  \left|
    \widetilde R_{N,\delta}(\mu^N_x)
  \right|
  &\le
  \alpha_\delta(\mu^N_x)
  +
  L_b\sqrt d\,\delta\,\cJ_\delta(\mu^N_x)^{1/2}
  +
  C\delta
  \left(
    1+\int_{\R^d}|y|^2\,\mu^N_x(dy)
  \right)^{1/2}
  +
  \frac{|\Delta W(0)|}{N}.
\end{aligned}
\end{equation}
\end{lemma}
\begin{proof}
Recalling the definition of $A_N$,
\begin{equation}\label{eq:96}
  A_N(\mu^N_x)
  =
  \int_{\R^d}h_{\mu^N_x}(y)\,\mu^N_x(dy)
  -
  \frac{\Delta W(0)}{N}.
\end{equation}

Set, for $\mu \in \clp_2(\R^d)$,
$s_\delta^\mu(z)
  \doteq
  \nabla\log\rho_\delta^\mu(z)$
and let
$q_\delta^\mu(z)
  \doteq
  B_\delta[\mu](z)-b_\mu(z)$.
Then, for $\mu \in \clp_2(\R^d)$,
\[
\begin{aligned}
  C_\delta(\mu)
  &=
  \int_{\R^d}
  s_\delta^\mu(z)\cdot b_\mu(z)
  \rho_\delta^\mu(z)\,dz+
  \int_{\R^d}
  s_\delta^\mu(z)\cdot q_\delta^\mu(z)
  \rho_\delta^\mu(z)\,dz\\
&=-
  \int_{\R^d}
  h_\mu(z)\rho_\delta^\mu(z)\,dz + \int_{\R^d}
  s_\delta^\mu(z)\cdot q_\delta^\mu(z)
  \rho_\delta^\mu(z)\,dz,
\end{aligned}
\]
where the last identity uses $s_\delta^\mu\rho_\delta^\mu
  =
  \nabla\rho_\delta^\mu$ and integration by parts.

Combining \eqref{eq:96} with the above observation, we obtain
\begin{equation}\label{eq:98}
\begin{aligned}
  \widetilde R_{N,\delta}(\mu^N_x)
  &=
  \int_{\R^d}h_{\mu^N_x}(y)\,\mu^N_x(dy)
  -
  \int_{\R^d}h_{\mu^N_x}(z)\rho_\delta^{\mu^N_x}(z)\,dz\\
  &\quad +
  \int_{\R^d}
  s_\delta^{\mu^N_x}(z)\cdot q_\delta^{\mu^N_x}(z)
  \rho_\delta^{\mu^N_x}(z)\,dz
  -
  \frac{\Delta W(0)}{N}.
\end{aligned}
\end{equation}
Now, for $\mu \in \clp_2(\R^d)$, set
$\tilde G_\delta(\mu,z)
  \doteq
  s_\delta^\mu(z)+B_\delta[\mu](z)$.
Then
\[
\begin{aligned}
  \int
  s_\delta^\mu(z)\cdot q_\delta^\mu(z)
  \rho_\delta^\mu(z) dz
  &=
  \int
  \tilde G_\delta(\mu,z)\cdot q_\delta^\mu(z)
  \rho_\delta^\mu(z)\,dz -
  \int
  B_\delta[\mu](z)\cdot q_\delta^\mu(z)
  \rho_\delta^\mu(z)\,dz.
\end{aligned}
\]
By Cauchy--Schwarz inequality, the definition of $\cJ_\delta(\mu)$, and
\eqref{eq:76} (recall the definition of $q^{\mu}_{\delta}(z)$),
\begin{equation}\label{eq:99}
\begin{aligned}
  \left|
  \int
  \tilde G_\delta(\mu, z)\cdot q_\delta^\mu(z)
  \rho_\delta^\mu(z) dz
  \right|
  &\le
  \cJ_\delta(\mu)^{1/2}
  \left(
    \int|q_\delta^\mu(z)|^2\rho_\delta^\mu(z) dz
  \right)^{1/2}
  \le
  L_b\sqrt d\,\delta\,\cJ_\delta(\mu)^{1/2}.
\end{aligned}
\end{equation}
Moreover, using \eqref{eq:71}, \eqref{eq:76}, and linear growth of $b_{\mu}$, 
\begin{equation}\label{eq:100}
\begin{aligned}
  \left|
  \int
  B_\delta[\mu](z)\cdot q_\delta^\mu(z)
  \rho_\delta^\mu(z) dz
  \right|
  &\le
  C\delta
  \left(
    1+\int_{\R^d}|x|^2\,\mu(dx)
  \right)^{1/2}.
\end{aligned}
\end{equation}
The result follows from
\eqref{eq:98},
\eqref{eq:99}, and \eqref{eq:100}.
\end{proof}

The next lemma estimates the error $\alpha_\delta$.
\begin{lemma}\label{lem:13}
We have,
$\lim_{\delta\downarrow0}
  \sup_N
  \E\int_0^1
  \alpha_\delta(\bar\mu_t^N)\,dt
  =0$.
\end{lemma}

\begin{proof}
Fix $\mu\in\cP_2(\R^d)$. Let $X,Y$ be independent random variables
with common law $\mu$, and let $G$ be a standard Gaussian random
variable independent of $(X,Y)$. Then
\begin{align}
  \alpha_\delta(\mu)
  &\le
  \E\left|
    \Delta V(X+\delta G)-\Delta V(X)
  \right|+
  \E\left|
    \Delta W(X+\delta G-Y)-\Delta W(X-Y)
  \right|.
  \label{eq:102}
\end{align}
Fix $R,M>0$, and let
$E_{R,M}
  \doteq
  \{|X|\le R,\ |Y|\le R,\ |G|\le M\}$. By uniform continuity of $\Delta V, \Delta W$ on every compact set in $\R^d$, there is a 
  modulus of continuity $\omega_{R,M}: \R_+ \to \R_+$ such that $\omega_{R,M}(r) \downarrow 0$ as $r \downarrow 0$, and
\[
\begin{aligned}
  \left|
    \Delta V(X+\delta G)-\Delta V(X)
  \right| + \left|
    \Delta W(X+\delta G-Y)-\Delta W(X-Y)
  \right|
  \le
  \omega_{R,M}(\delta M)
\end{aligned}
\]
on $E_{R,M}$.
Since $\Delta V$ and $\Delta W$ are bounded, for some $C>0$, and all $\mu \in \clp_2(\R^d)$,
\[
\begin{aligned}
  \alpha_\delta(\mu)
  &\le
  \omega_{R,M}(\delta M)
  +
  C\mu(B_R^c)
  +
  C\PP(|G|>M).
\end{aligned}
\]
By Markov's inequality we obtain
\[
\begin{aligned}
  \sup_N
  \E\int_0^1
  \alpha_\delta(\bar\mu_t^N)\,dt \le
  \omega_{R,M}(\delta M)
  +
  \frac{C}{R^2}
  \sup_N
  \E\int_0^1
  \int_{\R^d}|x|^2\,\bar\mu_t^N(dx)\,dt
  +
  C\PP(|G|>M).
\end{aligned}
\]
The result follows from Lemma \ref{lem:2} on first sending $\delta\downarrow0$, then $R\to\infty$, and finally
 $M\to\infty$. 
\end{proof}

We now derive a uniform bound on $\cJ_\delta$.
\begin{lemma}\label{lem:14}
There exist $\delta_0 \in (0,1)$ and $C<\infty$ such that
\begin{equation}\label{eq:103}
  \sup_{0<\delta\le\delta_0}
  \limsup_{N\to\infty}
  \E\int_0^1
  \cJ_\delta(\bar\mu_t^N)\,dt
  \le C.
\end{equation}
\end{lemma}

\begin{proof}
Define
$J_{N,\delta}
  \doteq
  \E\int_0^1
  \cJ_\delta(\bar\mu_t^N)\,dt$,
and
$J_\delta
  \doteq
  \limsup_{N\to\infty}J_{N,\delta}$.

Taking expectation in \eqref{eq:70}, we have, for some $c_1<\infty$,
\begin{align}
  \frac12J_{N,\delta}
  &\le
  \left(\frac12 +c_1\delta\right)
  \E\left[
    \frac1N\sum_{i=1}^N
    \int_0^1|u_i^N(t)|^2\,dt
  \right] + c_1\delta +
  \E\int_0^1
  \left|
    \widetilde R_{N,\delta}
    (\bar\mu_t^N)
  \right|dt
  +
  r_{N,\delta},
  \label{eq:104}
\end{align}
where, using Lemma \ref{lem:10} and Lemma \ref{lem:11}, for each fixed $\delta>0$,
$\limsup_{N\to\infty}|r_{N,\delta}|=0$.

By Lemma~\ref{lem:12}, Lemma \ref{lem:2}, and
Cauchy--Schwarz inequality, for some $c_2>0$,
\begin{align}
  \E\int_0^1
  \left|
    \widetilde R_{N,\delta}
    (\bar\mu_t^N)
  \right|dt\le
  \E\int_0^1
  \alpha_\delta(\bar\mu_t^N)\,dt
  +
  L_b\sqrt d\,\delta\,J_{N,\delta}^{1/2}
  +
  c_2\delta
  +
  \frac{|\Delta W(0)|}{N}.
  \label{eq:105}
\end{align}

Combining
\eqref{eq:104},
\eqref{eq:105}, the uniform bound \eqref{eq:32} and taking the limit as
$N\to\infty$, gives, for some $c_3>0$,
\begin{equation*}
  \frac12J_\delta
  \le
  c_3
  +
  a_\delta
  +
  L_b\sqrt d\,\delta\,J_\delta^{1/2},
\end{equation*}
where, using Lemma~\ref{lem:13}, $a_\delta
  \doteq
  \sup_N
  \E\int_0^1
  \alpha_\delta(\bar\mu_t^N)\,dt \to 0$
  as $\delta\downarrow0$.
The result now follows by an application of Young's inequality.
\end{proof}

We can now prove the asymptotic negligibility of the fifth term on the right side of
\eqref{eq:70}.
\begin{proposition}\label{prop:2}
We have
\begin{equation}\label{eq:106}
  \lim_{\delta\downarrow0}
  \limsup_{N\to\infty}
  \E\int_0^1
  \left|
    \widetilde R_{N,\delta}
    (\bar\mu_t^N)
  \right|dt
  =0.
\end{equation}
\end{proposition}

\begin{proof}
By Lemma~\ref{lem:12} and the Cauchy--Schwarz inequality,
\[
\begin{aligned}
\limsup_{N\to\infty}
  \E\int_0^1
  \left|
    \widetilde R_{N,\delta}
    (\bar\mu_t^N)
  \right|dt&\le
  \sup_N
  \E\int_0^1
  \alpha_\delta(\bar\mu_t^N)\,dt +
  L_b\sqrt d\,\delta
  \left[
    \limsup_{N\to\infty}
    \E\int_0^1
    \cJ_\delta(\bar\mu_t^N)\,dt
  \right]^{1/2}\\
  &\quad +
  C\delta
  \left[
    1+
    \sup_N
    \E\int_0^1
    \int_{\R^d}|x|^2
    \bar\mu_t^N(dx)\,dt
  \right]^{1/2}.
\end{aligned}
\]
As $\delta \to 0$, the first term converges to zero by
Lemma~\ref{lem:13}, the second converges
to zero by Lemma~\ref{lem:14}, and the third
converges to zero by Lemma \ref{lem:2}.
The result follows.
\end{proof}
\subsubsection{Proof of Proposition \ref{prop:1}.}
\label{sec:6}

Set 
\begin{equation*}
  C_N
  \doteq
  \int_{\cP(\R^d\times\R^d)}
  \int_{\R^d\times\R^d}
  |y|^2\,\eta(dx\,dy)\,[\Theta^N]_2(d\eta).
\end{equation*}
For $\delta>0$, let
\[
  A_{N,\delta}
  \doteq
  \int_{\cP(\R^d)}
  \cJ_\delta(\gamma)\,[\Theta^N]_1(d\gamma).
\]
Lemma~\ref{lem:6} gives
\begin{multline}
  \frac12 A_{N,\delta}
  \le
  \frac12 C_N
  +L_b\sqrt d\,\delta\,C_N^{1/2}
  -\frac1{T_N}
  \left[
    \cF_\delta(\bar\mu_1^N)-\cF_\delta(\bar\mu_0^N)
  \right]\\
  +\int_0^1D_{N,\delta}(\bar X^N(t))\,dt
  +\frac1{T_N}\widetilde M_{N,\delta}(1)
  +\int_0^1\widetilde R_{N,\delta}(\bar\mu_t^N)\,dt.
\end{multline}
Taking expectations and using Jensen's inequality, we obtain
\begin{multline}\label{eq:107}
  \E A_{N,\delta}
  \le
  \E C_N
  +2L_b\sqrt d\,\delta\,(\E C_N)^{1/2}
  +2\E\left|
    \frac1{T_N}
    \left[
      \cF_\delta(\bar\mu_1^N)-\cF_\delta(\bar\mu_0^N)
    \right]
  \right|\\
  +2\E\int_0^1D_{N,\delta}(\bar X^N(t))\,dt
  +2\E\left|\frac1{T_N}\widetilde M_{N,\delta}(1)\right|
  +2\E\int_0^1
    |\widetilde R_{N,\delta}(\bar\mu_t^N)|\,dt.
\end{multline}
Applying Lemma~\ref{lem:10} and
Lemma~\ref{lem:11}, and then sending
$N\to\infty$, gives
\begin{align}
  \liminf_{N\to\infty}\E A_{N,\delta}
  &\le
  \liminf_{N\to\infty}\E C_N
  +2L_b\sqrt d\,\delta
  \left(\sup_{N\ge1}\E C_N\right)^{1/2}
  +
  2\limsup_{N\to\infty}
  \E\int_0^1
  |\widetilde R_{N,\delta}(\bar\mu_t^N)|\,dt.
\end{align}
By Lemma~\ref{lem:7} and the convergence
$\Theta^N\Rightarrow\Theta$,
\begin{equation}\label{eq:108}
  \liminf_{N\to\infty}\E A_{N,\delta}
  \ge
  \E\int_{\cP(\R^d)}
  \cJ_\delta(\gamma)\,[\Theta]_1(d\gamma).
\end{equation}
Combining the last two displays, we find
\begin{align}
  \liminf_{N\to\infty}\E C_N
  &\ge
  \E\int_{\cP(\R^d)}
  \cJ_\delta(\gamma)\,[\Theta]_1(d\gamma)
  -2L_b\sqrt d\,\delta
  \left(\sup_{N\ge1}\E C_N\right)^{1/2}
  \notag\\
  &\quad
  -
  2\limsup_{N\to\infty}
  \E\int_0^1
  |\widetilde R_{N,\delta}(\bar\mu_t^N)|\,dt.
  \label{eq:109}
\end{align}
Sending $\delta\downarrow0$ and using
Proposition~\ref{prop:2} and
\eqref{eq:40},
\begin{equation}\label{eq:110}
  \liminf_{N\to\infty}\E C_N
  \ge
  \liminf_{\delta\downarrow0}
  \E\int_{\cP(\R^d)}
  \cJ_\delta(\gamma)\,[\Theta]_1(d\gamma).
\end{equation}
Finally, 
Corollary~\ref{cor:1} and Lemma 
\ref{lem:8} yield
\begin{align}
  \liminf_{\delta\downarrow0}
  \E\int_{\cP(\R^d)}
  \cJ_\delta(\gamma)\,[\Theta]_1(d\gamma)
  &\ge
  \E\int_{\cP(\R^d)}
  \cD(\gamma)\,[\Theta]_1(d\gamma)
  =
  4\E\int_{\cP(\R^d)}
  \mathfrak i(\gamma)\,[\Theta]_1(d\gamma),
  \label{eq:111}
\end{align}
This
completes the proof of the proposition. \hfill \qed

\subsection{Completing the upper bound proof.}
\label{sec:7}
We now prove the Laplace upper bound. 
Recall the inequality in \eqref{eq:39}. Also recall that, in order to prove \eqref{eq:30}, by a standard subsequential argument and by appealing to the Skorokhod representation theorem, we assume without loss of generality that $\Theta^N \to \Theta$, a.s. in the weak convergence topology.

\noindent {\bf Proof of the upper bound \eqref{eq:30}.} 
From \eqref{eq:39}, with $C_N$ as in Section \ref{sec:6},
\begin{equation}\label{eq:112}
  -\frac1{NT_N}
  \log
  \E\exp\{-NT_NF(\Gamma^{N,T_N})\}+\frac1N+2L_F\varepsilon_N
  \ge
  \frac14\E C_N
  +
  \E F([\Theta^N]_1).
\end{equation}
Using Proposition~\ref{prop:1}, the continuity of $F$ and the definition of $\cI$, we obtain
\begin{multline}
  \liminf_{N\to\infty}
  \left[
    \frac14\E C_N
    +
    \E F([\Theta^N]_1)
  \right]
  \ge
  \E F([\Theta]_1)
  +
  \E\int_{\cP(\R^d)}
  \mathfrak i(\gamma)\,[\Theta]_1(d\gamma)
  \\
  =
  \E\left[
    F([\Theta]_1)
    +
    \cI([\Theta]_1)
  \right] \ge
  \inf_{\Gamma\in\clx}
  \{F(\Gamma)+\cI(\Gamma)\}.
  \label{eq:113}
\end{multline}
Since
$\varepsilon_N\to0$, letting $N\to\infty$ in
\eqref{eq:112} and using
\eqref{eq:113} 
completes the proof of the Laplace upper bound. \hfill \qed

\section{Lower Bound}
\label{sec:8}
Assumption \ref{ass:2}(b) will be in force throughout this section.
The dissipativity condition, Assumption \ref{ass:3}, will not be used in this section.
We prove in this section that the large deviation lower bound holds. This (cf. \cite[Theorem 1.8]{buddupbook}) is equivalent to showing that, for 
every $F\in C_b(\clx)$
\begin{equation}\label{eq:114}
  \limsup_{N\to\infty}
  -\frac1{NT_N}\log
  \E\exp\{-NT_NF(\Gamma^{N,T_N})\}
  \le
  \inf_{\Gamma\in\clx}
  \left\{
    F(\Gamma)+\cI(\Gamma)
  \right\}.
\end{equation}
We remark that, although \cite[Theorem~1.8]{buddupbook} assumes that $\cI$ has compact
sublevel sets, this property is used only in proving the large-deviation
upper bound. The proof that the Laplace lower bound implies the
large-deviation lower bound does not require this assumption; see
\cite[pp.~12--13]{buddupbook}.

\subsection{Regularization of near minimizers}

In this section and the next we show that one can find near-minimizers for the infimum on the right side \eqref{eq:114} that are supported on finitely many measures in $\cP(\R^d)$, each of which have nice regularity properties. We now introduce this family of regular probability measures in $\cP(\R^d)$.

\begin{definition}\label{def:2}
Let $\cP_{\mathrm{reg}}(\R^d)$ denote the collection of all
$\nu\in\cP_2(\R^d)$ of the form
$\nu(dx)=\varrho(x)\,dx$,
where $\varrho\in C^\infty(\R^d)$ is strictly positive and the
following conditions hold.
\begin{enumerate}[label=(\alph*)]
\item
The score function
$x\mapsto s_\nu(x)\doteq\nabla\log\varrho(x)$
is globally Lipschitz.
\item
There exist $R_\nu<\infty$, such that $s_\nu(x)=-x$,
  for all $|x|\ge R_\nu$.
\end{enumerate}
\end{definition}
We remark that the second property of $\nu$ implies that, for some $\kappa_\nu, C_\nu>0$,
\begin{equation}\label{eq:115}
  x\cdot s_\nu(x)
  \le
  -\kappa_\nu|x|^2+C_\nu,
  \qquad x\in\R^d.
\end{equation}

For $\mu\in\cP_{\mathrm{reg}}(\R^d)$, let
$Z^{\mu,x}$ solve
\begin{equation}\label{eq:116}
  dZ^{\mu,x}(s)
  =
  s_\mu(Z^{\mu,x}(s))\,ds
  +
  \sqrt{2}\,dB(s),
  \qquad
  Z^{\mu,x}(0)=x,
\end{equation}
where $B$ is a standard $d$-dimensional Brownian motion.
Note that, since $s_\mu$ is Lipschitz, the equation has a unique pathwise solution.

We now establish a key approximation statement that we will use to pass from an
arbitrary measure $\mu$ with finite cost $\mathfrak i(\mu)$ to a measure that is in the  class $\cP_{\mathrm{reg}}(\R^d)$.
The proposition is based on a standard denseness, truncation and mollification
arguments in Sobolev spaces; see \cite{meyser}.
We include the proof because the weighted approximation, which ensures that the score
equals $-x$ outside a compact set, and the dependence of
$b_\mu$ on $\mu$ require some modifications.

\begin{proposition}\label{prop:3}
Let $\mu\in\cP_2(\R^d)$ satisfy $\mathfrak i(\mu)<\infty$.
Then there exists a sequence
$\{\mu_n\}_{n\ge1}\subset\cP_{\mathrm{reg}}(\R^d)$, such that
$\cW_2(\mu_n,\mu)\to 0$
and
$\mathfrak i(\mu_n)
  \to
  \mathfrak i(\mu)$.
\end{proposition}

\begin{proof}
Since $\mathfrak i(\mu)<\infty$, $\mu$ is absolutely continuous,
say $\mu(dx)=\rho(x)\,dx$, and $\rho$ is weakly differentiable with
$\int_{\R^d}\frac{|\nabla\rho(x)|^2}{\rho(x)}\,dx<\infty$
by the characterization discussed below Definition \ref{def:1}. Consequently, with $f=\sqrt{\rho}$,
the Sobolev chain rule gives
$f\in H^1(\R^d)$
 and
   $2\nabla f=f\nabla\log\rho$
   a.e,
and
\[
   4\int_{\R^d}|\nabla f|^2\,dx
   =
   \int_{\R^d}|\nabla\log\rho|^2\,d\mu.
\]
Moreover, since $\mu\in\cP_2(\R^d)$,
we in fact have that
$f\in H^1(\R^d)
   \cap L^2\bigl(\R^d;(1+|x|^2)\,dx\bigr)$.

 By a standard truncation, mollification, and approximation argument, there exist
$\phi_n\in C_c^\infty(\R^d)$,
 $\phi_n\ge0$,
such that
\begin{equation}\label{eq:118}
  \|\phi_n-f\|_{H^1(\R^d)}
  +
  \bigl\||x|(\phi_n-f)\bigr\|_{L^2(\R^d)}
  \to 0.
\end{equation}
In particular,
$\|\phi_n\|_{L^2}\to \|f\|_{L^2}=1$.
Let
$h(x)
  \doteq
  (2\pi)^{-d/4}e^{-|x|^2/4}$.
Then, $h>0$,
  $\|h\|_{L^2}=1$, and
 $2\nabla\log h(x)=-x$.
Choose any sequence $\varepsilon_n\downarrow0$, and define
\[
  \widetilde f_n
  \doteq
  \phi_n+\varepsilon_n h,
  \qquad
  a_n
  \doteq
  \|\widetilde f_n\|_{L^2}^{-1},
  \qquad
  f_n
  \doteq
  a_n\widetilde f_n.
\]
Since $\|\phi_n\|_{L^2}\to1$ and $\varepsilon_n\to0$, we have
$a_n\to 1$.
Moreover, $f_n\in C^\infty(\R^d)$,
$f_n>0$,
$\|f_n\|_{L^2}=1$.
Define $\rho_n\doteq f_n^2$,
and
  $\mu_n(dx)\doteq\rho_n(x)\,dx$. It is easily seen that
\begin{equation}\label{eq:119}
  \|f_n-f\|_{H^1(\R^d)}
  +
  \bigl\||x|(f_n-f)\bigr\|_{L^2(\R^d)}
  \to 0.
\end{equation}

We next verify convergence of the measures $\mu_n \to \mu$ in $\cW_2$. Since
$\rho_n-\rho
  =
  (f_n-f)(f_n+f)$,
Cauchy--Schwarz and \eqref{eq:119} give
$\|\rho_n-\rho\|_{L^1}
  \to 0$.
Also,
\[
\begin{aligned}
  \int_{\R^d}|x|^2|\rho_n(x)-\rho(x)|\,dx
  &\le
  \bigl\||x|(f_n-f)\bigr\|_{L^2}
  \bigl\||x|(f_n+f)\bigr\|_{L^2}
  \to 0.
\end{aligned}
\]
It follows that
$\cW_2(\mu_n,\mu)\to0$.

We now verify that $\mu_n \in \cP_{\mathrm{reg}}(\R^d)$. Clearly, $\mu_n$ has a $C^{\infty}$ strictly positive density $\rho_n$. Next, let
$R_n<\infty$ be such that
$\mbox{supp}(\phi_n)\subset B_{R_n}$.
Then, for $|x|>R_n$,
$f_n(x)=a_n\varepsilon_n h(x)$,
and hence, for such $x$,
$\nabla\log\rho_n(x)
  =
  2\nabla\log f_n(x)
  =
  2\nabla\log h(x)
  =
  -x$.
Thus
  $s_{\mu_n}(x)
  =
  -x$ for
 $|x|>R_n$, proving (b) in Definition \ref{def:2}.

Furthermore, on the compact set $\overline B_{R_n+1}$, the matrix
$D s_{\mu_n}
  =
  2D^2\log f_n$
is bounded and outside $B_{R_n}$, it equals $-I_d$. Therefore,
$\sup_{x\in\R^d}|D s_{\mu_n}(x)|<\infty$, which implies that
$s_{\mu_n}$ is globally Lipschitz, and thus (a) in Definition \ref{def:2} holds. We thus have that $\mu_n \in \cP_{\mathrm{reg}}(\R^d)$.

Finally, we prove the second convergence stated in the proposition. Since
$\nabla\log\rho_n\,f_n
  =
  2\nabla f_n$ and $2\nabla f
  =
  f\nabla\log\rho$,
\begin{equation}\label{eq:120}
  4\mathfrak i(\mu_n)
  =
  \int_{\R^d}
  \left|
    2\nabla f_n+b_{\mu_n}f_n
  \right|^2dx, \;\;\; 4\mathfrak i(\mu)
  =
  \int_{\R^d}
  \left|
    2\nabla f+b_\mu f
  \right|^2dx.
\end{equation}

  Also, the linear growth of $b_\nu$ and
\eqref{eq:119} imply
\begin{equation}\label{eq:123}
  \|b_\mu(f_n-f)\|_{L^2}
  \le
  C\left(
    \|f_n-f\|_{L^2}
    +
    \bigl\||x|(f_n-f)\bigr\|_{L^2}
  \right)
  \to 0.
\end{equation}
Also, letting $L_W$ be a Lipschitz constant for $\nabla W$, for any coupling
$\pi_n$ of $\mu_n$ and $\mu$,
\[
\begin{aligned}
  |\nabla W*\mu_n(x)-\nabla W*\mu(x)|
  &=
  \left|
    \int_{\R^d\times\R^d}
    \bigl[
      \nabla W(x-y)-\nabla W(x-z)
    \bigr]
    \,\pi_n(dy,dz)
  \right|
  \\
  &\le
  L_W
  \int_{\R^d\times\R^d}|y-z|\,\pi_n(dy,dz).
\end{aligned}
\]
Taking the infimum over couplings gives
$\sup_{x\in\R^d}
  |b_{\mu_n}(x)-b_\mu(x)|
  \le
  L_W \cW_2(\mu_n,\mu)$.
Since $\|f_n\|_{L^2}=1$, we now have,
$\|(b_{\mu_n}-b_\mu)f_n\|_{L^2}
  \le
  L_W \cW_2(\mu_n,\mu)
  \to 0$.
Combining this with \eqref{eq:123}, we see that
$b_{\mu_n}f_n
  \to
  b_\mu f$
  in $L^2(\R^d)$.
Together with \eqref{eq:119}, this says
$2\nabla f_n+b_{\mu_n}f_n
  \to
  2\nabla f+b_\mu f$
  in $L^2(\R^d)$.
Finally, using
\eqref{eq:120}, we conclude that
$\mathfrak i(\mu_n)
  \to
  \mathfrak i(\mu)$, and the result follows.
\end{proof}

The following is a standard consequence of the Foster--Lyapunov
criterion for geometric ergodicity of continuous-time Markov processes;
see \cite[\text{Theorem~5.2(c)},\text{ p.~1681}]{dowmeytwe}.
In the lemma below $\|\gamma\|_{\mathrm{TV}}$ denotes the total variation norm of a finite signed measure $\gamma$.

\begin{lemma}\label{lem:15}
Let $\nu\in\cP_{\mathrm{reg}}(\R^d)$.
 Let $P_t^\nu$ denote the transition semigroup
of the diffusion in \eqref{eq:116} (with $\mu$ there replaced by $\nu$).
Then $\nu$ is its unique invariant probability measure. Moreover,
there exist constants $c_\nu,C_\nu>0$ such that
\begin{equation}\label{eq:124}
  \|P_t^\nu(x,\cdot)-\nu\|_{\mathrm{TV}}
  \le
  C_\nu(1+|x|^2)e^{-c_\nu t},
  \qquad t\ge0.
\end{equation}
Consequently, for every $M<\infty$,
\begin{equation}\label{eq:125}
  \sup_{\eta\in\cP_2(\R^d):
    \int|x|^2\,\eta(dx)\le M}
  \|\int P_t^{\nu}(x, \cdot) \eta(dx)-\nu\|_{\mathrm{TV}} \le C_\nu(1+ M) e^{-c_\nu t}
  \to 0,
  \qquad t\to\infty.
\end{equation}
\end{lemma}
One can also easily obtain the following uniform second moment bound.
\begin{lemma}\label{lem:16}
Let $\mu\in\cP_{\mathrm{reg}}(\R^d)$ and let $Z^{\mu,\xi}$ solve
\eqref{eq:116} with $Z^{\mu,\xi}(0) = \xi$ where $\xi$ is a $\R^d$ valued square integrable random variable that is independent of the driving noise. Then, for some $C_{\mu} >0$
\begin{equation}\label{eq:126}
  \E|Z^{\mu,\xi}(s)|^2
  \le
  e^{-2\kappa_\mu s}\E|\xi|^2
  +
  C_\mu,
  \qquad s\ge0.
\end{equation}
\end{lemma}

\begin{proof}
By It\^o's formula applied to \eqref{eq:116} and recalling property (b) in Definition \ref{def:2}, we have, for some 
$\tilde \kappa_{\mu},D_{\mu} >0$,
\[
\begin{aligned}
  \frac{d}{ds}\E|Z^{\mu,\xi}(s)|^2
  &=
  2\E\left[
    Z^{\mu,\xi}(s)\cdot
    s_\mu(Z^{\mu,\xi}(s))
  \right]
  +2d
  \le
  -2\tilde \kappa_\mu\E|Z^{\mu,\xi}(s)|^2
  +2D_\mu.
\end{aligned}
\]
The result is now immediate.
\end{proof}

\subsection{Reduction to finitely many regular target measures}

The following lemma will allow us to work with a finite collection of regular target
measures in proving the lower bound.
\begin{lemma}\label{lem:17}
  Fix
$F\in C_b(\clx)$ and $\delta>0$. Then there exist
$k\in\N$, numbers $p_1,\ldots,p_k>0$ with
$\sum_{\ell=1}^kp_\ell=1$, and measures
$\mu_1,\ldots,\mu_k\in\cP_{\mathrm{reg}}(\R^d)$ such that, with
$\Gamma^\delta
  \doteq
  \sum_{\ell=1}^kp_\ell\delta_{\mu_\ell}$,
we have
\begin{equation}\label{eq:127}
  F(\Gamma^\delta)
  +
  \sum_{\ell=1}^kp_\ell\mathfrak i(\mu_\ell)
  \le
  \inf_{\Gamma\in\clx}
  \left\{
    F(\Gamma)+\cI(\Gamma)
  \right\}
  +4\delta.
\end{equation}
\end{lemma}

\begin{proof}
Choose $\Gamma_*\in\clx$ such that
\begin{equation}\label{eq:128}
  F(\Gamma_*)+\cI(\Gamma_*)
  \le
  \inf_{\Gamma\in\clx}
  \left\{
    F(\Gamma)+\cI(\Gamma)
  \right\}
  +\delta.
\end{equation}
We must have $\cI(\Gamma_*)<\infty$. Let
$M_1,M_2,\ldots$ be i.i.d.\ $\cP(\R^d)$-valued random variables
with common law $\Gamma_*$. Since
$\int\mathfrak i(\mu)\,\Gamma_*(d\mu)<\infty$,
the strong law of large numbers and the almost sure convergence of
empirical measures give
\[
  \frac1m\sum_{j=1}^m\delta_{M_j}
  \Rightarrow
  \Gamma_*, \quad \mbox{ and } \quad
  \frac1m\sum_{j=1}^m\mathfrak i(M_j)
  \to 
  \cI(\Gamma_*)
\]
almost surely. Fix a realization $\omega$ for which both convergences hold and choose sufficiently large $m$ such that,
\[
  \left|
    F\left(
      \frac1m\sum_{j=1}^m\delta_{M_j}(\omega)
    \right)
    -
    F(\Gamma_*)
  \right|
  \le
  \frac{\delta}{2}, \quad \mbox{ and }
  \frac1m\sum_{j=1}^m\mathfrak i(M_j(\omega))
  \le
  \cI(\Gamma_*)+\frac{\delta}{2}.
\]
We suppress $\omega$ hereafter.
For each $j=1,\ldots,m$, apply
Proposition \ref{prop:3} to choose
$\mu_j\in\cP_{\mathrm{reg}}(\R^d)$ sufficiently close to $M_j$
such that
$\mathfrak i(\mu_j)
  \le
  \mathfrak i(M_j)+\delta$
and, by appealing to the continuity of $F$,
\[
  \left|
    F\left(
      \frac1m\sum_{r=1}^{j-1}\delta_{\mu_r}
      +
      \frac1m\delta_{\mu_j}
      +
      \frac1m\sum_{r=j+1}^m\delta_{M_r}
    \right)
    -
    F\left(
      \frac1m\sum_{r=1}^{j-1}\delta_{\mu_r}
      +
      \frac1m\sum_{r=j}^m\delta_{M_r}
    \right)
  \right|
  \le
  \frac{\delta}{m}.
\]
  A telescoping argument gives
\[
  \left|
    F\left(
      \frac1m\sum_{j=1}^m\delta_{\mu_j}
    \right)
    -
    F\left(
      \frac1m\sum_{j=1}^m\delta_{M_j}
    \right)
  \right|
  \le
  \delta,
\]
whereas
\[
  \frac1m\sum_{j=1}^m\mathfrak i(\mu_j)
  \le
  \frac1m\sum_{j=1}^m\mathfrak i(M_j)+\delta.
\]
Thus, with $k=m$ and $p_j=1/m$,
\[
\begin{aligned}
  F(\Gamma^\delta)
  +
  \sum_{j=1}^kp_j\mathfrak i(\mu_j)
  &\le
  F(\Gamma_*)+\cI(\Gamma_*)+3\delta
  \le
  \inf_{\Gamma\in\clx}
  \left\{
    F(\Gamma)+\cI(\Gamma)
  \right\}
  +4\delta.
\end{aligned}
\]
\end{proof}

\subsection{Construction of the stabilizing controls}
Let 
$\Gamma^\delta$ be as in
   Lemma~\ref{lem:17}. Write, for $\ell=1,\ldots,k$,
$\mu_\ell(dx)=\rho_\ell(x)\,dx$,
 $s_\ell(x)
  \doteq
  \nabla\log\rho_\ell(x)$,
and 
\[
  P_0=0,
  \qquad
  P_\ell=\sum_{r=1}^\ell p_r.
\]

On the interval $(P_{\ell-1},P_\ell]$, define the feedback control
\begin{equation}\label{eq:129}
\begin{aligned}
  u_i^N(t)
  \doteq
  \nabla V(\bar X_i^N(t))
  +
  \nabla W*\bar\mu_t^N(\bar X_i^N(t))
  +
  s_\ell(\bar X_i^N(t)),\;
  t\in(P_{\ell-1},P_\ell].
\end{aligned}
\end{equation}
Since $s_\ell$ is globally Lipschitz,
the corresponding controlled system \eqref{eq:28} has a unique strong solution.
Note that,
under the control \eqref{eq:129}, on the interval
$(P_{\ell-1},P_\ell]$ the controlled particles satisfy
\begin{equation}\label{eq:130}
  d\bar X_i^N(t)
  =
  T_Ns_\ell(\bar X_i^N(t))\,dt
  +
  \sqrt{2T_N}\,d\widetilde B_i(t),
  \qquad i=1,\ldots,N.
\end{equation}
Consequently, the controlled particle processes are independent on
$[0,1]$.

\subsection{Uniform moment estimates}
\label{sec:9}

Recall that we assume that Assumption \ref{ass:2}(b) holds.
Also, recall that, since
$s_\ell(x)=-x$ outside a compact set, for each $\ell = 1, \ldots k$, there exist constants
$c_0>0$ and $C_0<\infty$, such that, 
\begin{equation}
\label{eq:131}
  x\cdot s_\ell(x)
  \leq
  -c_0|x|^2+C_0,
  \qquad
  x\in\R^d,\quad \ell=1,\dots,k.
\end{equation}
With $\alpha$ as in Assumption \ref{ass:2}(b), let $p=2+\alpha$.
The next lemma follows by using a suitable Lyapunov function; we provide a proof for completeness.

\begin{lemma}
\label{lem:18}
 There exists $C_p<\infty$, such that
\begin{equation}
\label{eq:132}
  \sup_{N\geq1}\sup_{t\in[0,1]}
  \frac1N\sum_{i=1}^N
  \E|\bar X_i^N(t)|^p
  \leq C_p.
\end{equation}
\end{lemma}

\begin{proof}
For $\varepsilon\in(0,1]$, and $x \in \R^d$, define
$\phi_\varepsilon(x)
  \doteq
  (\varepsilon+|x|^2)^{p/2}$.
Then
$\nabla\phi_\varepsilon(x)
  =
  p(\varepsilon+|x|^2)^{p/2-1}x$
and
\begin{align*}
  \Delta\phi_\varepsilon(x)
  &=
  pd(\varepsilon+|x|^2)^{p/2-1}
  +
  p(p-2)|x|^2
  (\varepsilon+|x|^2)^{p/2-2}
  \leq
  p(d+p-2)
  (\varepsilon+|x|^2)^{p/2-1}.
\end{align*}
By \eqref{eq:131},
\begin{align*}
  s_\ell(x)\cdot\nabla\phi_\varepsilon(x)
  &=
  p(\varepsilon+|x|^2)^{p/2-1}
  x\cdot s_\ell(x)
  \leq
  -pc_0|x|^2
  (\varepsilon+|x|^2)^{p/2-1}
  +
  pC_0
  (\varepsilon+|x|^2)^{p/2-1}.
\end{align*}
Using
\[
  |x|^2(\varepsilon+|x|^2)^{p/2-1}
  =
  (\varepsilon+|x|^2)^{p/2}
  -
  \varepsilon(\varepsilon+|x|^2)^{p/2-1}
\]
and Young's inequality, there exist constants
$c_p>0$ and $D_p<\infty$, independent of
$\varepsilon\in(0,1]$ and $\ell$, such that
\begin{equation}
\label{eq:135}
  s_\ell(x)\cdot\nabla\phi_\varepsilon(x)
  +
  \Delta\phi_\varepsilon(x)
  \leq
  -c_p\phi_\varepsilon(x)+D_p.
\end{equation}

Fix $\ell$ and $t\in[P_{\ell-1},P_\ell]$. Applying It\^o's formula
to $\phi_\varepsilon(\bar X_i^N(t))$ and using
\eqref{eq:130} and
\eqref{eq:135} gives
\[
  \frac{d}{dt}
  \E\phi_\varepsilon(\bar X_i^N(t))
  \leq
  -c_pT_N
  \E\phi_\varepsilon(\bar X_i^N(t))
  +
  D_pT_N.
\]
Consequently, for some $\tilde C_p$ (independent of $i$ and $N$)
\begin{align}
\label{eq:136}
  \E\phi_\varepsilon(\bar X_i^N(t))
  &\leq
  e^{-c_pT_N(t-P_{\ell-1})}
  \E\phi_\varepsilon(\bar X_i^N(P_{\ell-1}))
  +
  \tilde C_p.
\end{align}
Averaging over $i$ and iterating
\eqref{eq:136} over the finitely many
intervals yields
\[
  \sup_{N\geq1}\sup_{t\in[0,1]}
  \frac1N\sum_{i=1}^N
  \E\phi_\varepsilon(\bar X_i^N(t))
  <\infty,
\]
where the bound is uniform in $\varepsilon\in(0,1]$. Letting
$\varepsilon\downarrow0$ and applying Fatou's lemma proves
\eqref{eq:132}.
\end{proof}

\subsection{Stabilization}
\label{sec:10}
The main result of this section is Proposition \ref{prop:4}, which gives a key convergence property for $\bar \mu^N_t$.  
We first record a standard estimate that upgrades total variation
convergence to Wasserstein convergence when the tails of the second
moments are controlled, cf. \cite[Theorem~6.12]{vill}: There exists a constant $C<\infty$ such that, for
every $R\geq1$ and every $\nu,\mu\in\cP_2(\R^d)$,
\begin{equation}
\label{eq:137}
  \cW_2^2(\nu,\mu)
  \leq
  CR^2\|\nu-\mu\|_{\mathrm{TV}}
  +
  C\int_{\{|x|>R\}}|x|^2\,\nu(dx)
  +
  C\int_{\{|x|>R\}}|x|^2\,\mu(dx).
\end{equation}

Choose a sequence $a_N$ satisfying
$a_N\to\infty$,
$\eta_N\doteq\frac{a_N}{T_N}\to0$.
Since every $p_\ell>0$, for all sufficiently large $N$,
$\eta_N<p_\ell$ for every $\ell$. Hereafter, we will only consider such $N$. Define the transition and hold intervals
\[
  J_{\ell,N}^{\mathrm{tr}}
  \doteq
  (P_{\ell-1},P_{\ell-1}+\eta_N],
\quad
\mbox{ and }
\quad
  J_{\ell,N}^{\mathrm{hold}}
  \doteq
  (P_{\ell-1}+\eta_N,P_\ell].
\]
The next two lemmas concern convergence properties over the hold intervals.
For $t\in[P_{\ell-1},P_\ell]$, let
$\nu_{i,t}^N
  \doteq
  \Law(\bar X_i^N(t))$
and let
$\widehat\nu_t^N
  \doteq
  \frac1N\sum_{i=1}^N\nu_{i,t}^N$.
\begin{lemma}
\label{lem:19}
For every $\ell=1,\dots,k$,
\begin{equation}
\label{eq:138}
  \sup_{t\in J_{\ell,N}^{\mathrm{hold}}}
  \frac1N\sum_{i=1}^N
  \cW_2^2(\nu_{i,t}^N,\mu_\ell)
  \to 0.
\end{equation}
Consequently,
  $\sup_{t\in J_{\ell,N}^{\mathrm{hold}}}
  \cW_2(\widehat\nu_t^N,\mu_\ell)
  \to 0$.
\end{lemma}

\begin{proof}
Fix $\ell$. Since, for $t\in J_{\ell,N}^{\mathrm{hold}}$,
$T_N(t-P_{\ell-1})
  \geq
  T_N\eta_N
  =
  a_N$, we have from the second statement in 
  Lemma \ref{lem:15}, that, for some $C_{\ell}, c_{\ell} >0$,
\begin{align}
\label{eq:139}
  \|\nu_{i,t}^N-\mu_\ell\|_{\mathrm{TV}}
  &\leq
  C_\ell e^{-c_\ell a_N}
  \left(
    1+
    \E|\bar X_i^N(P_{\ell-1})|^2
  \right).
\end{align}
Averaging over $i$ and applying
Lemma~\ref{lem:18} gives, for some $\tilde C_{\ell}>0$,
\begin{equation}
\label{eq:140}
  \sup_{t\in J_{\ell,N}^{\mathrm{hold}}}
  \frac1N\sum_{i=1}^N
  \|\nu_{i,t}^N-\mu_\ell\|_{\mathrm{TV}}
  \leq
  \tilde C_\ell e^{-c_\ell a_N}.
\end{equation}
By Lemma~\ref{lem:18}, for some $C_1>0$ and all $R>0$,
\[
  \sup_{N\geq1}\sup_{t\in[0,1]}\frac1N\sum_{i=1}^N
  \int_{\{|x|>R\}}|x|^2\,\nu_{i,t}^N(dx)
  \leq
  \frac{C_1}{R^\alpha},
\]
uniformly in $N$ and $t$. 
Applying \eqref{eq:137} with
$\nu=\nu_{i,t}^N$ and $\mu=\mu_\ell$, and averaging over $i$,
we have on combining the above observations,
\begin{align}\label{eq:141}
  &\sup_{t\in J_{\ell,N}^{\mathrm{hold}}}
  \frac1N\sum_{i=1}^N
  \cW_2^2(\nu_{i,t}^N,\mu_\ell)
  \leq
  C_1\tilde C_\ell R^2e^{-c_\ell a_N}
  +
  \frac{CC_1}{R^\alpha}
  +
  C\int_{\{|x|>R\}}|x|^2\,\mu_\ell(dx),
\end{align}
where $C$ is as in \eqref{eq:137}.
From Definition \ref{def:2} (b) we see that, since $\mu_\ell$ has Gaussian tails outside
a compact set, 
$\int_{\{|x|>R\}}|x|^2\,\mu_\ell(dx)
  \to 0$
as $R\to\infty$.
Now in \eqref{eq:141} sending first $N\to\infty$, using $a_N\to\infty$, and then letting
$R\to\infty$, proves
\eqref{eq:138}.
Finally, convexity of $\cW_2^2$ in its first argument proves the final statement in the lemma.
\end{proof}

The next lemma is a standard Glivenko--Cantelli estimate for independent
(not necessarily identically distributed) random variables; see
\cite[Theorem~1]{Wellner1981}. The argument there applies rowwise to the
present triangular array. Lemma~\ref{lem:18} gives the
 tightness uniformly in $t$ needed for the proof. We omit the details.
\begin{lemma}
\label{lem:20}
For every $\ell=1,\dots,k$,
$\int_{J_{\ell,N}^{\mathrm{hold}}}
  \E d_{\mathrm{BL}}
  (\bar\mu_t^N,\widehat\nu_t^N)\,dt
  \to 0$.
\end{lemma}

\begin{proposition}
\label{prop:4}
For every $\ell=1,\dots,k$,
$\int_{J_{\ell,N}^{\mathrm{hold}}}
  \E d_{\mathrm{BL}}(\bar\mu_t^N,\mu_\ell)\,dt
  \to 0$.
\end{proposition}

\begin{proof}
This is immediate from Lemmas~\ref{lem:19} and
\ref{lem:20} and an application of triangle inequality.
\end{proof}

\subsection{Convergence of the occupation measure}
\label{sec:11}

\begin{proposition}
\label{prop:5}
Under the controls $\{u^N_i\}$, defined by \eqref{eq:129},
$\bar\Gamma^{N,T_N}
  \to
  \Gamma^\delta$
  in probability in $\cX$.
\end{proposition}

\begin{proof}
Let $d_{\mathrm{BL}}^{(1)}$ denote the bounded-Lipschitz metric on
$\cP(\R^d)$ and let $d_{\mathrm{BL}}^{(2)}$ denote the corresponding
metric on
$\cX=\cP(\cP(\R^d))$.
Fix $\Phi\in\mathrm{BL}_1(\cP(\R^d))$. Then
\begin{align*}
  |\langle \Phi, \bar\Gamma^{N,T_N}\rangle - \langle \Phi, \Gamma^\delta\rangle | &= \left|
    \int_0^1\Phi(\bar\mu_t^N)\,dt
    -
    \sum_{\ell=1}^kp_\ell\Phi(\mu_\ell)
  \right|\\
  &\leq
  \sum_{\ell=1}^k
  \int_{J_{\ell,N}^{\mathrm{hold}}}
  |\Phi(\bar\mu_t^N)-\Phi(\mu_\ell)|\,dt
  +
  \sum_{\ell=1}^k
  \int_{J_{\ell,N}^{\mathrm{tr}}}
  |\Phi(\bar\mu_t^N)-\Phi(\mu_\ell)|\,dt\\
  &\le \sum_{\ell=1}^k
  \int_{J_{\ell,N}^{\mathrm{hold}}} d_{\mathrm{BL}}^{(1)} (\bar\mu_t^N, \mu_\ell) + 2k\eta_N,
  \end{align*}
  where the last inequality uses the facts that $|\Phi|\leq1$ and $\Phi$ is one-Lipschitz.
Taking supremum over $\Phi \in \mathrm{BL}_1(\cP(\R^d))$
$$ d_{\mathrm{BL}}^{(2)}
  (\bar\Gamma^{N,T_N},\Gamma^\delta) \le \sum_{\ell=1}^k
  \int_{J_{\ell,N}^{\mathrm{hold}}} d_{\mathrm{BL}}^{(1)} (\bar\mu_t^N, \mu_\ell) + 2k\eta_N.$$
  Finally, taking
expectations, using
Proposition~\ref{prop:4}, and recalling that
$\eta_N\to0$, we obtain
$\E d_{\mathrm{BL}}^{(2)}
  (\bar\Gamma^{N,T_N},\Gamma^\delta)
  \to 0$, which completes the proof of the proposition.
\end{proof}

\subsection{Convergence of the control costs}
\label{sec:12}

Fix $\ell$ and define
\begin{equation}
\label{eq:142}
  g_\ell(x)
  \doteq
  \nabla V(x)
  +
  \nabla W*\mu_\ell(x)
  +
  s_\ell(x).
\end{equation}
On $(P_{\ell-1},P_\ell]$, write
\begin{equation}
\label{eq:143}
  u_i^N(t)
  =
  g_\ell(\bar X_i^N(t))
  +
  R_{i,\ell}^N(t),
\end{equation}
where
\begin{equation}
\label{eq:144}
  R_{i,\ell}^N(t)
  \doteq
  \nabla W*\bar\mu_t^N(\bar X_i^N(t))
  -
  \nabla W*\mu_\ell(\bar X_i^N(t)).
\end{equation}
The first lemma below estimates the contribution from the second term in the control in 
\eqref{eq:143} while the second lemma estimates the first.
\begin{lemma}
\label{lem:21}
For every $\ell=1,\dots,k$,
$
  \E\int_{J_{\ell,N}^{\mathrm{hold}}}
  \frac1N\sum_{i=1}^N
  |R_{i,\ell}^N(t)|^2\,dt
 \to 0$.
\end{lemma}

\begin{proof}
Fix $\ell$ and
$t\in J_{\ell,N}^{\mathrm{hold}}$. Since $\{\bar X^N_i, i \in [N]\}$ are independent, for each $i$ we may choose an optimal coupling
$(\bar X_i^N(t),Y_{i,t}^N)$
of $\nu_{i,t}^N$ and $\mu_\ell$ in such a way that the pairs are
independent over $i$. Thus the variables
$Y_{1,t}^N,\dots,Y_{N,t}^N$ are independent with common law
$\mu_\ell$, and
\begin{equation}
\label{eq:145}
  \E|\bar X_i^N(t)-Y_{i,t}^N|^2
  =
  \cW_2^2(\nu_{i,t}^N,\mu_\ell).
\end{equation}
Write
\begin{align*}
  R_{i,\ell}^N(t)
  &=
  \frac1N\sum_{j=1}^N
  \left[
    \nabla W(\bar X_i^N(t)-\bar X_j^N(t))
    -
    \nabla W(Y_{i,t}^N-Y_{j,t}^N)
  \right]
  \\
  &\quad+
  \frac1N\sum_{j=1}^N
  \left[
    \nabla W(Y_{i,t}^N-Y_{j,t}^N)
    -
    \nabla W*\mu_\ell(Y_{i,t}^N)
  \right]
  \\
  &\quad+
  \nabla W*\mu_\ell(Y_{i,t}^N)
  -
  \nabla W*\mu_\ell(\bar X_i^N(t))
  \\
  &\doteq
  R_{i,\ell}^{N,1}(t)
  +
  R_{i,\ell}^{N,2}(t)
  +
  R_{i,\ell}^{N,3}(t).
\end{align*}
Using the Lipschitz property of $\nabla W$, for some $C<\infty$,
\begin{align*}
  |R_{i,\ell}^{N,1}(t)|^2
  &\leq
  C\left(
    |\bar X_i^N(t)-Y_{i,t}^N|^2
    +
    \frac1N\sum_{j=1}^N
    |\bar X_j^N(t)-Y_{j,t}^N|^2
  \right),\;\;\;
  |R_{i,\ell}^{N,3}(t)|^2
  &\leq
  C|\bar X_i^N(t)-Y_{i,t}^N|^2.
\end{align*}
Therefore, by \eqref{eq:145},
\begin{align}
\label{eq:146}
  \frac1N\sum_{i=1}^N
  \E\left[
    |R_{i,\ell}^{N,1}(t)|^2
    +
    |R_{i,\ell}^{N,3}(t)|^2
  \right]
  \leq
  \frac{C}{N}\sum_{i=1}^N
  \cW_2^2(\nu_{i,t}^N,\mu_\ell).
\end{align}

We next estimate $R_{i,\ell}^{N,2}$. Write
\[
  R_{i,\ell}^{N,2}(t)
  =
  \frac1N\sum_{\substack{j=1\\j\neq i}}^N
  \zeta_{i,j}
  +
  \frac1N\zeta_{i,i},
\]
where
$\zeta_{i,j}
  \doteq
  \nabla W(Y_{i,t}^N-Y_{j,t}^N)
  -
  \nabla W*\mu_\ell(Y_{i,t}^N)$.
Since, conditionally on $Y_{i,t}^N$, the variables
$\{\zeta_{i,j}:j\neq i\}$ are independent and centered, we have
\[
  \E\left[
    \left|
      \frac1N\sum_{1\le j \le N, j\neq i}
      \zeta_{i,j}
    \right|^2
    \,\Big|\,
    Y_{i,t}^N
  \right]
  \leq
  \frac{C}{N}
  \left(
    1+|Y_{i,t}^N|^2
    +
    \int_{\R^d}|y|^2\,\mu_\ell(dy)
  \right).
\]
Furthermore, the diagonal term satisfies the same estimate with $N^{-2}$ in
place of $N^{-1}$. Since $Y_{i,t}^N$ has law $\mu_\ell$, combining the above observations, for some $C_{\ell}<\infty$,
\begin{equation}
\label{eq:147}
  \frac1N\sum_{i=1}^N
  \E|R_{i,\ell}^{N,2}(t)|^2
  \leq
  \frac{C_\ell}{N}.
\end{equation}
Combining \eqref{eq:146} and \eqref{eq:147}, integrating over the hold
interval, and applying Lemma~\ref{lem:19}, 
proves the result.
\end{proof}

\begin{lemma}
\label{lem:22}
For every $\ell=1,\dots,k$,
$
  \int_{J_{\ell,N}^{\mathrm{hold}}}
  \frac1N\sum_{i=1}^N
  \E|g_\ell(\bar X_i^N(t))|^2\,dt
  \to
  p_\ell
  \int_{\R^d}|g_\ell(x)|^2\,\mu_\ell(dx)$.
\end{lemma}

\begin{proof}
Because $\nabla V$, $\nabla W$, and $s_\ell$ are globally Lipschitz,
$g_\ell$ has at most linear growth.
Lemma~\ref{lem:18} therefore yields
\begin{equation}
\label{eq:148}
  \sup_{N\geq1}
  \sup_{t\in J_{\ell,N}^{\mathrm{hold}}}
  \int_{\R^d}
  |g_\ell(x)|^{2+\alpha}\,
  \widehat\nu_t^N(dx)
  <\infty.
\end{equation}
In particular, the family $|g_\ell|^2$ is uniformly integrable with
respect to the measures
$\widehat\nu_t^N$, uniformly in $N$ and
$t\in J_{\ell,N}^{\mathrm{hold}}$.
Thus, using Lemma~\ref{lem:19},
together with the above uniform integrability, we obtain
\begin{equation}
\label{eq:149}
  \sup_{t\in J_{\ell,N}^{\mathrm{hold}}}
  \left|
    \int_{\R^d}|g_\ell(x)|^2\,\widehat\nu_t^N(dx)
    -
    \int_{\R^d}|g_\ell(x)|^2\,\mu_\ell(dx)
  \right|
  \to 0.
\end{equation}
Since
$
  |J_{\ell,N}^{\mathrm{hold}}|
  =
  p_\ell-\eta_N
  \to p_\ell$,
the conclusion follows on integrating over the interval $J_{\ell,N}^{\mathrm{hold}}$.
\end{proof}

We can now establish the desired convergence of costs.

\begin{proposition}
\label{prop:6}
For every $\ell=1,\dots,k$,
\begin{equation}
\label{eq:150}
  \lim_{N\to\infty}
  \E\left[
    \frac1{4N}\sum_{i=1}^N
    \int_{J_{\ell,N}^{\mathrm{hold}}}
    |u_i^N(t)|^2\,dt
  \right]
  =
  p_\ell\mathfrak i(\mu_\ell).
\end{equation}
\end{proposition}

\begin{proof}
Using Cauchy--Schwarz inequality,
\begin{align*}
  &\left|
    \E\int_{J_{\ell,N}^{\mathrm{hold}}}
    \frac1N\sum_{i=1}^N
    g_\ell(\bar X_i^N(t))
    \cdot R_{i,\ell}^N(t)\,dt
  \right|
  \\
  &\quad\leq
  \left(
    \E\int_{J_{\ell,N}^{\mathrm{hold}}}
    \frac1N\sum_{i=1}^N
    |g_\ell(\bar X_i^N(t))|^2\,dt
  \right)^{1/2}
  \left(
    \E\int_{J_{\ell,N}^{\mathrm{hold}}}
    \frac1N\sum_{i=1}^N
    |R_{i,\ell}^N(t)|^2\,dt
  \right)^{1/2}.
\end{align*}
The first factor is uniformly bounded by
Lemma~\ref{lem:22}, while the second from Lemma 
\ref{lem:21} converges to zero.
Thus, recalling \eqref{eq:143} and using Lemma  
\ref{lem:21} and Lemma~\ref{lem:22} again,
\[
  \E\left[
    \frac1{4N}\sum_{i=1}^N
    \int_{J_{\ell,N}^{\mathrm{hold}}}
    |u_i^N(t)|^2\,dt
  \right]
  \to
  \frac{p_\ell}{4}\int_{\R^d}|g_\ell(x)|^2\,\mu_\ell(dx).
\]
Recalling the definition of $\mathfrak i$, we see that
  $\int_{\R^d}|g_\ell(x)|^2\,\mu_\ell(dx)
  =
  4\mathfrak i(\mu_\ell)$.
The result follows.
\end{proof}

We next show that the costs on the transition intervals are
negligible.

\begin{lemma}
\label{lem:23}
We have
\begin{equation}
\label{eq:151}
  \E\left[
    \frac1{4N}\sum_{i=1}^N
    \int_{\cup_{\ell=1}^kJ_{\ell,N}^{\mathrm{tr}}}
    |u_i^N(t)|^2\,dt
  \right]
  \to 0.
\end{equation}
\end{lemma}

\begin{proof}
Since $\nabla V$, $\nabla W$, and the functions
$s_\ell$ have at most linear growth, there exists a constant
$C<\infty$, independent of $\ell$, such that
\begin{align}
\label{eq:152}
  \frac1N\sum_{i=1}^N|u_i^N(t)|^2
  &\leq
  C\left(
    1+
    \frac1N\sum_{i=1}^N|\bar X_i^N(t)|^2
  \right).
\end{align}
Taking expectations and applying
Lemma~\ref{lem:18} gives
$\sup_{N\geq1}\sup_{t\in[0,1]}
  \frac1N\sum_{i=1}^N
  \E|u_i^N(t)|^2
  <\infty$.
Since the total length of the transition intervals is $k\eta_N$ and $\eta_N \to 0$, the result follows.
\end{proof}

Combining the hold-interval and transition-interval estimates gives
the asymptotic control cost.

\begin{proposition}
\label{prop:7}
We have,
$
\limsup_{N\to\infty}
  \E\left[
    \frac1{4N}\sum_{i=1}^N
    \int_0^1|u_i^N(t)|^2\,dt
  \right]
  =
  \sum_{\ell=1}^kp_\ell\mathfrak i(\mu_\ell)$.
\end{proposition}

\begin{proof}
Decompose $[0,1]$ into the hold intervals and transition intervals.
Apply Proposition~\ref{prop:6} on every hold interval
and Lemma~\ref{lem:23} on the union of the
transition intervals.
\end{proof}

\subsection{Completion of the lower bound}
\label{sec:13}
We now prove the large deviation lower bound.
\begin{theorem}
\label{thm:3}
For every
$F\in C_b(\cX)$,
\begin{equation}
\label{eq:153}
  \limsup_{N\to\infty}
  -\frac1{NT_N}
  \log
  \E\exp\{-NT_NF(\Gamma^{N,T_N})\}
  \leq
  \inf_{\Gamma\in\cX}
  \{F(\Gamma)+\cI(\Gamma)\}.
\end{equation}
\end{theorem}

\begin{proof}
Fix $\delta>0$ and choose
$\Gamma^\delta
  =
  \sum_{\ell=1}^kp_\ell\delta_{\mu_\ell}$
as in Lemma \ref{lem:17}, so that \eqref{eq:127} holds.
Let $u^N$ be the feedback control defined as in \eqref{eq:129}. The variational representation \eqref{eq:29} gives
\begin{align*}
  -\frac1{NT_N}
  \log
  \E\exp\{-NT_NF(\Gamma^{N,T_N})\}
  \leq
  \E\left[
    F(\bar\Gamma^{N,T_N})
    +
    \frac1{4N}\sum_{i=1}^N
    \int_0^1|u_i^N(t)|^2\,dt
  \right].
\end{align*}
By Proposition~\ref{prop:5},
$
  \E F(\bar\Gamma^{N,T_N})
  \to
  F(\Gamma^\delta)$
and by Proposition~\ref{prop:7},
\[\limsup_{N\to\infty}
  \E\left[
    \frac1{4N}\sum_{i=1}^N
    \int_0^1|u_i^N(t)|^2\,dt
  \right]
  =
  \sum_{\ell=1}^kp_\ell\mathfrak i(\mu_\ell).
\]
Consequently,
\begin{align*}
  \limsup_{N\to\infty}
  -\frac1{NT_N}
  \log
  \E\exp\{-NT_NF(\Gamma^{N,T_N})\}
  &\leq
  F(\Gamma^\delta)
  +
  \sum_{\ell=1}^kp_\ell\mathfrak i(\mu_\ell)
  \\
  &\leq
  \inf_{\Gamma\in\cX}
  \{F(\Gamma)+\cI(\Gamma)\}
  +
  4\delta.
\end{align*}
Since $\delta>0$ is arbitrary, the result follows.
\end{proof}

\section{Compactness of the level sets of the rate function}
\label{sec:14}
Throughout this section, Assumption 
\ref{ass:3} will hold.
We prove in this section that the functional $\cI$ defined in
\eqref{eq:15} is a rate function on
$\cP(\cP(\R^d))$, equipped with the topology of weak convergence.
Several arguments in this section follow by suitably adapting standard results (cf. \cite{AGS}).

A key result in this section is the coercivity estimate for the static cost
$\mathfrak i$ in Lemma \ref{lem:25}. In preparation for this we start with the following lemma. 
Recall that, from Assumption \ref{ass:3}, $\kappa \doteq \lambda-C>0$.
The lemma below and its proof will follow the terminology and notation from \cite[Chapter 9]{AGS}.
\begin{lemma}\label{lem:24}
The functional $\cF$  is $\kappa$-convex along generalized geodesics
in $(\cP_2(\R^d),\cW_2)$ in the sense of
\cite[Definition~9.2.4]{AGS}. In particular, $\cF$ is
$\kappa$-geodesically convex.
\end{lemma}

\begin{proof}
Let $\mu_0,\mu_1,\mu_2\in\cP_2(\R^d)$ and let
$\boldsymbol\mu\in\cP((\R^d)^3)$ satisfy
$\pi^{0,i}_\#\boldsymbol\mu\in\Pi_o(\mu_0,\mu_i)$, $i=1,2$,
where $\Pi_o(\mu_0, \mu_i)$ is the space of optimal couplings (in Wasserstein-$2$ distance) between $\mu_0$ and $\mu_i$ and $\pi^{0,i}$ are the projection maps on the corresponding coordinates.
Define the generalized geodesic from $\mu_1$ to $\mu_2$, with base
$\mu_0$, by
$\mu_t
  \doteq
  \bigl((1-t)\pi^1+t\pi^2\bigr)_\#\boldsymbol\mu$,
  $t\in[0,1]$,
cf.\ \cite[Definition~9.2.2]{AGS}, and set
\[
  \cW_{\boldsymbol\mu}^2(\mu_1,\mu_2)
  \doteq
  \int_{(\R^d)^3}|x_1-x_2|^2\,
  \boldsymbol\mu(dx_0,dx_1,dx_2).
\]
We need to show that
\begin{equation}\label{eq:154}
  \cF(\mu_t)
  \le
  (1-t)\cF(\mu_1)+t\cF(\mu_2)
  -
  \frac{\kappa}{2}t(1-t)
  \cW_{\boldsymbol\mu}^2(\mu_1,\mu_2),
  \qquad t\in[0,1].
\end{equation}

Write
$\cF(\nu)=\cH(\nu)+\cE(\nu)$
where
$\cH(\nu)=
  \int_{\R^d}\rho\log\rho\,dx$,
      if $\nu$ has a density $\rho$ and is
    $\infty$ otherwise,
and
\[
  \cE(\nu)
  =
  \int_{\R^d}V(x)\,\nu(dx)
  +
  \frac12\iint_{\R^d\times\R^d}
  W(x-z)\,\nu(dx)\nu(dz).
\]
The entropy $\cH$ is $0$-convex along generalized geodesics in
$\cP_2(\R^d)$; indeed, this follows from
\cite[Proposition~9.3.9 and Remark~9.3.8]{AGS}, applied to
$F(r)=r\log r$, for which
$s^dF(s^{-d})=-d\log s$
is convex and nonincreasing on $(0,\infty)$. Thus
\begin{equation}\label{eq:155}
  \cH(\mu_t)
  \le
  (1-t)\cH(\mu_1)+t\cH(\mu_2).
\end{equation}
We therefore only need to verify the corresponding
$\kappa$-convexity estimate for the combined potential and interaction
energy $\cE$. The standard separate convexity results for potential and
interaction energies are given in
\cite[Proposition~9.3.2, Example~9.3.4 and Proposition~9.3.5]{AGS};
we need a minor modification of these arguments since here we instead use the overall dissipativity condition in
Assumption~\ref{ass:3}.
Fix $(x_0,x_1,x_2)$, $(z_0,z_1,z_2)$ in $(\R^d)^3$ and set
$x_t=(1-t)x_1+tx_2$,
$z_t=(1-t)z_1+tz_2$.
For $0\le s<t\le1$, Assumption~\ref{ass:3}, applied to
$(x_t,x_s,z_t,z_s)$, gives
\begin{multline*}
 (x_t-x_s)\cdot
 \bigl(\nabla V(x_t)-\nabla V(x_s)\bigr)
\\
 +
 (x_t-x_s)\cdot
 \bigl(
   \nabla W(x_t-z_t)-\nabla W(x_s-z_s)
 \bigr)
 \ge
 \lambda|x_t-x_s|^2-C|z_t-z_s|^2.
\end{multline*}
Integrating the first term with respect to
$\boldsymbol\mu(dx_0,dx_1,dx_2)$ and the second with respect to
$\boldsymbol\mu(dx_0,dx_1,dx_2)
 \boldsymbol\mu(dz_0,dz_1,dz_2)$, and using
\[
  x_t-x_s=(t-s)(x_2-x_1),
  \qquad
  z_t-z_s=(t-s)(z_2-z_1),
\]
we obtain
\begin{align*}
 &\int
 (x_t-x_s)\cdot
 \bigl(\nabla V(x_t)-\nabla V(x_s)\bigr)\,\boldsymbol\mu(dx_0,dx_1,dx_2)
\\
 &\quad+
 \iint
 (x_t-x_s)\cdot
 \bigl(
   \nabla W(x_t-z_t)-\nabla W(x_s-z_s)
 \bigr)\,
 \boldsymbol\mu(dx_0,dx_1,dx_2)\,\boldsymbol\mu(dz_0,dz_1,dz_2)
\\
 &\qquad\ge
 (\lambda-C)(t-s)^2
 \int |x_2-x_1|^2\,\mu(dx_0,dx_1,dx_2)
 = \kappa(t-s)^2
 \cW_{\boldsymbol\mu}^2(\mu_1,\mu_2).
\end{align*}

Now set
$e(t)\doteq\cE(\mu_t)$.
Using the linear growth of $\nabla V$ and $\nabla W$ to differentiate
under the integrals, and using that $\nabla W$ is odd, we obtain
\[
  e'(t)
  =
  \int
  \nabla V(x_t)\cdot(x_2-x_1)\,d\boldsymbol\mu
  +
  \iint
  \nabla W(x_t-z_t)\cdot(x_2-x_1)\,
  d\boldsymbol\mu\,d\boldsymbol\mu,
\]
where $d\boldsymbol\mu\,d\boldsymbol\mu$ in the last display is short for
$\boldsymbol\mu(dx_0,dx_1,dx_2)\,\boldsymbol\mu(dz_0,dz_1,dz_2)$.
Therefore the preceding estimate implies
\[
  e'(t)-e'(s)
  \ge
  \kappa(t-s)
  \cW_{\boldsymbol\mu}^2(\mu_1,\mu_2),
  \qquad 0\le s<t\le1.
\]
Hence
$t\mapsto
  e(t)
  -
  \frac{\kappa}{2}t^2
  \cW_{\boldsymbol\mu}^2(\mu_1,\mu_2)$
is convex. It follows that
\begin{equation}\label{eq:156}
  \cE(\mu_t)
  \le
  (1-t)\cE(\mu_1)+t\cE(\mu_2)
  -
  \frac{\kappa}{2}t(1-t)
  \cW_{\boldsymbol\mu}^2(\mu_1,\mu_2).
\end{equation}
Combining
\eqref{eq:155} and
\eqref{eq:156} proves
\eqref{eq:154}.

Finally, convexity along generalized geodesics implies ordinary
$\kappa$-geodesic convexity; see
\cite[Remark~9.2.8]{AGS}.
\end{proof}
The following coercivity property is key.
\begin{lemma}\label{lem:25}
For every $\mu\in\cP(\R^d)$,
$\mathfrak i(\mu)
  \ge
  \frac{\kappa^2}{4}\,
  \cW_2^2(\mu,\mu_*)$.
\end{lemma}

\begin{proof}
The assertion is immediate if $\mathfrak i(\mu)=\infty$, so suppose that
$\mathfrak i(\mu)<\infty$. Then
$\mu(dx)=\rho(x)\,dx\in\cP_2(\R^d)$ and
$G_\mu
  \doteq
  \nabla\log\rho+\nabla V+\nabla W*\mu
  \in L^2(\mu;\R^d)$,
with
$4\mathfrak i(\mu)
  =
  \|G_\mu\|_{L^2(\mu)}^2$.
Note that $\cF$ is lower semicontinuous on $(\cP_2(\R^d),\cW_2)$. Indeed, for the
entropy term this follows from
the classical lower semicontinuity of relative entropy with respect to
a Gaussian reference measure, together with convergence of second
moments under $\cW_2$-convergence; see also
\cite[Remark~9.3.8]{AGS}. The potential and interaction terms are
continuous under $\cW_2$-convergence by their at-most-quadratic growth,
which follows from the global Lipschitz continuity of $\nabla V$ and
$\nabla W$.

Although, the hypotheses of \cite[Theorem~10.4.13]{AGS}, which impose separate
convexity conditions on the potential and interaction terms, do not
apply directly here, a minor modification of its argument
can be used for the current setting. Indeed, \cite[Theorem~10.4.6]{AGS} identifies
$\nabla\log\rho$ as the Wasserstein subgradient of the entropy, while
Taylor expansion along optimal transport maps, using the global
Lipschitz continuity of $\nabla V$ and $\nabla W$, shows that the
potential and interaction terms have Wasserstein gradients
$\nabla V$ and $\nabla W*\mu$, respectively, with an
$O(\cW_2^2)$ remainder. It follows that
$G_\mu\in\partial\cF(\mu)$. Consequently, from the definition
of the metric slope it follows that,
$|\partial\cF|(\mu)
  \le
  \|G_\mu\|_{L^2(\mu)}$.
Applying the same argument at $\mu_*$ and using \eqref{eq:17} gives
$0\in\partial\cF(\mu_*)$. By
Lemma~\ref{lem:24}, $\cF$ is
$\kappa$-geodesically convex with $\kappa>0$, so the subgradient
inequality \cite[(10.1.7)]{AGS} yields
\[
  \cF(\nu)-\cF(\mu_*)
  \ge
  \frac{\kappa}{2}\cW_2^2(\nu,\mu_*),
  \qquad \nu\in D(\cF).
\]
Hence $\mu_*$ is the unique minimizer of $\cF$.

As shown in \cite[(2.4.5) and the discussion preceding
Assumption~2.4.5]{AGS}, $\kappa$-geodesic convexity implies
Assumption~2.4.5 therein with $\lambda=\kappa$. Together with the
lower semicontinuity established above, $\kappa>0$, and the unique
minimizer property of $\mu_*$, this allows us to apply
\cite[Lemma~2.4.13, (2.4.19)]{AGS}, giving
\[
  \frac{\kappa}{2}
  \cW_2^2(\mu,\mu_*)
  \le
  \cF(\mu)-\cF(\mu_*)
  \le
  \frac{1}{2\kappa}
  |\partial\cF|^2(\mu) \le \frac{1}{2\kappa}\|G_\mu\|_{L^2(\mu)}^2 = \frac{2}{\kappa}\mathfrak i(\mu),
\]
and the result follows.
\end{proof}

We next prove the lower semicontinuity of the static cost for the
topology of weak convergence.

\begin{lemma}\label{lem:26}
The map
$\mathfrak i:\cP(\R^d) \to [0,\infty]$
is lower semicontinuous for the topology of weak convergence.
\end{lemma}

\begin{proof}
Let $\mu_n\Rightarrow\mu$ in $\cP(\R^d)$. It suffices to consider the
case
$\liminf_{n\to\infty}\mathfrak i(\mu_n)<\infty$.
Passing to a subsequence, we may assume that
$\sup_n\mathfrak i(\mu_n)<\infty$. By
Lemma~\ref{lem:25},
$\sup_n\cW_2(\mu_n,\mu_*)<\infty$, and hence
$\sup_n\int |x|^2\,\mu_n(dx)<\infty$. In particular,
$\mu\in\cP_2(\R^d)$.

By Lemma~\ref{lem:8}, for 
$\nu \in \cP_2(\R^d)$,
\[
  4\mathfrak i(\nu)
  =
  \sup_{\varphi\in C_c^\infty(\R^d;\R^d)}
  \left\{
    -2\int\nabla\cdot\varphi\,d\nu
    +2\int\varphi\cdot b_\nu\,d\nu
    -\int|\varphi|^2\,d\nu
  \right\}.
\]
For each fixed $\varphi\in C_c^\infty(\R^d;\R^d)$, the first and third
terms in braces are continuous under weak convergence. The same is true
of the middle term using the linear growth of $\nabla W$ and uniform integrability
property following from the uniform second-moment bound above.

Thus, for every fixed $\varphi\in C_c^\infty(\R^d;\R^d)$, the
functional inside the supremum is continuous at $\mu$ along
$\{\mu_n\}$. Taking the supremum over $\varphi$ therefore gives
$\mathfrak i(\mu)
  \le
  \liminf_{n\to\infty}\mathfrak i(\mu_n)$,
as needed.
\end{proof}

We can now prove the rate-function property.

\begin{proposition}\label{prop:8}
The functional $\cI$ in \eqref{eq:15}
is a rate function for the topology of weak convergence on
$\cP(\cP(\R^d))$.
\end{proposition}

\begin{proof}
Fix $M<\infty$ and let
$K_M
  \doteq
  \left\{
    \Gamma\in\cP(\cP(\R^d)):
    \cI(\Gamma)\le M
  \right\}$.
We show that $K_M$ is relatively compact and closed.
Define $\Psi(\mu)
  \doteq
  \cW_2^2(\mu,\mu_*)$,
 $\mu\in\cP(\R^d)$,
with the convention that $\Psi(\mu)=\infty$ when
$\mu\notin\cP_2(\R^d)$. The function $\Psi$ is lower
semicontinuous for the topology of weak convergence. Moreover, its
level sets are compact in $\cP(\R^d)$. Indeed, if
$\Psi(\mu)\le R$, then
\[
\begin{aligned}
  \int_{\R^d}|x|^2\,\mu(dx)
  &\le
  2\cW_2^2(\mu,\mu_*)
  +
  2\int_{\R^d}|y|^2\,\mu_*(dy)
  \le
  2R+
  2\int_{\R^d}|y|^2\,\mu_*(dy).
\end{aligned}
\]
Thus the level set $\{\Psi\le R\}$ is tight, and it is closed by the
lower semicontinuity of $\Psi$.

By Lemma~\ref{lem:25}, for every
$\Gamma\in K_M$,
\begin{equation}\label{eq:157}
\begin{aligned}
  \int_{\cP(\R^d)}
  \Psi(\mu)\,\Gamma(d\mu)
  &=
  \int_{\cP(\R^d)}
  \cW_2^2(\mu,\mu_*)\,\Gamma(d\mu)
  \le
  \frac{4}{\kappa^2}
  \int_{\cP(\R^d)}
  \mathfrak i(\mu)\,\Gamma(d\mu)
  \le
  \frac{4M}{\kappa^2}.
\end{aligned}
\end{equation}
Consequently, for every $R>0$,
\[
\begin{aligned}
  \sup_{\Gamma\in K_M}
  \Gamma\left(
    \left\{
      \mu:\Psi(\mu)>R
    \right\}
  \right)
  &\le
  \frac1R
  \sup_{\Gamma\in K_M}
  \int_{\cP(\R^d)}
  \Psi(\mu)\,\Gamma(d\mu)
  \le
  \frac{4M}{\kappa^2R}.
\end{aligned}
\]
Since $\{\Psi\le R\}$ is compact in $\cP(\R^d)$, this proves that
$K_M$ is tight as a family of probability measures on
$\cP(\R^d)$. 
It remains to prove that $K_M$ is closed. Let
$\Gamma_n\in K_M$ and suppose that
$\Gamma_n\Rightarrow\Gamma$ in
$\cP(\cP(\R^d))$. By
Lemma~\ref{lem:26}, the map
$\mathfrak i$ is nonnegative and lower semicontinuous on
$\cP(\R^d)$. Fatou's lemma therefore gives
\[
\begin{aligned}
  \cI(\Gamma)
  &=
  \int_{\cP(\R^d)}
  \mathfrak i(\mu)\,\Gamma(d\mu)
  \le
  \liminf_{n\to\infty}
  \int_{\cP(\R^d)}
  \mathfrak i(\mu)\,\Gamma_n(d\mu)
  \le M.
\end{aligned}
\]
Hence $\Gamma\in K_M$, and thus $K_M$ is closed. The result follows.
\end{proof}

\section{An explicit quadratic mean-field model}
\label{sec:15}

In this section we illustrate the rate function of
Theorem~\ref{thm:1} in a setting where several natural constrained
variational problems can be solved explicitly.  
Throughout this section, let
\begin{equation}\label{eq:158}
  V(x)=\frac a2|x|^2,
  \qquad
  W(x)=\frac b2|x|^2,
  \qquad
  a>0,\quad b>-\frac a2,
\end{equation}
and write
  $c\doteq a+b$.
We note that this example satisfies
Assumption \ref{ass:3}. Indeed,
one can take $\lambda=\frac c2$ and
$C=\frac{b^2}{2c}$. Then 
$\lambda-C =
  \frac{a(a+2b)}{2(a+b)}
  >0$.

For $\mu\in\cP_2(\R^d)$, let
$m_\mu\doteq\int_{\R^d}x\,\mu(dx)$ and
  $\mu^\circ\doteq\tau_{-m_\mu}\#\mu$, where $\tau_{-m_\mu}(x) \doteq x-m_\mu$ for $x \in \R^d$.
  When $\mu^\circ(dx)=\rho(x)\,dx$
and $\rho$ has a weak logarithmic gradient, write, for $c>0$,
\[
  I(\mu^\circ)
  \doteq
  \int_{\R^d}
  |\nabla\log\rho(x)|^2\,\mu^\circ(dx), \; \mathsf{FI}_c(\mu^\circ)
  \doteq
  \int_{\R^d}
  |\nabla\log\rho(x)+cx|^2\,\mu^\circ(dx).
\]

If $\nu(dx)=\rho(x)\,dx \in \cP_2(\R^d)$ is centered (i.e. it has mean $0$), and $\rho$ has a weak logarithmic gradient, then an elementary cutoff argument in the defining integration-by-parts identity gives
\begin{equation}\label{eq:159}
  \int_{\R^d}\nabla\log\rho(x)\,\nu(dx)=0,
  \qquad
  \int_{\R^d}
  x\otimes\nabla\log\rho(x)\,\nu(dx)
  =
  -I_d.
\end{equation}
Since $\nu$ has a density, its covariance matrix
$\Sigma_\nu\doteq\int_{\R^d}xx^{\mathsf T}\,\nu(dx)$ is positive definite. Moreover,
\[
  J(\nu)-\Sigma_\nu^{-1}
  =
  \int_{\R^d}
  \bigl(\nabla\log\rho(x)+\Sigma_\nu^{-1}x\bigr)
  \bigl(\nabla\log\rho(x)+\Sigma_\nu^{-1}x\bigr)^{\mathsf T}
  \,\nu(dx)\succeq0,
\]
where we used \eqref{eq:159}. Thus
\begin{equation}\label{eq:161}
  I(\nu)
  =
  \Tr J(\nu)
  \geq
  \Tr(\Sigma_\nu^{-1}).
\end{equation}
In the above, equality holds if and only if
$\nabla\log\rho(x)=-\Sigma_\nu^{-1}x$ $\nu$-a.e., which is equivalent to
$\nu = \mathcal N(0,\Sigma_\nu)$.

It follows from \eqref{eq:159} and \eqref{eq:161} that, for 
$\nu$ as above
\begin{align}
  \mathsf{FI}_c(\nu)
  &=
  \int_{\R^d}
  \left|
    \nabla\log\rho(x)+cx
  \right|^2
  \nu(dx)
  =
  I(\nu)-2cd+c^2\Tr(\Sigma_\nu) \geq
  \Tr(\Sigma_\nu^{-1})
  -2cd
  +c^2\Tr(\Sigma_\nu).
\label{eq:162}
\end{align}
with equality if and only if
$\nu=\mathcal N(0,\Sigma_\nu)$.

\begin{proposition}
\label{prop:9}
For every $\mu\in\cP_2(\R^d)$ for which
$\mathfrak i(\mu)<\infty$,
\begin{equation}\label{eq:163}
  \mathfrak i(\mu)
  =
  \frac{a^2}{4}|m_\mu|^2
  +
  \frac14\mathsf{FI}_c(\mu^\circ).
\end{equation}
\end{proposition}

\begin{proof}
With $y= x-m_\mu$,
$\nabla V(x)+(\nabla W*\mu)(x)
  =
  a m_\mu+cy$ and if 
$\rho^\circ$ denotes the density of $\mu^\circ$, then 
$\nabla\log\rho(x)
  =
  \nabla\log\rho^\circ(y)$
Hence, using the representation of $\mathfrak i$ in
\eqref{eq:13},
\begin{align*}
  4\mathfrak i(\mu)
  &=
  \int_{\R^d}
  \left|
    \nabla\log\rho^\circ(y)
    +cy
    +a m_\mu
  \right|^2
  \mu^\circ(dy) =
  a^2|m_\mu|^2
  +
  \int_{\R^d}
  \left|
    \nabla\log\rho^\circ(y)+cy
  \right|^2
  \mu^\circ(dy),
\end{align*}
where the second equality follows on using
\eqref{eq:159}.
The result follows.
\end{proof}

The decomposition of $\mathfrak i(\mu)$ in the proposition separates the translational mode of the free-energy
gradient from its centered component. Since the interaction energy is
translation invariant, its force has zero mean and does not contribute to
the center-of-mass term. In the present quadratic setting, the affine form
of the confinement force makes this separation exact, namely, the 
cost for translating the mean depends only on the confinement strength $a$, while the cost associated with shaping the centered cloud
involves both confinement and interaction through $a+b$. This is related
to, but more explicit than, the barycentric correction in the HWBI
inequalities of \cite{AguehGhoussoubKang2004}, where generalized relative
Fisher information is coupled with an explicit barycentric term, and to
the separate treatment of fixed and moving centers of mass in
\cite{carmccvil}. Those works treat substantially more general free
energies and potentials, but establish comparison and convergence
inequalities rather than the exact additive identity above.

As an immediate corollary we have the following.

\begin{corollary}\label{cor:2}
For $\Gamma \in \cP(\cP(\R^d))$,
$\cI(\Gamma) = \frac{1}{4}\int \left(a^2|m_\mu|^2
	  +
	  \mathsf{FI}_c(\mu^\circ)\right) \Gamma(d\mu)$
	  where the integrand is taken to be $\infty$ if $\mu \notin \cP_2(\R^d)$ or $\mu$ does not have a density admitting a weak logarithmic gradient.
\end{corollary}

We now note some other elementary consequences of Proposition \ref{prop:9}.
\begin{corollary}
\label{cor:4}
Let $m\in\R^d$ and let $\Sigma$ be a symmetric positive definite
$d\times d$ matrix. Then
\begin{align}
 &\inf_{\mu\in\cP_2(\R^d):
                  m_\mu=m,\;\Sigma_\mu=\Sigma}
 \mathfrak i(\mu)
 =
 \frac14
 \left[
   a^2|m|^2
   +
   \Tr(\Sigma^{-1})
   -2cd
   +c^2\Tr(\Sigma)
 \right].
\label{eq:164}
\end{align}
The unique minimizer is
$\mathcal N(m,\Sigma)$.
\end{corollary}

\begin{proof}
If $m_\mu=m$ and $\Sigma_\mu=\Sigma$, then $\mu^\circ$ is centered
with covariance $\Sigma$. Proposition~\ref{prop:9}
and \eqref{eq:162}
 therefore give
\[
  \mathfrak i(\mu)
  \geq
  \frac14
  \left[
    a^2|m|^2
    +
    \Tr(\Sigma^{-1})
    -2cd
    +c^2\Tr(\Sigma)
  \right],
\]
with
equality if and only if
$\mu^\circ=\mathcal N(0,\Sigma)$, i.e.
$\mu=\mathcal N(m,\Sigma)$.
\end{proof}

\begin{corollary}
\label{cor:5}
Let $m\in\R^d$ and $q>0$. Then
\begin{equation}\label{eq:166}
  \inf_{\mu\in\cP_2(\R^d):
        m_\mu=m,\; \Tr(\Sigma_{\mu})=q
        }
  \mathfrak i(\mu)
  =
  \frac{a^2}{4}|m|^2
  +
  \frac{(cq-d)^2}{4q}.
\end{equation}
The unique minimizer is
$\mathcal N\left(m,\frac qd I_d\right)$.
\end{corollary}
\begin{proof}
Let $\Sigma=\Sigma_\mu$.  
Note that $\Tr(\Sigma^{-1})$ over positive definite $\Sigma$ satisfying
$\Tr(\Sigma)=q$ is minimized uniquely at
$\Sigma=\frac qd I_d$.
Substitution into
\eqref{eq:164} gives the desired result.
\end{proof}

As a consequence of Corollary \ref{cor:4} we have the following result. Let $\mathbb S_{++}^d$ denote the space of symmetric positive definite
$d\times d$ matrices and define
\[
  G:\R^d\times\mathbb S_{++}^d \to \cP(\R^d),
  \qquad
  G(m,\Sigma)\doteq\mathcal N(m,\Sigma).
\]
The map $G$ is continuous when $\cP(\R^d)$ is equipped with its weak
topology.

\begin{corollary}
\label{cor:7}
Let
$\Pi\in\cP(\R^d\times\mathbb S_{++}^d)$
and define
$
  \Gamma_\Pi
  \doteq
  G_\#\Pi
  =
  \int_{\R^d\times\mathbb S_{++}^d}
  \delta_{\mathcal N(m,\Sigma)}
  \,\Pi(dm\,d\Sigma)$.
 Then
\begin{align}
  \cI(\Gamma_\Pi)
  &=
  \frac14
  \int_{\R^d\times\mathbb S_{++}^d}
  \left[
    a^2|m|^2
    +\Tr(\Sigma^{-1})
    -2cd
    +c^2\Tr(\Sigma)
  \right]
  \Pi(dm\,d\Sigma)
\label{eq:170}
\end{align}
\end{corollary}

\begin{proof}
This is immediate from Corollary \ref{cor:4}.
\end{proof}

 Informally, the event
$\Gamma^{N,T_N}\approx\Gamma_\Pi$
means that, over the long time interval, the empirical distribution of
the particle system visits Gaussian profiles in such a way that the
empirical distribution of their means and covariances is approximately
$\Pi$. The above corollary says that the cost for such behavior is simply the $\Pi$-average of the cost of maintaining
each individual Gaussian profile, which is given by the explicit formula in \eqref{eq:164}.

This formula will also clarify the two competing mechanisms in the
variance-constrained problem below which illustrates that excess dispersion may be produced
either by enlarging the covariances of the individual Gaussian profiles
or by varying their centers over time, and these two mechanisms are optimal in different regimes.

\subsection{Atypical sample means and variances}
\label{rem:3}
We now analyze large deviation mechanisms that lead to atypical values of the 
overall sample means $M^{N,T_N}$ and sample variances $R^{N,T_N}$, defined as
\[
 M^{N,T_N}
  \doteq
  \frac1{T_N}\int_0^{T_N}
  m_{\mu_t^N}\,dt, \;\;
  R^{N,T_N}
  \doteq
  \frac1{T_N}\int_0^{T_N}
  \int_{\R^d}
  |x- M^{N,T_N}|^2
  \,\mu_t^N(dx)\,dt.
\]
We assume that $b>0$.
For $\Gamma\in\clx$, define its average center and overall dispersion by
\[
   m(\Gamma)
  \doteq
  \int_{\cP(\R^d)}m_\mu\,\Gamma(d\mu),
  \qquad
  \mathcal R(\Gamma)
  \doteq
  \int_{\cP(\R^d)}
  \int_{\R^d}
  |x- m(\Gamma)|^2
  \,\mu(dx)\,\Gamma(d\mu),
\]
whenever the integrals are finite.

Fix $m\in\R^d$ and $r>0$, and consider 
\begin{equation}\label{eq:172}
  J(m,r)
  \doteq
  \inf\left\{
    \cI(\Gamma):
     m(\Gamma)=m,\;
    \mathcal R(\Gamma)=r
  \right\}.
\end{equation}
Then a formal application of the contraction principle says that
$$\PP\left((M^{N,T_N}, R^{N,T_N}) \approx (m,r)\right) = \exp\{-N T_N(J(m,r)+o(1))\}.$$
We note that since the maps
$\Gamma\mapsto m(\Gamma)$ and 
$\Gamma\mapsto\mathcal R(\Gamma)$
are not continuous on
$\cP(\cP(\R^d))$ equipped with the weak topology, a rigorous derivation
of the above asymptotic statement requires additional work and estimates. We do not pursue
this here but rather merely present some
variational calculations that illustrate some interesting qualitative features of
the quadratic granular-media rate function.
Specifically, we will show that, with
\begin{equation}\label{eq:173}
  q_*
  \doteq
  \frac{d}{\sqrt{c^2-a^2}}
  =
  \frac{d}{\sqrt{b(2a+b)}},
\end{equation}
\begin{equation}\label{eq:174}
  J(m,r)
  =
  \begin{cases}
  \displaystyle
  \frac{a^2}{4}|m|^2
  +
  \frac{(cr-d)^2}{4r},
  & 0<r\leq q_*,
  \\[3mm]
  \displaystyle
  \frac{a^2}{4}|m|^2
  +
  \frac14
  \left[
    a^2r-2cd
    +2d\sqrt{c^2-a^2}
  \right],
  & r\geq q_*.
  \end{cases}
\end{equation}
Furthermore, the optimizer $\Gamma_{m,r}$ for the variational problem in 
\eqref{eq:172}, if $r \le q_*$ is given as
$
  \Gamma_{m,r}
  =
  \delta_{\mathcal N(m,(r/d)I_d)}$ whereas when $r>q_*$, the form of the optimizer changes and every measure of the form
\begin{equation}\label{eq:175}
  \Gamma
  =
  \int_{\R^d}
  \delta_{\mathcal N(z,(q_*/d)I_d)}
  \,\nu(dz),
\end{equation}
where
\[
  \int_{\R^d}z\,\nu(dz)=m,
  \qquad
  \int_{\R^d}|z-m|^2\,\nu(dz)=r-q_*,
\]
is a minimizer. In particular, when $r> q_*$, the minimizing occupation measure is no longer
a Dirac measure. For example, if $e\in\R^d$ is any unit vector, one may take
\[
  \nu
  =
  \frac12\delta_{m+\sqrt{r-q_*}\,e}
  +
  \frac12\delta_{m-\sqrt{r-q_*}\,e}.
\]
This corresponds to spending asymptotically equal fractions of time near
the two Gaussian profiles
\[
  \mathcal N\left(
    m+\sqrt{r-q_*}\,e,
    \frac{q_*}{d}I_d
  \right)
  \quad\text{and}\quad
  \mathcal N\left(
    m-\sqrt{r-q_*}\,e,
    \frac{q_*}{d}I_d
  \right).
\]
We also note that as $N\to \infty$, $(M^{N,T_N}, R^{N,T_N}) \to (0, d/c)$ in probability and $J(0, d/c)=0$. In particular, since $q_*> d/c$, the change in the form of the minimizer occurs beyond the equilibrium variance value $d/c$.

We now prove \eqref{eq:174}.
For $\mu\in\cP_2(\R^d)$, let
$q_\mu
  \doteq
  \int_{\R^d}|x-m_\mu|^2\,\mu(dx)$ and $h(q)
  \doteq
  \frac{(cq-d)^2}{4q}$.
Suppose now that $m(\Gamma)=m$. Set
\[
  s
  \doteq
  \int_{\cP(\R^d)}
  |m_\mu-m|^2\,\Gamma(d\mu), \mbox{ and }
  \overline q
  \doteq
  \int_{\cP(\R^d)}
  q_\mu\,\Gamma(d\mu).
\]
Then,
$
  \mathcal R(\Gamma)
  =
  s+\overline q$.
Moreover,
$\int_{\cP(\R^d)}|m_\mu|^2\,\Gamma(d\mu)
  =
  |m|^2+s$.
Using Corollary~\ref{cor:5} and the strict
convexity of $h$,
\begin{align}
  \cI(\Gamma)
  &\geq
  \frac{a^2}{4}\bigl(|m|^2+s\bigr)
  +
  \int_{\cP(\R^d)}h(q_\mu)\,\Gamma(d\mu)
  \geq
  \frac{a^2}{4}\bigl(|m|^2+s\bigr)
  +
  h(\overline q).
\label{eq:178}
\end{align}
Under the constraint $\mathcal R(\Gamma)=r$, we have
$s=r-\overline q$,
$0<\overline q\leq r$.
Consequently,
\begin{equation}\label{eq:179}
  J(m,r)
  \geq
  \frac{a^2}{4}|m|^2
  +
  \inf_{0<q\leq r}
  \left\{
    \frac{a^2}{4}(r-q)
    +
    \frac{(cq-d)^2}{4q}
  \right\}.
\end{equation}
To see that the lower bound is achieved and the inequality is in fact an equality, fix $q\in(0,r]$ and let
$\nu\in\cP(\R^d)$ satisfy
$\int_{\R^d}z\,\nu(dz)=m$ and
$\int_{\R^d}|z-m|^2\,\nu(dz)=r-q$.
Define
$\Gamma_{\nu,q}
  \doteq
  \int_{\R^d}
  \delta_{\mathcal N(z,(q/d)I_d)}
  \,\nu(dz)$.
Then
$m(\Gamma_{\nu,q})=m$,
 and
  $\mathcal R(\Gamma_{\nu,q})=r$,
and Corollary~\ref{cor:5} shows that
\[
  \cI(\Gamma_{\nu,q})
  =
  \frac{a^2}{4}
  \bigl(|m|^2+r-q\bigr)
  +
  \frac{(cq-d)^2}{4q}.
\]
Hence, combining with \eqref{eq:179}, we have
\begin{equation}\label{eq:180}
  J(m,r)
  =
  \frac{a^2}{4}|m|^2
  +
  \inf_{0<q\leq r}
  \left\{
    \frac{a^2}{4}(r-q)
    +
    \frac{(cq-d)^2}{4q}
  \right\}.
\end{equation}
A straightforward calculation shows that
the optimal $q_{\mathrm{opt}}$ in the above minimization is 
$q_{\mathrm{opt}} = \min\{r, q_*\}$. The statement in \eqref{eq:174} follows.

\appendix

\section{Proof of Lemma
\ref{lem:8}.}
For $\gamma\in\cP_2(\R^d)$, the linear-growth assumptions imply
$b_\gamma\in L^2(\gamma;\R^d)$. Set
\[
  \Lambda_\gamma(\varphi)
  =
  -\int\nabla\cdot\varphi\,d\gamma
  +\int\varphi\cdot b_\gamma\,d\gamma,
  \qquad
  \varphi\in C_c^\infty(\R^d;\R^d).
\]
Recall that
$\sup_{\varphi}
  \{2\Lambda_\gamma(\varphi)-\|\varphi\|_{L^2(\gamma)}^2\}
  <\infty$
if and only if $\Lambda_\gamma$ extends to a bounded linear functional
on $L^2(\gamma;\R^d)$, in which case the supremum equals the squared
$L^2(\gamma)$-norm of its Riesz representer.

Since $b_\gamma\in L^2(\gamma)$, this is equivalent to the statement that
$\varphi\mapsto-\int\nabla\cdot\varphi\,d\gamma$ is a bounded linear functional on
$L^2(\gamma)$.
By the standard distributional characterization of finite Fisher
information, see \cite[Lemma~3.5]{HaurayMischler2014},
this holds if and only if $\gamma(dx)=\rho(x)\,dx$ and $\rho$ admits a
weak logarithmic gradient
$\nabla\log\rho\in L^2(\gamma;\R^d)$; in that case the Riesz
representer of the latter functional is $\nabla\log\rho$.
Consequently the Riesz representer of $\Lambda_\gamma$ is
$\nabla\log\rho+b_\gamma$, and therefore
$\cD(\gamma)
  =
  \|\nabla\log\rho+b_\gamma\|_{L^2(\gamma)}^2$.
This also proves $\cD(\gamma)=4\mathfrak i(\gamma)$. Moreover, finite
Fisher information and $\gamma\in\cP_2(\R^d)$ imply
$\gamma\in\operatorname{Dom}(\cF)$. Adapting the first-variation
argument in the proof of \cite[Theorem~10.4.13]{AGS}, with its separate
convexity hypotheses replaced by
Lemma~\ref{lem:24}, identifies
$\nabla\log\rho+b_\gamma$ as the minimal Wasserstein subgradient of
$\cF$ at $\gamma$. This proves the remaining equality in
\eqref{eq:84}.
\hfill \qed

\section*{Acknowledgments}
This research was partially supported by NSF DMS-2506010 and NSF RTG grant DMS-2134107. 

\section*{Declaration on the use of generative AI}

During the preparation of this manuscript, the authors used generative
AI tools to assist with editing, exposition, references, and mathematical
checks.
The authors take full responsibility for the correctness of all
mathematical arguments and conclusions.

\bibliographystyle{abbrv}

\bibliography{main}

\end{document}